\documentclass[a4paper,11pt, oneside]{amsart}
\usepackage[english]{babel}
\usepackage{amsthm}
\usepackage{amssymb}
\usepackage{amsfonts}
\usepackage{amsmath}
\usepackage{amsthm}
\usepackage[all]{xy}
\usepackage{geometry}
\usepackage{ae,aecompl}
\usepackage{eurosym}
\usepackage{verbatim}
\usepackage{mathrsfs}
\usepackage{caption}
\usepackage{url}
\usepackage{tikz}
\usepackage{pifont}
\usepackage{hyperref}
\usepackage{calligra}
 \usepackage{mathtools}
\usepackage{lineno}

\theoremstyle{plain}

\newtheorem{thm}{Theorem}[section]

\newtheorem{lem}[thm]{Lemma}
\newtheorem{prop}[thm]{Proposition}
\newtheorem{cor}[thm]{Corollary}
\newtheorem{conj}[thm]{Conjecture}

\theoremstyle{definition}

\newtheorem{eg}[thm]{Example}
\newtheorem{defn}[thm]{Definition}

\newtheorem{remark}[thm]{Remark}

\newtheorem{faidhb}[thm]{Problem}

\date{}

\newcommand\bit{\begin{itemize}}
\newcommand\eit{\end{itemize}}
\newcommand\bet{\begin{enumerate}}
\newcommand\eet{\end{enumerate}}
\newcommand\ed{\end{document}}

\DeclareFontFamily{U}{mathx}{\hyphenchar\font45}
\DeclareFontShape{U}{mathx}{m}{n}{
      <5> <6> <7> <8> <9> <10>
      <10.95> <12> <14.4> <17.28> <20.74> <24.88>
      mathx10
      }{}
\DeclareSymbolFont{mathx}{U}{mathx}{m}{n}
\DeclareFontSubstitution{U}{mathx}{m}{n}
\DeclareMathAccent{\widecheck}{0}{mathx}{"71}
\DeclareMathAccent{\wideparen}{0}{mathx}{"75}

\renewcommand{\d}{\delta}
\newcommand{\e}{\varepsilon}
\newcommand{\f}{\varphi}

\newcommand\w{\omega}

\newcommand\Om{\Omega}

\newcommand\adel{\ol{\partial}}
\newcommand\DEL{\Delta}
\newcommand\G{\Gamma}

\newcommand\bC{{\mathbb C}}

\newcommand\bR{{\mathbb R}}

\newcommand\F{{\mathcal F}}

\newcommand\MM{{\mathcal M}}
\newcommand\NN{{\mathcal N}}
\renewcommand{\O}{\mathcal{O}}

\newcommand\co{\mathrm{co}}

\newcommand\exd{\mathrm{d}}

\newcommand\unit{\mathrm{U}}
\newcommand\counit{\mathrm{C}}

\newcommand\id{\mathrm{id}}
\newcommand\proj{\mathrm{proj}}

\newcommand\inv{^{-1}}

\newcommand\oby{\otimes}

\newcommand\sseq{\subseteq}

\def\qbinom#1#2{\ensuremath{\left[\kern-.3em\left[\genfrac{}{}{0pt}{}{#1}{#2}\right]\kern-.3em\right]_q}}

\newcommand\ol{\overline}

\newcommand\EE{{\mathcal E}}
\newcommand\FF{{\mathcal F}}

\newcommand{\OO}{\mathcal{O}}
\newcommand{\xmark}{\ding{55}}%

\usepackage{tikz}
\usetikzlibrary{decorations.pathreplacing}

\usepackage{ytableau}

\author[R. \'O Buachalla]{R\'eamonn \'O Buachalla}
\address{Mathematical Institute of Charles University, Sokolovsk\'a 83, Prague, Czech Republic} 
\email{reamonnobuachalla@gmail.com}

\author[P. Somberg]{Petr Somberg}
\address{Mathematical Institute of Charles University, Sokolovsk\'a 83, Prague, Czech Republic} \email{somberg@karlin.mff.cuni.cz}

\title[A Dolbeault Complex  for $A$-Series Full Quantum Flag Manifolds]{Lusztig's Positive Root Vectors and a Dolbeault Complex for the $A$-Series Full Quantum Flag Manifolds} 

\dedicatory{I nd\'il chuimhne ar She\'an \'O Buachalla}

\thanks{R\'{O}B is supported by the GA\v{C}R/NCN grant \emph{Quantum Geometric Representation Theory and Non- commutative Fibrations} 24-11728K, and acknowledges support from COST Action 21109 CaLISTA, supported by COST (European Cooperation in Science and Technology) and HORIZON- MSCA-2022-SE-01-01 CaLIGOLA. AC acknowledges support from MSCA-DN CaLiForNIA - 101119552. PS is supported by the GA\v{C}R grant GA22-00091S.}

\begin{document}

\begin{abstract}
For the Drinfeld--Jimbo quantum enveloping algebra $U_q(\mathfrak{sl}_{n+1})$, we show that the span of Lusztig's positive root vectors,  with respect to Littlemann's nice reduced decompositions of the longest element of the Weyl group $S_{n+1}$, form quantum tangent spaces  for the full quantum flag manifold $\mathcal{O}_q(\mathrm{F}_{n+1})$. The associated differential calculi are direct $q$-deformations of the anti-holomorphic Dolbeault complex of the classical full flag manifold $\mathrm{F}_{n+1}$. As an application we establish a quantum Borel--Weil theorem for $\mathcal{O}_q(\mathrm{F}_{n+1})$, giving a noncommutative differential geometric realisation of all the finite-dimensional type-$1$ irreducible representations of $U_q(\mathfrak{sl}_{n+1})$. Restricting this differential calculus to the quantum Grassmannians is shown to reproduce the celebrated Heckenberger--Kolb anti-holomorphic Dolbeault complex. Lusztig's positive root vectors for non-nice decompositions of the longest element of $S_{n+1}$ are examined for low orders, and are exhibited  to either not give tangents spaces, or to produce differential calculi of non-classical dimension.
\end{abstract}

\maketitle

\tableofcontents

\section{Introduction}

The Heckenberger--Kolb differential calculi for the irreducible quantum flag manifolds are among the most important and well-studied structures in the noncommutative geometry of quantum groups. They are an  essentially unique covariant $q$-deformation of the de Rham complex of the classical flag manifolds of Hermitian symmetric type, a striking phenomemon given the many difficulties in formulating a comprehensive theory of noncommutative geometry for Drinfeld--Jimbo quantum groups. Moreover, as has become increasingly clear in recent years, these calculi come with an extremely rich $q$-deformed  K\"ahler geometry. They admit direct noncommutative generalisations of Lefschetz decomposition and the K\"ahler identities \cite{MMF3,MarcoConj}, of the Borel--Weil theorem \cite{BwGrassDM,CDOBBW}, and Kodaira vanishing \cite{OSV}. This has allowed for the construction of Dolbeault--Dirac Fredholm operators from noncommutative complex geometric arguments \cite{DOSFred}, and for the construction of spectral triples, in the sense of Connes, for special cases such as quantum projective space \cite{DOS1} and the first odd quantum quadric \cite{SpectTripBGG}.

It is natural to ask is what happens outside the irreducible setting, which is to say, when the classical flag manifold is no longer of Hermitian symmetric type. In this more general setting, quantum flag manifolds exhibit a richer and more challenging structure. Even in the classical setting, the  holomorphic and anti-holomorphic tangent spaces of a non-irreducible flag manifold are no longer irreducible as modules over the corresponding Levi subalgebra. For example, the space of one forms of  the full flag manifold decomposes into a direct sum of line bundles. Moreover, in the quantum setting, the  dual coalgebras of the non-irreducible quantum flag manifolds contain far fewer twisted primitive elements.  Even though the Heckenberger--Kolb calculi appeared over twenty years ago, up until now, the question of noncommutative differential structures for non-irreducible examples has received very little attention. There are, to our knowledge, just two studies: First is the quantum BGG sequence of Heckenberger and Kolb \cite{HKBGG}. Second is the work of Yuncken and Voigt  \cite{VoigtYuncken} on the  noncommutative geometry of the full quantum flag manifold of $\OO_q(\mathrm{SU}_3)$, which resulted in a proof of the quantum Baum--Connes conjecture for the discrete quantum dual of $\OO_q(\mathrm{SU}_3)$.

Here we adopt an approach to this problem based on Lusztig's celebrated root vectors. Classically the tangent space of the full flag manifold is isomorphic to the nilpotent Lie algebra $\mathfrak{n} \subseteq \mathfrak{g}$.  Formulating a quantum analogue of the inclusion $\mathfrak{g} \hookrightarrow U(\mathfrak{g})$ is a subtle question, see for example the approach of Majid \cite{BraidLieMaj}. In \cite{LusztigLeabh}, Lusztig showed  that the embedding of the nilpotent subalgebra $\mathfrak{n} \hookrightarrow  U(\mathfrak{g})$ has a direct quantum analogue for each choice of reduced decomposition of the longest element of the Weyl group of $\mathfrak{g}$. These quantum root vectors were originally introduced to prove a quantum generalisation of the classical Poincar\'e--Birkhoff--Witt theorem, and would later find many striking applications, most notably the theory of canonical bases \cite{LusztigLeabh}.  Our approach to extending the anti-holomorphic Heckenberger--Kolb construction starts by taking the span of Lusztig's positive root vectors as a tangent space for the full quantum flag, and then constructing a differential calculus from it. This allows us to make novel connections between the structural theory  of Drinfeld--Jimbo quantum groups and their noncommutative differential geometry.

In this paper we focus on the $A$-series Drinfeld--Jimbo quantum groups, and the reduced decomposition $\mathbf{j}$ given by
\begin{align} \label{eqn:nicereduceddecomp}
w_0 = (s_ns_{n-1} \cdots s_1)(s_ns_{n-1} \cdots s_{2}) \cdots (s_ns_{n-1})s_n.
\end{align}
This decomposition is one of the two nice decompositions in the sense of Littlemann \cite[\textsection 5]{LittleCrystal}. We show that the associated space of positive Lusztig root vectors forms a quantum tangent space for the full quantum flag manifold. The associated covariant differential calculus is a direct $q$-deformation of the anti-holomorphic subcomplex of the classical $A$-series full quantum flag manifold, in particular it has classical dimension. Its quantum exterior algebra admits a concise generator and relation presentation in terms of the $A$-series root system. The quantum exterior algebra is moreover a Frobenius algebra. Notably, its Nakayama automorphism is \emph{not} $q$-deformed. Restricting this differential calculus to the partial quantum flag manifolds then gives a differential calculus for each of them. In particular, for the quantum Grassmannians, we recover the anti-holomorphic subcomplex of the Heckenberger--Kolb differential calculus.

In order to explore and test our new construction, we prove a direct $q$-deformation of the Borel--Weil theorem for the full flag manifold. This gives noncommutative geometric realisations of all the type-$1$ finite-dimensional representations of $U_q(\mathfrak{sl}_{n+1})$, extending previous work for the irreducible quantum flag case \cite{BwGrassDM,CDOBBW}. The proof requires the introduction of a quantum principal bundle presentation, in the sense of Majid and Brzezinski, of our differential calculus. Moreover, it presents the relative line modules over $\OO_q(\mathrm{F}_{n+1})$ as noncommutative holomorphic line modules in the sense of Beggs and Majid \cite[\textsection 7]{BeggsMajid:Leabh}.

For low-dimensional cases, we also examine the behaviour of the positive Lusztig roots vectors associated to the other reduced decompositions. We see that for $A_3$, all $8$ commutation classes of reduced decompositions give tangent spaces. However, apart from the decompositon dual to \eqref{eqn:nicereduceddecomp} (that is the decomposition given by the action of the opposite automorphism of the $A$-series Dynkin diagram on $\mathbf{j}$) the associated differential calculi have non-classical dimension. Moreover, for the case of $A_4$, there are reduced decompositions that are not even tangent spaces. This motivates the conjecture that for the $A$-series, Littlemann's nice reduced decompositions are the only decompositions with a well-behaved noncommutative geometry.

This work leads naturally to a number of future topics. First is the question of a \mbox{$q$-deformation} of the entire Dolbeault double complex, along with its classical complex geometry. This will be addressed in subsequent work using Lusztig's negative root vectors  \cite{ACROBJR}. Subsequently, the $B,C$, and $D$ series will be treated following the same program. Ultimately, it is hoped that this will show the way to an understanding of the noncommutative geometry of Drinfeld--Jimbo coordinate algebras themselves.

\subsection{Summary of the Paper}

The paper is organised as follows: In \textsection 2 we recall some basic material about covariant differential calculi over Hopf algebras and their associated tangent spaces, as well as the maximal prolongation of covariant first-order differential calculi. 

In \textsection 3 we treat the space of positive root vectors associated to the reduced decomposition of the longest element of the Weyl group given in \eqref{eqn:nicereduceddecomp}, and show that their span forms a quantum tangent space for $\OO_q(\mathrm{SU}_{n+1})$. We then determine the module structure of the dual cotangent space, and give a full set of relations for the maximal prolongation of the associated differential calculus $\Omega^{(0,\bullet)}_q(\mathrm{SU}_{n+1})$. We also present a number of instructive low-dimensional examples. We then examine the algebraic properties of the maximal prolongation, and show that it has classical dimension by examining the associated graded algebra of a canonical filtration on the space of left-coinvariant forms $\Lambda^{(0,\bullet)}_q$. We also show that the graded algebra $\Lambda^{(0,\bullet)}_q$ is a Frobenius algebra and show that the associated Nakayama automorphism is of classical type.

In \textsection 4 we begin by recalling some basic results on quantum homogeneous spaces and their tangent spaces. We then examine the restriction of $\Omega^{(0,\bullet)}_q(\mathrm{SU}_{n+1})$ to the quantum Grassmannians, showing  that we recover the anti-holomorphic Heckenberger--Kolb differential calculus.  

In \textsection 5 we examine the restriction of $\Omega^{(0,\bullet)}_q(\mathrm{SU}_{n+1})$ to the full quantum flag manifold $\OO_q(\mathrm{F}_{n+1})$. We show that this differential calculus has classical dimension, and present it as a direct sum of line modules over $\OO_q(\mathrm{F}_{n+1})$. Notably, we show that its right module structure is more complicated than that of the Heckenberger--Kolb calculus. In particular, we show the differential calculus is not contained the subcategory of relative Hopf modules to which the simple monoidal version of Takeuchi's categorical equivalence applies.

In \textsection 6 we prove a direct $q$-deformation of the Borel--Weil theorem for the classial full flag manifold, giving noncommutative geometric realisations of all the type-$1$ finite-dimensional irreducible representations of $U_q(\mathfrak{sl}_{n+1})$.

In \textsection 7, for the special cases of $U_q(\mathfrak{sl}_4)$ and $U_q(\mathfrak{sl}_5)$, we examine the situation for reduced decompositions other than Littlemann's nice decompositions. In all cases examined, we see that the space of root vectors either do not give a tangent space, or the associated differential calculus does not have classical dimension. Motivated by this, we make some conjectures about the general $U_q(\mathfrak{sl}_{n+1})$ situation.

We finish with two appendices. In the first we present necessary details on Drinfeld--Jimbo quantum groups and Lusztig's root vectors. In the second we make some comments about filtered algebras and their associated graded algebras. 

\subsubsection*{Acknowledgements:} We would like to thank Arnab Bhattacharjee, Alessandro Caro-tenuto, Andrey Krutov, Karen Strung, and  Bart Vlaar, for many useful discussions. We would also like to thanks Nicolas M. Thiery and Nathan Williams for providing references related to commutation classes of reduced decompositions. Finally, we would like to thank the referee for their helpful remarks and their very careful reading of the article.


\section{Preliminaries}

In this section we recall some basic material about covariant differential calculi over Hopf algebras and their associated tangent spaces. We use Sweedler notation, denote by $\Delta, \e$, and $S$ the coproduct, counit, and antipode of a Hopf algebra respectively. We write $A^{\circ}$ for the dual coalgebra (Hopf algebra) of a (Hopf) algebra $A$, and denote the pairing between $A$ and $A^{\circ}$ by angular brackets. Throughout the paper, all algebras are over $\mathbb{C}$ and assumed to be unital, all unadorned tensor products are over $\mathbb{C}$, and all Hopf algebras are assumed to have bijective antipodes.

\subsection{Covariant Differential Calculi over Hopf algebras}

A {\em differential calculus}, or a \emph{dc}, is a differential graded algebra (dga) 
$$
\Big(\Om^\bullet \cong \bigoplus_{k \in \mathbb{Z}_{\geq 0}} \Om^k, \exd\Big)
$$  
which is generated as an algebra by the elements $a, \exd b$, for $a,b \in \Om^0$. When no confusion arises, we denote the dc by $\Omega^{\bullet}$, omitting the exterior derivative $\exd$. We denote the degree of a homogeneous element $\w \in \Om^{\bullet}$ by $|\w|$. For a given algebra $B$, a differential calculus {\em over} $B$ is a differential calculus such that $B = \Om^0$. A differential calculus is said to be of {\em total degree} $m \in \mathbb{Z}_{\geq 0}$ if $\Om^m \neq 0$, and $\Om^{k} = 0$, for all $k > m$.

A {\em first-order differential calculus} (fodc) over an algebra $B$ is a pair $(\Om^1(B),\exd)$, where $\Omega^1(B)$ is a $B$-bimodule and $\exd: B \to \Omega^1(B)$ is a derivation such that $\Om^1$ is generated as a left  (or equivalently right) $B$-module by those elements of the form~$\exd b$, for~$b \in B$. We say that a differential calculus $(\G^\bullet,\exd_\G)$ {\em extends} a first-order calculus $(\Om^1,\exd_{\Om})$ if there exists a bimodule isomorphism $\f:\Om^1 \to \G^1$ such that $\exd_\G = \f \circ \exd_{\Om}$. It can be shown  \cite[\textsection 2.5]{MMF2} that any first-order calculus admits an extension $\Om^\bullet$ which is \emph{maximal},  in the sense that there exists a unique differential map from $\Om^\bullet$ onto any other extension of $\Om^1$. We call this extension the {\em maximal prolongation} of the first-order calculus.

For $A$ a Hopf algebra, a dc $\Omega^\bullet$ over a left $A$-comodule algebra $(P,\Delta_L)$ is said to be \emph{left covariant} if the coaction $\Delta_L : P \to A \otimes P$ extends to a (necessarily unique) coaction $\Delta_L : \Omega^\bullet \to A \otimes \Omega^\bullet$, giving $\Omega^\bullet$ the structure of a left $A$-comodule algebra, such that~$\exd$ is a left $A$-comodule map.  A right covariant dc over a right $A$-comodule algebra is defined similarly.

\subsection{Tangent Spaces and Two-Sided Hopf Modules} \label{subsection:TangentSpaces}

A \emph{tangent space} for a Hopf algebra $A$ is a linear subspace $T$ of the dual Hopf algebra $A^{\circ}$ such that $X(1) = 0$, for all $X \in T$, and $T \oplus \mathbb{C}1$ is a right coideal of $A^{\circ}$. For any tangent space $T$, a right $A$-ideal of $A^+ := \mathrm{ker}(\e)$ is given by 
\begin{align} \label{eqn:theideal}
I := \Big\{ x \in A^+ \,|\, X(x) = 0, \textrm{ for all } X \in T\Big\},
\end{align}
meaning that the quotient $\Lambda^1(A) := A^+/I$ is naturally an object in the category $\mathrm{Mod}_A$ of right $A$-modules.  We call $\Lambda^1(A)$ the \emph{cotangent space} of the fodc.

Consider next the category ${}^A_A\mathrm{Mod}_A$: This has as objects left $A$-comodules $\F$, endowed with an $A$-bimodule structure, such that the comodule structure map is an $A$-bimodule map with respect to the diagonal $A$-bimodule structure on $A \otimes \F$. Its morphisms are left $A$-comodule, $A$-bimodule, maps. We have a functor 
\begin{align*}
A \otimes - : \mathrm{Mod}_A \to {}^A_A\mathrm{Mod}_A,  & & V \mapsto A \otimes V,
\end{align*}
where the left $A$-module, left $A$-comodule, structure of $A \otimes V$ is given by the first tensor factor, and the right $A$-module structure is given by the diagonal action. In the other direction, we have a functor 
\begin{align*}
F: {}^A_A\mathrm{Mod}_A \to \mathrm{Mod}_A, & & \F \mapsto \F/A^+\F.
\end{align*}
This determines an adjoint equivalence between the two categories, with unit and counit
\begin{align*}
\unit: \MM \to A \otimes F(\MM), & & g \mapsto g_{(1)} \otimes [g_{(0)}],\\
\counit: F(A \otimes V) \to V, & & a \otimes v \mapsto \e(a)v.
\end{align*}
This equivalence is known as the \emph{fundamental theorem of two-sided Hopf modules}, see \cite[Theorem 5.7]{Sch.YD}, or \cite[\textsection 4.1]{Sweedler}, for more details. The category ${}^A_A\mathrm{Mod}_A$ has a monoidal structure given by the usual tensor products of $A$-bimodules and left $A$-comodules. Moreover, $\mathrm{Mod}_A$ has a monoidal structure given by the tensor product of vector spaces endowed with the diagonal $A$-module structure. With respect to these two monoidal structures, the natural transformation with components 
\begin{align} \label{eqn:generalINVERSEMonMult}
\mu:F(\MM) \otimes F(\NN) \to F(\MM \otimes_A \NN), & & [m] \otimes [n] \mapsto [mS(n_{(-1)}) \otimes n_{(0)}].
\end{align}
gives $F$ the structure of a monoidal functor. Note that the inverse of $\mu$ is given by 
\begin{align*} 
\mu^{-1}:F(\MM \otimes \NN) \to F(\MM) \otimes_A F(\NN), & & [m \otimes n] \mapsto [mn_{(-1)}] \otimes [n_{(0)}]
\end{align*}
Consider now the object
\begin{align*}
{}^A_A\mathrm{Mod}_A \ni \Omega^1(A) := A \otimes \Lambda^1(A).
\end{align*}
If $\{X_i\}_{i=1}^n$ is a basis for $T$, and $\{e_i\}_{i=1}^n$ is the dual basis of $\Lambda^1(A)$, then the map
\begin{align} \label{eqn:tangent.exd}
\exd: A \to \Omega^1(A), & &  a \mapsto [a_{(1)}] \otimes [a_{(2)}^+] = \sum_{i=1}^n (X_i \triangleright a) \otimes e_i
\end{align}
is a derivation, and the pair $(\Omega^1(A),\exd)$ is a left $A$-covariant fodc over $A$.  Moreover, it follows from the theorem of two-sided Hopf modules that all  fodc, which are finitely generated as left $A$-modules, arise from a finite-dimensional tangent space. Indeed, this gives a bijective correspondence between isomorphism classes of finite-dimensional tangent spaces and finitely-generated  left-covariant fodc  \cite{WoroDC}.

Let $W$ be a Hopf subalgebra of $A^{\circ}$, and denote by $\pi:(A^{\circ})^{\circ} \to W^{\circ}$ the Hopf algebra map given by restriction of domains. Considering $A$ as contained in $(A^{\circ})^{\circ}$, we denote $K:= \pi(A)$, and call $K$ the \emph{quantum subgroup} associated to $W$. Now if a tangent space $T \subseteq A^{\circ}$ is invariant under the adjoint action $\mathrm{ad}$ of $W$, that is, if $\mathrm{W}T \subseteq T$, then the corresponding ideal $I \subseteq A^+$ is a subcomodule with respect to the \emph{relative adjoint coaction}  
\begin{align*}
\mathrm{Ad_{R,\pi}}: A^+ \to A^+ \otimes K, & & a \mapsto a_{(2)} \otimes \pi(S(a_{(1)})a_{(3)}).
\end{align*}
This endows $\Lambda^1(A)$ with a right $K$-coaction, which in turn gives the associated fodc $\Omega^1$ the structure of a right $K$-covariant dc, defined by the tensor product $K$-coaction. Moreover, via the dual pairing $W \times K \to \mathbb{C}$, both $\Lambda^1(A)$ and $\Omega^{1}$ have left $W$-actions dual to their right $K$-coactions.

Just as for any covariant fodc, this right $K$-coaction extends to a right $K$-coaction on $\Omega^{\bullet}$, the maximal prolongation of $\Omega^1$. This gives $\Omega^{\bullet}$ the structure of a right $K$-comodule algebra. Together the left $A$-coaction and the right $K$-coaction give $\Omega^{\bullet}$ the structure of an $(A,K)$-bicomodule. This means that the space of left $A$-coinvariant forms $1 \otimes \Lambda^{\bullet}(A)$ is a right $K$-subcomodule of $\Omega^{\bullet}$. In particular, it is a right $K$-comodule algebra.

\subsection{Generating Relations of the Maximal Prolongation} \label{subsection:MC}

An explicit presentation of the maximal prolongation of a left $A$-covariant fodc $\Omega^1(A)$ over a Hopf algebra $A$ is given as follows: For $I \subseteq A^+$ the classifying ideal of the fodc, and $\Lambda^1(A)$ the cotangent space, consider the subspace
\begin{align*}
I^{(2)} := \left\{\omega(x) := [x_{(1)}^+] \otimes [x_{(2)}^+]  ~ | ~ x \in I \right\} \subseteq \Lambda^1(A) \otimes  \Lambda^1(A).
\end{align*}
From the tensor algebra  $\mathcal{T}(\Lambda^1(A))$ of $\Lambda^1(A)$, we construct the $\mathbb{Z}_{\geq 0}$-graded algebra
\begin{align*}
\Lambda^{\bullet}(A) :=  \bigoplus_{k \in \mathbb{Z}_{\geq 0}} \Lambda^k(A) := \mathcal{T}\big(\Lambda^1(A)\big)/\langle I^{(2)} \rangle,
\end{align*}
which we call the \emph{quantum exterior algebra} of $\Omega^1(A)$, and whose multiplication we denote by $\wedge$. It follows from the fundamental theorem of Hopf modules, and \eqref{eqn:tangent.exd}, that an isomorphism between $F(\Omega^k(A))$ and $\Lambda^{k}(A)$, for $k \in \mathbb{Z}_{\geq 0}$, is determined by  
\begin{align*}
[\exd a_1 \wedge \cdots \wedge \exd a_k] ~ \mapsto  ~ \bigwedge_{i=1}^k \, [(a_i)^+_{(i)}(a_{(i+1)})_{(i)} \cdots (a_{(k)})_{(i)}],
\end{align*}
where we have denoted $(a_1)_{(1)} := a_1$. See \cite{KSLeabh} for further details, or \cite{MMF2} for the more general quantum homogeneous space case.


\section{A First-Order Differential Calculus for $\OO_q(\mathrm{SU}_{n+1})$} \label{section:LusztigRoots}

In this section we show that, for a distinguished decomposition of the longest element of the Weyl group,  Lusztig's root vectors give a tangent space for $\OO_q(\mathrm{SU}_{n+1})$. The maximal prolongation of the associated fodc is then described, and shown to have classical dimension. We also show that quantum exterior algebra $\Lambda^{(0,\bullet)}_q$ has the structure of a Frobenius algebra and prove that its Nakayama automorphism is of classical type. Throughout, low rank examples are treated in detail, and we use the notation set in Appendix \ref{app:Lusztigroots}. Moreover, we find it convenient to identify  $\Delta^+$ with the set of natural number pairs $(i,j)$ for which $\alpha_{ij}$ is a positive root.

\subsection{The Lusztig Root Vectors for $U_q(\mathfrak{sl}_{n+1})$}

Consider the decomposition $\mathbf{j}$ of the longest element of the Weyl group $w_0 \in S_{n+1}$:
\begin{align*}
w_0 =  (s_ns_{n-1} \cdots s_1)(s_ns_{n-1} \cdots s_{2}) \cdots (s_ns_{n-1})s_n,
\end{align*}
where $s_1, \dots, s_{n}$ denote the standard generators of $S_{n+1}$. As calculated in Appendix \ref{app:Lusztigroots},  the root vectors corresponding to this decomposition are given explicitly by 
\begin{align*}
E_{ji} := [E_{j-1},[E_{j-2},[ \cdots [E_{i+1},E_i]_{q^{-1}} \cdots ]_{q^{-1}}, & & \textrm{ for } (i,j) \in \Delta^+.
\end{align*} 
We now produce a general coproduct identity that will be used to verify the tangent space condition for our choice of decomposition of $w_0$.

\begin{lem} \label{lem:coproductformula}
Let $X$ be an element of $U_q(\mathfrak{sl}_{n+1})$ whose coproduct admits a presentation $\Delta(X) = X_{(1)} \otimes X_{(2)}$ such that the second  tensor factor of each summand  is a weight vector of weight $\rho$, with respect to the right adjoint action of $U_q(\mathfrak{sl}_{n+1})$ on itself. Then it holds that 
\begin{align} \label{eqn:coproductformula}
\Delta([X,E_i]_{q^{-1}}) = [X_{(1)},E_i]_{q^{\langle\alpha_i,\rho\rangle - 1}} \otimes X_{(2)}K_i - X_{(1)} \otimes [X_{(2)},E_{i}]_{q^{-1}}.
\end{align}
\end{lem}
\begin{proof}
Expanding the expression 
$$
\Delta\left([X,E_i]_{q^{-1}}\right) = \Delta(X)\Delta(E_i) - q^{-1}\Delta(E_i)\Delta(X)
$$
we arrive at the sum
$$
 X_{(1)}E_i \otimes X_{(2)}K_i + X_{(1)} \otimes X_{(2)}E_i  - q^{-1} E_i X_{(1)} \otimes K_iX_{(2)} - q^{-1} X_{(1)} \otimes E_iX_{(2)}.
$$
Since the second tensor factor of each summand  is a weight vector, this is equal to 
\begin{align*}
\left(X_{(1)}E_i  - q^{\langle \alpha_i,\rho\rangle - 1} E_iX_{(1)}\right)\! \otimes X_{(2)}K_i  + X_{(1)} \otimes \! \Big(X_{(2)}E_i - q^{-1}  X_{(2)}E_i\Big),
\end{align*}
which gives the claimed sum.
\end{proof}

\begin{eg} \label{eg:singlecomm}
Consider the $q$-deformed commutator $[E_i,E_{i - 1}]_{q^{-1}}$. It follows from the definition of $U_q(\mathfrak{sl}_{n+1})$ that the second tensor factor of each summand of $\Delta(E_i)$ is a weight vector. From the formula above, we now see that $\Delta([E_i,E_{i - 1}]_{q^{-1}})$ is equal to 
\begin{align*}
 [E_i,E_{i-1}]_{q^{-1}} \otimes K_iK_{i-1} + [1,E_{i-1}]_{q^{-2}} \otimes E_iK_{i-1}  + E_i \otimes [K_i,E_{i-1}]_{q^{-1}} + 1 \otimes [E_i,E_{i-1}]_{q^{-1}},
\end{align*}
where we have used the fact that $\mathrm{wt}_{i-1}(E_i) = -1$. This expression then simplifies to 
\begin{align*}
 [E_i,E_{i-1}]_{q^{-1}} \otimes K_iK_{i-1} + q^{-1}\nu E_{i-1} \otimes E_iK_{i-1} +  1 \otimes [E_i,E_{i-1}]_{q^{-1}},
\end{align*}
where we recall the notation $\nu := q - q^{-1}$.
\end{eg}

We will now use an inductive argument to generalise the formulae given in this example to the case of  $U_q(\mathfrak{sl}_{n+1})$.

\begin{prop} \label{prop:coidealformula}
For any $(i,j) \in \Delta^+\!,$ it holds that 
\begin{align*}
\Delta(E_{ji}) = E_{ji} \otimes K_{j-1,i} ~ + ~ q^{-1}\nu \sum_{a = i+1}^{j-1} E_{ai} \otimes E_{ja}K_{a-1,i} ~ + ~ 1 \otimes E_{ji},
\end{align*}
where we have denoted $K_{j-1,i} := K_{j-1} \cdots K_i$, for $(i,j) \in \Delta^+\!,$ and $K_{jj} := K_j$.
\end{prop}
\begin{proof}
We begin by noting that the claimed identity clearly holds for each simple generator $E_{i+1,i} = E_i$. Let us now assume that the claimed formula holds for some $E_{ji}$, such that $i \geq 2$. Since the requirements of Lemma \ref{lem:coproductformula} are satisfied, we can use  formula \eqref{eqn:coproductformula} to calculate the coproduct of  $E_{j,i-1}$. The first term of \eqref{eqn:coproductformula} is given by the sum
\begin{align*}
 [E_{ji},E_{i-1}]_{q^{-1}} \otimes K_{j-1,i-1} ~ + ~ q^{-1}\nu \sum_{a = i+1}^{j-1} [E_{ai},E_{i-1}]_{q^{-1}} \otimes E_{ja}K_{a,i-1} \\
  + ~ [1,E_{i-1}]_{q^{-2}} \otimes E_{ji}K_{i-1}, ~~~~~~~~~~~~~~~~~~~~~~~~~~~~~~~~~~~~~~~~~~~~~~
\end{align*}
which reduces to the sum
\begin{align*}
 \, E_{j,i-1} \otimes K_{j-1,i-1} ~ + ~ q^{-1}\nu \sum_{a = i}^{j-1} E_{a,i-1} \otimes E_{ja}K_{a,i-1}.
\end{align*}
The second term of \eqref{eqn:coproductformula} is given by the sum
\begin{align*}
 \, E_{ji} \otimes [K_{j-1,i},E_{i-1}]_{q^{-1}} ~ + ~ q^{-1}\nu \sum_{a = i+1}^{j-1} E_{ja} \otimes [E_{ai}K_{j-1,a},E_{i-1}]_{q^{-1}} \\
+ ~ 1 \otimes [E_{ji},E_{i-1}]_{q^{-1}}, ~~~~~~~~~~~~~~~~~~~~~~~~~~~~~~~~~~~~~~~~~~~~~~
\end{align*}
which reduces to the single term $ 1 \otimes E_{j,i-1}$. Adding these two terms together gives the claimed identity for $E_{j,i-1}$, and hence for the general case by an inductive argument. 
\end{proof}

\begin{cor} \label{cor:tangentspace}
A tangent space is given by 
\begin{align*}
T^{(0,1)} := \mathrm{span}_{\mathbb{C}}\{ E_{\beta} \,|\, \beta \in \Delta^+ \}.
\end{align*}
\end{cor}
\begin{proof}
This follows directly from the Proposition \ref{prop:coidealformula} and the definition of a tangent space given in \textsection \ref{subsection:TangentSpaces}.
\end{proof}

Each root vector is a weight vector (see Appendix \ref{app:Lusztigroots}), hence $\mathrm{ad}(U_q(\mathfrak{h}))T^{(0,1)} \subseteq T^{(0,1)}$. The quantum  subgroup associated  to $U_q(\mathfrak{h})$ is the group Hopf  algebra of $\mathbb{Z}^{n}$, which we denote by $\OO(T^{n})$. This implies the following corollary.

\begin{cor} \label{cor:FOBasisElements}
Denote by $(\Omega^1_q(\mathrm{SU}_{n+1}),\adel)$ the fodc associated to the tangent space $T^{(0,1)}$. It holds that $\Omega^1_q(\mathrm{SU}_{n+1})$ is right $\OO(T^{n})$-covariant. Moreover, with respect to the  associated left $U_q(\mathfrak{h})$-action, the basis elements have weight
\begin{align*}
|E_{ji}| = \alpha_{j-1} + \alpha_{j-2} + \cdots + \alpha_i, & & \textrm{ for } (i,j) \in \Delta^+.
\end{align*}
\end{cor}

\subsection{The Module Structure of the Cotangent Space $\Lambda^{(0,1)}_q$}

 In this subsection we determine the right $\OO_q(\mathrm{SU}_{n+1})$-module structure of  $\Lambda^{(0,1)}_q$ by detailing how the matrix generators act on the basis dual to the root vector basis of $T^{(0,1)}$. 
We observe that the dual pairing of Hopf algebras $\langle -,- \rangle:  U_q(\mathfrak{sl}_{n+1}) \times \OO_q(\mathrm{SU}_{n+1}) \to \mathbb{C}$, as given in Appendix \ref{app:Lusztigroots},  descends to a non-degenerate pairing 
  \begin{align*}
 \langle -,- \rangle: T^{(0,1)} \times \Lambda^{(0,1)}_q \to \mathbb{C},
 \end{align*}
for which we use the same symbol by abuse of notation.

\begin{lem} \label{lem:dualpairing}
It holds that  
\begin{align*}
\langle E_{ji}, [u_{rs}] \rangle = \delta_{jr}\delta_{is},  & & \textrm{ for all } (i,j) \in \Delta^+, \textrm{\emph{ and  }}r,s = 1, \dots, n + 1.
\end{align*}
Thus a basis of $\Lambda^{(0,1)}_q$, dual to the defining basis of $T^{(0,1)}$, is given by 
\begin{align*}
\Big\{e_{\alpha_{ij}} := e_{ji} := [u_{ji}] \,|\, \textrm{ for } (i,j) \in \Delta^+ \Big\}.
\end{align*}
\end{lem}
\begin{proof}
It follows from the dual pairing formulae given in \eqref{eqn:dualpairing} that the formula holds for all the simple generators $E_i = E_{i+1,i}$. Let us now assume that the formula holds for some $E_{ji}$. If $i > 1$, then it follows from Proposition \ref{prop:coidealformula} that the pairing 
\begin{align*}
\langle E_{j,i-1}, u_{kl} \rangle = & \, \langle E_{ji}E_{i-1}, u_{kl} \rangle - q^{-1}\langle E_{i-1}E_{ji}, u_{kl}\rangle 
\end{align*}
reduces to the product of Kroenecker delta symbols 
\begin{align*}
\sum_{a=1}^{n+1} \langle E_{ji}, u_{ka} \rangle \langle E_{i-1}, u_{al} \rangle - q^{-1}\sum_{a=1}^{n+1} \langle E_{i-1}, u_{ka} \rangle \langle E_{ji}, u_{al} \rangle
=  \, \delta_{jk}\delta_{i-1,l}.
\end{align*}
Thus  by an inductive argument, the formula holds for all root vectors.
\end{proof}

\begin{prop} \label{prop:Gens}
A generating set $G$ for the ideal $I$ associated to the tangent space $T^{(0,1)}$ is given by the union of the following three sets of elements:
\begin{align*}
G_1 \, :=  & \{ u_{ij}, \, u_{kk}-1, \,|\, (i,j) \in \Delta^+,  k = 1, \dots, n+1\},\\
G_2\, := & \{ u_{ji}u_{kl} \, | \, (i,j), \, (k,l) \in \Delta^+\},\\
G_3 \, := & \{u_{ji}u_{lk} \,| \,(i,j), \, (k,l) \in \Delta^+, \, j \neq k \}.
\end{align*}
Moreover,  for $(i,j), \, (j,j') \in \Delta^+$, and $k =1, \dots, n$, the actions
\begin{align} \label{eqn:quantumaction}
e_{ji}u_{kk} = q^{\delta_{jk} - \delta_{ik}}e_{ji}, & & e_{ji}u_{j'j} =  \nu e_{j'i},  & &
\end{align}
are the only non-zero actions of the generators $u_{kl}$ on the basis elements $e_{ji}$.
\end{prop}
\begin{proof} 
It follows from Lemma \ref{lem:dualpairing} that the elements of $G_1$ are contained in the kernel of $T^{(0,1)}$, which is to say, they are contained in $I$. Consider now an element $u_{gf}u_{kl} \in G_2$, it holds that 
\begin{align*}
\langle E_{ji}, u_{gf}u_{kl} \rangle = & \, \langle (E_{ji})_{(1)},u_{gf}\rangle \langle (E_{ji})_{(2)}, u_{kl}. \rangle.
\end{align*}
If follows from Proposition \ref{prop:coidealformula} that $\Delta(E_{ji})$ can be expressed as a sum of tensor products of root vectors. Now Lemma \ref{lem:dualpairing} tells us that $u_{kl}$ pairs trivially with any root vector, and so, $u_{gf}u_{kl}$ must be contained in $I$. Consider next an element  $u_{gf}u_{lk}  \in G_3$, evaluated with a root vector $E_{ji}$:
\begin{align*}
\langle E_{ji}, u_{gf}u_{lk} \rangle = & \, \langle (E_{ji})_{(1)},u_{gf}\rangle \langle (E_{ji})_{(2)}, u_{lk} \rangle.
\end{align*}
It follows from Proposition \ref{prop:coidealformula} that this expression is equal to 
\begin{align*}
q^{-1}\nu \sum_{a = i+1}^{j-1} \langle E_{ai},u_{gf}\rangle \langle E_{ja},u_{lk} \rangle  \langle K_{a-1,i},u_{kk} \rangle 
\end{align*}
which vanishes since we are assuming that $g \neq k$. Thus it holds that $G_3 \subseteq I$.

Let us next show that $G$ generates $I$ as a right $\OO_q(\mathrm{SU}_{n+1})$-module. Note that an isomorphism of vector spaces is given by
\begin{align*}
\OO_q(\mathrm{SU}_{n+1}) \to \OO_q(\mathrm{SU}_{n+1})^+ \oplus \mathbb{C}, & & [a] \mapsto ([a^+],\e(a)),
\end{align*}
which restricts to an isomorphism 
\begin{align*}
\OO_q(\mathrm{SU}_{n+1})/I \to \Lambda^{(0,1)}_q \oplus \mathbb{C}.
\end{align*}
Thus if we can show that, for $J := G\OO_q(\mathrm{SU}_{n+1})$, 
$$
\dim\big(\OO_q(\mathrm{SU}_{n+1})/J\big) = \dim\big(\Lambda_q^{(0,1)}\big) + 1,
$$
then we can conclude that $G$ is a generating set. Note first that since $u_{ii}-1 \in G$,  the coset of any monomial in the diagonal generators is a scalar multiple of the coset of the identity. Moreover, for any monomial of the form $fa$, where $f$ is a monomial in the diagonal generators, and $a \in \OO_q(\mathrm{SU}_{n+1})$, the coset $[fa] = \e(f)[a]$. This means that, if $a = u_{ij}a'$, for some $(i,j) \in \Delta^+$, and $a' \in \OO_q(\mathrm{SU}_{n+1})$, then since $u_{ij} \in G$, we have 
$$
[fu_{ij}a'] = \e(f)[u_{ij}a'] = \e(f)[u_{ij}]a' = 0.
$$
Similarly, if $a = u_{ji}a'$, then we have 
$$
[fu_{ji}a'] = \e(f)[u_{ji}]a'.
$$
Thus it remains to show that the span of the cosets $[u_{ji}]$, for $(i,j) \in \Delta^+$, is closed under the action of the generators of $\OO_q(\mathrm{SU}_{n+1})$. This amounts to showing that every coset $[u_{ji}u_{ab}]$ is a linear combination of the cosets $[u_{ji}]$. There are four types of coset to check. First, we have the case
\begin{align*}
[u_{ji}u_{kk}] =  [u_{kk}u_{ji}] + \mathrm{sign}(j-k)\nu [u_{jk}u_{ki}] = [u_{ji}], & & \textrm{ for } i > k, \textrm{ or } k > j,
\end{align*}
Secondly, we have the case
\begin{align*}
[u_{ji}u_{kk}] = [u_{kk}u_{ji}] = [u_{ji}], ~~~~ \textrm{ for } j  > k > i.
\end{align*}
Finally, we have the case
\begin{align*}
[u_{ji}u_{kj}] = [u_{kj}u_{ji}] - \nu [u_{jj}u_{ki}] = - \nu[u_{ki}].
\end{align*}
Moreover, we see that we have also verified the claimed formulae for the actions of the generators on the basis elements.
\end{proof}

\subsection{Relations of the Maximal Prolongation} \label{subsection:MP}

In this subsection we calculate the degree two relations for the maximal prolongation of the fodc $\Omega^1_q(\mathrm{SU}_{n+1})$, using the presentation of higher relations given in \textsection \ref{subsection:MC}. We begin by deriving a general technical formula.

\begin{lem} \label{lem:MCProduct}
Let $A$ be a Hopf algebra, endowed with a left $A$-covariant fodc $\Omega^1(A)$. For any $x,y \in A^+$, it holds that
\begin{align*}
\omega(xy) = [x^+y_{(1)}] \otimes [y_{(2)}^+] + [y_{(1)}^+] \otimes [xy_{(2)}] +  [x_{(1)}^+y_{(1)}] \otimes [x_{(2)}^+y_{(2)}].
\end{align*}
Moreover, for any set of generators $G$ of the right $A$-ideal $I \subseteq A^+$, it holds that $I^{(2)}$ is generated as a right $A$-ideal by the set of elements
\begin{align} \label{eqn:genRelas}
\big\{ \omega(g) \, | \, \textrm{ for } g \in G \big\}.
\end{align}
\end{lem}
\begin{proof}
For any $a,b \in A$, we have that $(ab)^+ = a^+b + \e(a)b^+$. From this we see that
\begin{align*}
\omega(xy) =  \, & [(x_{(1)}y_{(1)})^+] \otimes [(x_{(2)}y_{(2)})^+] \\
= \, & [x_{(1)}^+y_{(1)} + \e(x_{(1)})y_{(1)}^+] \otimes [x_{(2)}^+y_{(2)} + \e(x_{(2)})y_{(2)}^+]\\
= \, & [x^+_{(1)}y_{(1)}] \otimes [x^+_{(2)}y_{(2)}] + [xy_{(1)}] \otimes [y_{(2)}^+] + [y_{(1)}^+] \otimes [xy_{(2)}],
\end{align*}
where we have used the fact that $x,y \in I \subseteq A^+$.

We now move on to verifying that the elements in \eqref{eqn:genRelas} generate $I^{(2)}$ as a right $A$-module. Any element of $I$ is a linear combination of elements of the form $ga$, for $g \in G$, $a \in A$. Thus we see that the space of degree two relations is spanned, as a vector space, by elements of the form
\begin{align*}
\omega(ga) = [(ga)_{(1)}^+] \otimes [(ga)_{(2)}^+] = &  \,  [g_{(1)}^+a_{(1)}] \otimes [g_{(2)}^+a_{(2)}] + [ga_{(1)}] \otimes [a_{(2)}^+] + [a_{(1)}^+] \otimes [ga_{(2)}]\\
= & \,  [g_{(1)}^+a_{(1)}] \otimes [g_{(2)}^+a_{(2)}] + [g]a_{(1)} \otimes [a_{(2)}^+] + [a_{(1)}^+] \otimes [g]a_{(2)}\\
=  & \,  ([g_{(1)}^+] \otimes [g_{(2)}^+])a\\
=  & \, \omega(g)a.
\end{align*}
Hence the elements in \eqref{eqn:genRelas} generate $I^{(2)}$ as claimed.
\end{proof}

Using this lemma, and the set of generators $G = G_1 \cup G_2 \cup G_3$ from Lemma \ref{prop:Gens}, we will now produce an explicit description of $\Lambda^{(0,\bullet)}_q$ the quantum exterior algebra of $\Omega^{(0,1)}_q(\mathrm{SU}_{n+1})$. 

\begin{lem} \label{lem:THERELATIONS}
The algebra $\Lambda^{(0,\bullet)}_q$ is isomorphic to the quotient of the tensor algebra of $\Lambda^{(0,1)}_q$ by the right $\OO_q(\mathrm{SU}_{n+1})$-ideal generated by the elements,  
\begin{align*}
  e_{ji} \otimes  e_{ji'} + q e_{ji'} \otimes  e_{ji}, ~~~~~~~  e_{ji} \otimes  e_{j'i} +  q e_{j'i} \otimes  e_{ji},  ~~~~~~~~~~~~ e_{ji} \otimes  e_{ji} =  0,\\
 ~   e_{ji'} \otimes  e_{j'i} + e_{j'i} \otimes  e_{ji'}  + \nu e_{j'i'} \otimes  e_{ji}, ~~~~~~~~~~~~~~~~~~
\end{align*}
where $(i,j),(i',j') \in \Delta^+$, and $j' > j > i' > i$, and by the elements 
\begin{align*}
e_{jk} \otimes  e_{j'k'} + q^{-\delta_{jk'}} e_{j'k'} \otimes  e_{jk}, 
\end{align*}
where  $(j,k), (j',k') \in \Delta^+$,  and  $j' > j$ and $k' > k$. 
\end{lem}
\begin{proof}
It is easily checked that $\omega(g) = 0$, for all $g \in G_1$. For example, 
\begin{align*}
\omega(u_{ij}) = \sum_{a=1}^{n+1} [u^+_{ia}] \otimes [u^+_{aj}] = 0,
\end{align*}
since $[u_{ij}] = 0$, for all $(j,i) \in \Delta^+$. 
For a generator $u_{ji}u_{kl} \in G_2$, it follows from Lemma \ref{lem:MCProduct} that
\begin{align*}
\omega(u_{ji}u_{kl}) =  \sum_{a=1}^{n+1} [u_{ji}u_{ka}] \otimes [u_{al}] + \sum_{a=1}^{n+1} [u_{ka}] \otimes [u_{ji}u_{al}] + \sum_{a,b=1}^{n+1} [u^+_{ja}u_{kb}] \otimes [u^+_{ai}u_{bl}] = 0,
\end{align*}
where the vanishing of each summand can be concluded from the set of generators given in Proposition \ref{prop:Gens}. Thus we see that the given relations are generated by the elements of $G_3^+$.

Let us look at a general element $u_{ji}u_{lk}$ in $G_3^+$. As we will see, each generator gives a non-zero degree two relation, whose form depends on the relationship of the associated pair of root vectors $\alpha_{ij}$ and $\alpha_{kl}$. We group distinct pairs of positive roots into four types as presented in the following table. (We recall from Appendix \ref{app:Lusztigroots} that $(-,-)$ denotes the real inner product of the $A_n$-series root system.)
\begin{center} \label{table:matchingindex}
\begin{tabular}{ |c|c|c| } 
 \hline
 & &\\
\textrm{Type} & $(\beta,\gamma)$   & $\prec$\textrm{ Comparable}  \\
  & &  \\
 \hline
 & & \\
\textrm{(a)} & $> 0$ &  \textrm{ \checkmark}  \\ 
  & & \\
\textrm{(b)} & $< 0$ &  \textrm{ \checkmark}  \\ 
  & & \\
\textrm{(c)} & $= 0$ &  \textrm{ \xmark}  \\ 
  & & \\
\textrm{(d)} & $= 0$ &  \textrm{\checkmark} \\ 
  & & \\
 \hline
\end{tabular}
\end{center}
Moreover, we find it convenient to consider the situation where the two roots coincide as an additional Type (e).

Type (a): ~  We further subdivide this case into two sub-cases. First we consider the situation where $j=l$. We can assume, without loss of generality, that $k < i$, since $u_{ji}u_{jk}$ is a scalar multiple of $u_{jk}u_{ji}$. In this case, it follows from \eqref{eqn:quantumaction} that 
\begin{align*}
\omega(u_{ji}u_{jk}) =  e_{jk} \otimes e_{ji}  + qe_{ji} \otimes e_{jk},
\end{align*}
which gives the first relation. 

Next we consider the situation where  $i=k$. We can assume, without loss of generality, that $l < j$, since $u_{ji}u_{li}$ is a scalar multiple of $u_{li}u_{ji}$. Just as for the previous case, $\omega(u_{ji}u_{li})$ gives the second relation.

Type (b): ~ Explicitly, for our pair of roots $\alpha_{ji}$ and $\alpha_{lk}$ this is the case for which $i=l$, or of $j=k$. However, since $u_{ji}u_{lj} \not \in G$, we need only consider the case of $u_{ji}u_{ik}$. We see
\begin{align*}
\omega(u_{ji}u_{ik}) = q^{-1} e_{ji} \otimes e_{ik} +  e_{ik} \otimes e_{ji},
\end{align*}
which gives us, up to scalar multiple, the fifth relation for the special case of $i'=j$.

Type (c): ~  Explicitly, this is the case of $j>l$ and $i>k$ or of $l >j$ and $k>i$. It follows from the relations of $\OO_q(\mathrm{SU}_{n+1})$ that in the later case
$$
u_{ji}u_{lk} = u_{lk}u_{ji} + \nu u_{jk}u_{li}.
$$
Generators of the form $u_{jk}u_{li}$ are consider below in case $(d)$. Hence, without loss of generality, we can assume that $j>l$ and $i>k$. Making this assumption, we have 
$$
\omega(u_{ji}u_{lk}) = e_{ji} \otimes e_{lk} + e_{lk} \otimes e_{ji},
$$
which gives us the fifth relation for the special case of $j \neq i'$.

Type (d): ~ Explicitly, this is the case of  $l < j$ and $k > i$ or of $l > j$ and $k < i$. Since in both cases the generators $u_{ji}$ and $u_{lk}$ commute, we can, without loss of generality assume that $j > l$ and $k > i$. This gives us the relation 
\begin{align*}
\omega(u_{lk}u_{ji}) = e_{lk} \otimes e_{ji} + e_{ji} \otimes e_{lk} + \nu e_{li} \otimes e_{jk}.
\end{align*} 
which implies the fourth relation.

Type (e):  ~ Finally we treat the case of $(j,i) = (l,k)$. This gives the relation
\begin{align*}
\omega(u_{ji}u_{ji}) = \nu e_{ji} \otimes e_{ji},
\end{align*}
which is to say the third relation.

It remains to show that the linear span of this collection of relations is a right $\OO_q(\mathrm{SU}_{n+1})$-submodule of $\Lambda^{(0,1)}_q \otimes \Lambda^{(0,1)}_q$. Observe that any diagonal generator $u_{kk}$ will act on our relations by a scalar multiple. Moreover, any generator $u_{kl}$, for $k < l$, will act trivially. Thus we need only look at the action of strictly lower triangular generators. We treat each of the five relations on a case by case basis.

1. ~ The only strictly lower triangular generators that act on relations of the first type are of the form $u_{pj}$, for some $(j,p) \in \Delta^+$. Multiplying by $u_{pj}$ gives the element
\begin{align*}
(e_{ji} \otimes e_{ji'} + qe_{ji'} \otimes e_{ji})u_{pj} = & \, \nu (e_{ji} \otimes e_{pi'} + qe_{pi} \otimes e_{ji'} + qe_{ji'} \otimes e_{pi} + q^2 e_{pi'} \otimes e_{ji})\\
= & \, q \nu \left(e_{pi} \otimes e_{ji'} + e_{ji'} \otimes e_{pi} + q e_{pi'} \otimes e_{ji} + q^{-1}e_{ji} \otimes e_{pi'}\right)\\
= & \, q \nu\big(e_{pi} \otimes e_{ji'} + e_{ji'} \otimes e_{pi} + \nu e_{pi'} \otimes e_{ji} \\
& ~~~~~~~ + q^{-1}(e_{pi'} \otimes e_{ji} + e_{ji} \otimes e_{pi'})\big).
\end{align*}
Thus we see that the element lies in the linear span of the given relations.

2. ~ The strictly lower triangular generators that act on relations of the second type are of the form $u_{pj}$ or $u_{pj'}$, for $(j,p), (j',p) \in \Delta^+$. For $u_{pj}$, we see that 
\begin{align*}
(e_{ji} \otimes e_{j'i} + qe_{j'i} \otimes e_{ji})u_{pj} = \nu \Big((1+\theta(p-j')q\nu)e_{pi} \otimes e_{j'i} + q^{1+\delta_{pj'}} e_{j'i} \otimes e_{pi}\Big),
\end{align*}
where $\theta$ is the Heaviside step function. This needs to be further divided into the following three cases, depending on whether $p < j'$, $p=j'$, or $p > j'$. It is then easily checked that for each case we again get a scalar multiple of one of the given relations. The case of the generator $u_{pj'}$ is dealt with similarly.

3. The only non-diagonal generators that act on the relation of the third type are of the form $u_{pj}$, for $(j,p) \in \Delta^+$. For such a generator, we see that 
\begin{align*}
(e_{ji} \otimes e_{ji})u_{pj} = \nu(e_{ji} \otimes e_{pi} + q e_{pi} \otimes e_{ji}).
\end{align*} 
Thus we get a scalar multiple of a relation of the second type.

4. The only strictly lower triangular generators that act on relations of the fourth type are of the form $u_{pj}$ or $u_{pj'}$, for $(p,j), \, (p,j') \in \Delta^+$. Multiplying the relation $e_{j'i'} \otimes e_{ji} +  q^{\delta_{i'j}} e_{ji} \otimes e_{j'i'}$ by the generator $u_{pj}$, we get the expression
\begin{align*}
q^{\delta_{i'j}} \nu\!\left(q^{\delta_{j'p} - \delta_{i'p}}  e_{j'i'} \otimes e_{pi} +   e_{pi} \otimes e_{j'i'} +   \theta(p-j') \nu e_{pi'} \otimes e_{j'i}\right)\!.
\end{align*} 
This needs to be further divided into the following three cases, depending on whether $p < j'$, $p=j'$, or $p > j'$. It is then easily checked that for each case we again get a scalar multiple of one of the given relations. The simpler case of the generator $u_{pj'}$ is dealt with similarly.

5. ~ A relation of the fifth type can be acted on non-trivially by strictly lower triangular generators of the form $u_{pj}$ or $u_{pj'}$, for $(j,p), (j',p) \in \Delta^+$. Multiplying by $u_{pj}$, and dividing by $\nu$, gives the element 
\begin{align*}
e_{pi'} \otimes e_{j'i} + q^{\delta_{pj'}}(e_{j'i} \otimes e_{pi'} + \nu e_{j'i'} \otimes e_{pi}) 
+ \theta(p-j')\nu \left(e_{pi} \otimes e_{j'i'} + \nu e_{pi'} \otimes e_{j'i}\right).
\end{align*}
This needs to be further divided into the following three cases, depending on whether $p < j'$, $p=j'$, or $p > j'$. It is then easily checked that for each case we again get a linear combination of the given relations. The simpler case of the generator $u_{pj'}$ is dealt with similarly.

From this we can conclude that the linear span of the set $\{\omega(g) \, | \, \textrm{ for } g \in G \}$  is closed under the right action of $\OO_q(\mathrm{SU}_{n})$. It now follows from Lemma \ref{lem:MCProduct} that the given relations span the space of degree two relations.
\end{proof}

We will now produce a compact presentation of the relations, given in terms of the \mbox{$A_{n}$-root} system. Interestingly, the pairs of basis elements that fail to commute up to a scalar correspond to pairs of positive roots which are comparable with respect to the partial order $\prec$ on $\Delta$. For the statement of the result, we find it convenient to introduce the notion of the \emph{prime pair} of a pair of orthogonal roots $\{\beta,\gamma\} = \{\alpha_{ab},\alpha_{cd}\}$:
$$
\{\beta',\gamma'\} = \{\alpha_{ab}',\alpha_{cd}'\} := \{\alpha_{ad},\alpha_{cb}\}.
$$
It is important to note that if $\beta$ and $\gamma$ are comparable, and both are positive, then $\beta'$ and $\gamma'$ are also both positive. We will also use the convex order $\leq$ associated to our choice of reduced decomposition of $w_0$, as presented in Appendix \ref{subsection:AnRootSystem}. In particular, we put an ordering on the basis elements of the cotangent space, by setting 
\begin{align*}
e_{\gamma} < e_{\gamma'}, & & \text{ if } \gamma < \gamma',
\end{align*}
with respect to the convex order on $\Delta^+$.

\begin{thm} \label{thm:THERELATIONS}
For all $\beta \leq  \gamma \in \Delta^+$, the relations 
\begin{align*}
e_{\beta} \wedge e_{\gamma} \, = & \,- e_{\gamma} \wedge e_{\beta} - \nu \, e_{\beta'} \wedge e_{\gamma'}, & & \textrm{ for } (\beta,\gamma) = 0 \textrm{ and } \gamma  \prec \beta, \\
e_{\beta} \wedge e_{\gamma} \, =  & \, - q^{(\beta,\gamma)} e_{\gamma} \wedge e_{\beta}, & & \textrm{ otherwise},
\end{align*}
give a full set of relations for the algebra $\Lambda^{(0,\bullet)}_q$.
\end{thm}
\begin{proof}
We establish this presentation by showing that, for each pair of positive roots $\{\beta,\gamma\}$, the relation given above coincides with a corresponding relation in Lemma \ref{lem:THERELATIONS}, and that all relations in Lemma \ref{lem:THERELATIONS} arise in this way. 
Note first that a pair of positive orthogonal roots $\{\beta,\gamma\}$ satisfies $\beta \prec \gamma$ if and only if it is of the form $\{\alpha_{ij'},\alpha_{i'j}\}$, for $i <i', \, j < j'$.  Thus, for this case, the relation given above corresponds to the fourth relation in Lemma \ref{lem:THERELATIONS}.  
The case when $\beta$ and $\gamma$ are positive, orthogonal, and \emph{not} comparable happens precisely when the pair of roots is of the form $\{\alpha_{kj},\alpha_{k'j'}\}$, for $k < k', \, j < j'$. Thus this case corresponds to the fifth relation in Lemma \ref{lem:THERELATIONS}, for $j' \neq k$.

We have $(\beta,\gamma) < 0$ precisely when the pair of roots is of the form $\{\alpha_{jk},\alpha_{kk'}\}$, for $k < k'$. Thus the relation above corresponds to the fifth relation in Lemma \ref{lem:THERELATIONS}, for $j'=k$. 
Next we consider the case when $(\beta,\gamma) > 0$. This happens precisely when the pair of roots is of the form $\{\alpha_{ij},\alpha_{ij'}\}$, for $j < j'$,  or $\{\alpha_{ij},\alpha_{i'j}\}$, for $i < i'$. For both pairs of roots 
$
(\beta,\gamma) = 1
$
meaning that we recover the first and second relations in Lemma \ref{lem:THERELATIONS}. 
Finally, the case $\beta = \gamma$ corresponds to  the third relation in Lemma \ref{lem:THERELATIONS}.
\end{proof}

\begin{cor} \label{cor:BergmannBasis}
For  $k=1, \dots, |\Delta^+|$, a basis of $\Lambda_q^{(0,k)}$ is given by 
\begin{align*}
\Big\{  e_{\gamma_1} \wedge \cdots \wedge e_{\gamma_k} \,|\, \gamma_1 < \cdots < \gamma_k \in \Delta^+ \Big\}. 
\end{align*}
In particular, it holds that
\begin{align*}
\mathrm{dim}\Big(V^{(0,k)}\Big) = \binom{\,|\Delta^+|\,}{k}, & & \textrm{ and } & & \mathrm{dim}\Big(V^{(0,\bullet)}\Big) = 2^{|\Delta^+|}.
\end{align*}
\end{cor}
\begin{proof}
We will use Bergman's diamond lemma \cite{BergDiam} to show that the proposed basis is indeed a basis. We first extend the convex order on the basis elements $e_{\gamma}$, for $\gamma \in \Delta^+$, to the degree lexicographic order $\geq$ on the associated basis of the tensor algebra of the cotangent space $\Lambda^{(0,1)}_q$. We see that the relation set given in Theorem \ref{thm:THERELATIONS} gives a compatible reduction system. Explicitly, the leading terms are given by,  for  $\beta, \gamma \in \Delta^+$, 
\begin{align*}
e_{\beta} \otimes e_{\beta}, & &  e_{\gamma} \otimes e_{\beta}, \textrm{ for } (\beta,\gamma) = 0 \textrm{ and } \gamma  \prec \beta, & & e_{\beta} \otimes e_{\gamma} \textrm{ otherwise}.
\end{align*}

Since the reduction system is composed of homogeneous polynomials of degree $2$ there are no inclusion ambiguities. Thus we need only concern ourselves with overlap ambiguities. Such an ambiguity will be of the form 
\begin{align*}
e_{\gamma} \otimes e_{\beta} \otimes e_{\alpha}, & & \textrm{ where } \alpha \leq \beta \leq \gamma.
\end{align*}
Let us write this element as 
\begin{align*}
e_{ja} \otimes e_{kb} \otimes e_{lc}, & & \textrm{ where } j \leq k \leq l.
\end{align*}
Reducing with respect to the leading term $e_{ja} \otimes e_{kb}$, we arrive at the expression
\begin{align*}
- q^{(\alpha_{ja},\alpha_{kb})} e_{kb} \otimes e_{ja} \otimes e_{lc} - \theta(k-j)\theta(a-b) \nu e_{ka} \otimes e_{jb} \otimes e_{lc}.
\end{align*}
Reducing with respect to the leading term $e_{kb} \otimes e_{lc}$, we arrive at the expression
\begin{align*}
- q^{(\alpha_{kb},\alpha_{lc})} e_{ja} \otimes e_{lc} \otimes e_{kb} - \theta(l-k)\theta(b-c) \nu e_{ja} \otimes e_{lb} \otimes e_{kc}.
\end{align*}


To resolve this general ambiguity, we need to consider the four possible cases
\begin{align} \label{eqn:BergmannSpecialCases}
j = k = l, & & j = k < l, & & j < k = l, & & j < k < l. 
\end{align}
For the first case, we see that all three basis elements commute up to a scalar multiple, and so, the ambiguity is clearly resolvable. 


For the second case in \eqref{eqn:BergmannSpecialCases}, we need to further consider the possible relationships between $a,b$, and $c$. When
$
a \leq b \leq c,
$
we see that all three basis elements commute up to a scalar multiple, and so, the ambiguity is clearly resolvable. Whenever $a > b$, we do not have an overlap ambiguity, so this possibility does not need to be considered. For 
\begin{align*}
a = b > c
\end{align*}
reducing $e_{ja} \otimes e_{ja} \otimes e_{lc}$ with respect to the  leading term $e_{ja} \otimes e_{ja}$ gives zero. Reducing with respect to the leading term $e_{ja} \otimes e_{lc}$, we arrive at the expression
$$
- e_{ja} \otimes e_{lc} \otimes e_{ja} - \nu e_{ja} \otimes e_{la} \otimes e_{jc}. 
$$
Reducing this further, we arrive at the expression
$$
e_{lc} \otimes e_{ja} \otimes e_{ja} - q \nu e_{la} \otimes e_{ja} \otimes e_{jc} + q \nu e_{la} \otimes e_{ja} \otimes e_{jc} = e_{lc} \otimes e_{ja} \otimes e_{ja},
$$
and subsequently at zero. Thus the ambiguity can be resolved. Finally, we consider the case 
$$
a < b > c.
$$ 
Reducing with respect to the leading term $e_{ja} \otimes e_{jb}$ we arrive at the expression
$$
-q e_{jb} \otimes e_{ja} \otimes e_{lc}.
$$
Reducing with respect to the leading term $e_{jb} \otimes e_{lc}$ we arrive at the expression
$$
- e_{ja} \otimes e_{lc} \otimes e_{jb} - \nu e_{ja} \otimes e_{lb} \otimes e_{jc}.
$$
The reduced form of both these expressions depends on the relationship between $a$ and $c$. For the case $a=c$, we can reduce both expressions to
$$
-q^2 e_{la} \otimes e_{jb} \otimes e_{ja}.
$$ 
For the case $a < c$, we can reduce both expressions to
$$
-q e_{lc} \otimes e_{jb} \otimes e_{ja} - q\nu e_{lb} \otimes e_{jc} \otimes e_{ja}.
$$ 
Finally, for the case that $a > c$, we can reduce both expressions to
$$
-q e_{lc} \otimes e_{jb} \otimes e_{ja} +  \nu e_{lb} \otimes e_{ja} \otimes e_{jc} - q\nu e_{la} \otimes e_{jb} \otimes e_{jc}.
$$ 
Thus this ambiguity can be resolved.


The third and fourth cases given in \eqref{eqn:BergmannSpecialCases} can be resolved analogously. Thus we can conclude from Bergman's diamond lemma that we have a basis as claimed.
\end{proof}

\subsection{Some Low Rank Examples}

In this subsection, we find it instructive to present some low rank examples. We begin with the simplest case of  $U_q(\mathfrak{sl}_2)$. We then move to the case of $U_q(\mathfrak{sl}_3)$, whose structure, while still quite simple,  exhibits much more interesting behaviour. Finally, we look at the case of $U_q(\mathfrak{sl}_4)$ which exhibits behaviour more characteristic of the general case.

\begin{eg}
For the simplest case of $\OO_q(\mathrm{SU}_2)$, we see that $\Lambda^{(0,1)}_q$ is the one-dimensional vector space spanned by the basis element $e := e _{21} = [u_{21}]$, and its right $\OO_q(\mathrm{SU}_2)$-module structure is determined by  $eu_{11} = q^{-1}e_{11}$, and $eu_{22} = qe_{22}$, with all other generators $u_{ij}$ acting trivially. Since $e \wedge e = 0$, we have $\Lambda^k_q = 0$, for all $k > 1$.
\end{eg}

\begin{eg} \label{eg:SU3.module}
The case of $\OO_q(\mathrm{SU}_3)$ is more instructive. Here the cotangent space $\Lambda^{(0,1)}_q$ is a three-dimensional vector space with basis
\begin{align*}
e_{21} = [u_{21}]~  & & e_{32} = [u_{32}].~ &  & e_{31} = [u_{31}].
\end{align*}
%
The diagonal matrix generators of $\OO_q(\mathrm{SU}_3)$ act on the basis elements by a power of $q$, just as for $\OO_q(\mathrm{SU}_2)$.  
Moreover, we have a single non-diagonal action
\begin{align*}
e_{21}u_{32} = \nu e_{31}.
\end{align*}
In what follows, we find it instructive to present the non-diagonal actions in the form of a graph: Arrange the basis elements in lower triangular form and draw an arrow from one basis element $e$ to another $e'$ if there exists a generator $u_{ji}$ such that $eu_{ji} = \nu e'$. Thus the graph for $U_q(\mathfrak{sl}_3)$ is as follows:
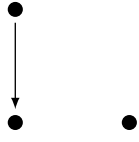
\begin{center}
\begin{tikzpicture}
  \draw (0,0) node[circle,fill,inner sep=2pt] (A) {}
        (0,-1.5) node[circle,fill,inner sep=2pt] (B) {}
        (1.5,-1.5) node[circle,fill,inner sep=2pt] (C) {};
  \draw[->, shorten >=2pt, shorten <=2pt, , >=latex, line width=0.5pt] (A) -- node[right] {} (B);
\end{tikzpicture}
\end{center}

The relations of $\Lambda^{(0,\bullet)}_q$ can be conveniently presented  in terms of the following diagram:
\begin{center}
\begin{tikzpicture}
  \draw (0,0) node[circle,fill,inner sep=2pt] (A) {}
        (0,-1.5) node[circle,fill,inner sep=2pt] (B) {}
        (1.5,-1.5) node[circle,fill,inner sep=2pt] (C) {};
  
  \draw[->, shorten >=2pt, shorten <=2pt, >=latex, line width=0.5pt, blue] (A) -- (B);
  \draw[->, shorten >=2pt, shorten <=2pt, >=latex, line width=0.5pt, blue] (B) -- (C);
  \draw[->, shorten >=2pt, shorten <=2pt, >=latex, line width=0.5pt, blue] (C) -- (A);
  
\end{tikzpicture}.
\end{center}
This gives the relations
\begin{align*}
e_{\beta} \wedge e_{\beta} = 0, & & e_{\beta} \wedge e_{\gamma} = - q e_{\gamma} \wedge e_{\beta},  
\end{align*}
whenever an arrow points from a basis element $e_{\beta}$ to a basis element $e_{\gamma}$. Note that the reversal of the dotted arrow follows from the fact that the inner product of the two roots $\alpha_{12}$ and $\alpha_{23}$ is negative.

Thus we see that $\Lambda^{(0,\bullet)}_q$ is a quantum affine space in the sense of  \cite[I.2]{BG02}, and that the higher orders have classical binomial dimension:
\begin{align*}
\dim\!\left(\Lambda^{(0,2)}_q\right) = 3, & & \dim\!\left(\Lambda^{(0,3)}_q\right) = 1, 
\end{align*}
with all higher forms being zero.

That the commutation relations take such a simple form can be understood as following from the explicit   dimensions of the weight spaces of $\Lambda^{(0,1)}_q \otimes \Lambda^{(0,1)}_q$:
\begin{align*}
\begin{pmatrix}
\,|e_{21} \otimes e_{21}| & |e_{21} \otimes e_{32}| & |e_{21} \otimes e_{31}| \, \\
\,|e_{32} \otimes e_{21}| & |e_{32} \otimes e_{32}| & |e_{32} \otimes e_{31}| \, \\
\,|e_{31} \otimes e_{21}| & |e_{31} \otimes e_{32}| & |e_{31} \otimes e_{31}| \, \\
\end{pmatrix}_{} ~~~~ = ~~~~
\begin{pmatrix}
2 \alpha_1 & \alpha_1 + \alpha_2 & 2\alpha_1 + \alpha_2 \\
\alpha_1 + \alpha_2 & 2\alpha_2 & \alpha_1 + 2 \alpha_2 \\
\, 2 \alpha_1 + \alpha_2 & \alpha_1 + 2 \alpha_2 & 2 \alpha_3 \\
\end{pmatrix}_{.}
\end{align*}
Since the calculus is right $\OO(T^{2})$-covariant, it follows that $I^{(2)}$ is a $U_q(\mathfrak{h})$ submodule of $\Lambda^{(0,1)}_q \otimes \Lambda^{(0,1)}_q$. In other words, $I^{(2)}$ is homogeneous with respect to the $\mathbb{Z}^2$-grading of the tensor product. Thus we can choose a basis of the degree two relations which is of the form
\begin{align*}
e_{ji} \otimes e_{kl} + \lambda e_{kl} \otimes e_{ji}, & & e_{ji} \otimes e_{ji}, &  &\textrm{ where } (j,i),(k,l) \in \Delta^+.
\end{align*}
Thus we see that, for this simple case, the form of the relations is dictated by the dimension of the weight spaces.

We finish with an observation that the commutation relations of the root vectors $E_{ji}$ are intimately connected with the commutation relations of the dual forms $e_{ji}$. For sake of notational convenience, we first relabel the root vectors 
\begin{align*}
X_1 := E_{21} = E_1, & & X_2 := E_{31} = [E_2,E_1]_{q^{-1}}, & & X_3 := E_{32} = E_2.
\end{align*}
For any $x \in I$, we see that 
\begin{align*}
[x_{(1)}] \otimes [x_{(2)}] = & \, \sum_{i,j} \langle X_i,x_{(1)} \rangle \langle X_j,x_{(2)} \rangle  e_i \otimes e_j 
=  \,  \sum_{i,j = 1}^3  \langle X_iX_j, x \rangle e_i \otimes e_j. 
\end{align*}
Separating this into its homogeneous summands we see that the elements
\begin{align} \label{eqn:firstXrels}
\langle X_iX_j, x \rangle e_i \otimes e_j + \langle X_jX_i, x \rangle e_j \otimes e_i, & &  \langle X_i^2,x\rangle e_i \otimes e_i, & &  \text{ ~~~ for } 1 \leq i < j \leq 3,
\end{align}
are all contained in $I^{(2)}$. Now in the linear dual of the ideal $I$, we have the following $q$-commutation relations
\begin{align*}
X_1X_2 = q^{-1} X_2X_1,  & & X_1X_3 = q X_3X_1, & & X_2X_3 = q^{-1}X_3X_2,
\end{align*}
where the first and third relations follow from the quantum Serre relations of $U_q(\mathfrak{sl}_3)$. This allows us to rewrite \eqref{eqn:firstXrels} as 
\begin{align*}
\langle E_i^2,x\rangle e_i, & & \langle X_iX_j, x \rangle(e_i \otimes e_j + q^{\chi_{ij}} e_j \otimes e_i), \textrm{ ~~~ for } 1 \leq i< j \leq 3, 
\end{align*}
where $\chi_{12} = \chi_{23} = 1$ and $\chi_{13} = -1$. Thus we see that the dimension of the space of relations is bounded above by six. This means that $ \omega(G)$ (where $G$ is the generating set given in Proposition \ref{prop:Gens}) and $I^{(2)}$ coincide. This offers an alternative proof for the fact that $\omega(G)$  is closed under the right action of $\OO_q(\mathrm{SU}_3)$, as established in Lemma \ref{lem:THERELATIONS}.
\end{eg}

\begin{eg} \label{eg:exterior.sl4}
Finally, we treat the case of $U_q(\mathfrak{sl}_4)$. The $6$-dimensional cotangent space $\Lambda^{(0,1)}_q$ admits non-zero actions by four non-diagonal matrix generator $u_{ji}$ which we present in the form of a diagram,  just as for $U_q(\mathfrak{sl}_3)$:
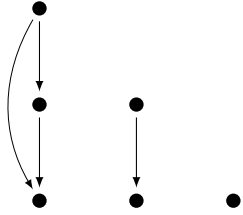
\begin{center}
\begin{tikzpicture} [scale=0.85]
\draw (0,0) node[circle,fill,inner sep=1.9pt] (A) {}
(0,-1.5) node[circle,fill,inner sep=1.9pt] (B) {}
(0,-3) node[circle,fill,inner sep=1.9pt] (C) {}
(1.5,-1.5) node[circle,fill,inner sep=1.9pt] (D) {}
(1.5,-3) node[circle,fill,inner sep=1.9pt] (E) {}
(3,-3) node[circle,fill,inner sep=1.9pt] (F) {};
\draw[->, shorten >=2pt, shorten <=2pt, , >=latex, line width=0.4pt] (A) -- node[right] {} (B);
\draw[->, shorten >=2pt, shorten <=2pt, , >=latex, line width=0.4pt] (B) -- node[right] {} (C);
\draw[->, shorten >=2pt, shorten <=2pt, , >=latex, line width=0.4pt] (D) -- node[right] {} (E);
\draw[->, shorten >=2pt, shorten <=2pt, , >=latex, line width=0.3pt, bend right] (A) to (C);
\end{tikzpicture}
\end{center}
We note that as an $\OO_q(\mathrm{SU}_4)$-module $\Lambda^{(0,1)}_q$ decomposes into a sum of three obvious submodules. Moreover, the first two submodules are reducible, and consist of a series of one-dimensional extensions of a one-dimensional submodule.

In this case the algebra $\Lambda^{(0,\bullet)}_q$ is no longer a quantum affine space. This can be seen as a consequence of the fact that the weight space of weight $\alpha_1 + 2\alpha_2 + \alpha_3$ in $\Lambda^{(0,1)}_q \otimes \Lambda^{(0,1)}_q$ is four dimensional. This permits the relation 
\begin{align} \label{eqn:sl4incomparable}
e_{32} \wedge e_{41} = - e_{41} \wedge e_{32} + \nu e_{31} \wedge e_{42}. 
\end{align}
Let us now give a graphic presentation of the relations, just as we did for $U_q(\mathfrak{sl}_3)$ above. For the case when the roots have a non-zero pairing we have the following diagrams
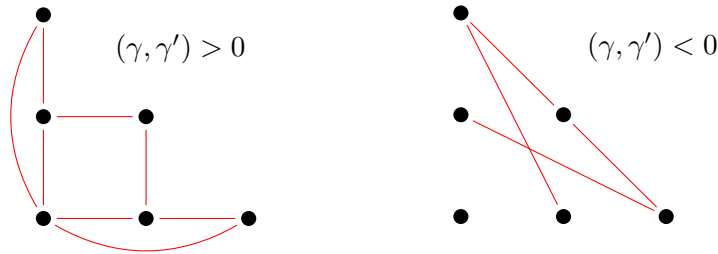
\begin{align*}
\begin{tikzpicture} [scale=0.9]
\draw (0,0) node[circle,fill,inner sep=2pt] (A) {}
(0,-1.5) node[circle,fill,inner sep=2pt] (B) {}
(0,-3) node[circle,fill,inner sep=2pt] (C) {}
(1.5,-1.5) node[circle,fill,inner sep=2pt] (D) {}
(1.5,-3) node[circle,fill,inner sep=2pt] (E) {}
(3,-3) node[circle,fill,inner sep=2pt] (F) {}
(2,-0.5) node[circle] (G) {$(\gamma,\gamma')>0$};
\draw[shorten >=2pt, shorten <=2pt, , >=latex, line width=0.4pt, red] (A) -- node[right] {} (B);
\draw[ shorten >=2pt, shorten <=2pt, , >=latex, line width=0.4pt, red] (B) -- node[right] {} (D);
\draw[ shorten >=2pt, shorten <=2pt, , >=latex, line width=0.4pt, red] (B) -- node[right] {} (C);
\draw[ shorten >=2pt, shorten <=2pt, , >=latex, line width=0.4pt, red] (C) -- node[right] {} (E);
\draw[ shorten >=2pt, shorten <=2pt, , >=latex, line width=0.4pt, red] (D) -- node[right] {} (E);
\draw[ shorten >=2pt, shorten <=2pt, , >=latex, line width=0.4pt, red] (E) -- node[right] {} (F);
\draw[ shorten >=2pt, shorten <=2pt, , >=latex, line width=0.4pt, red, bend right] (A) to (C);
\draw[ shorten >=2pt, shorten <=2pt, , >=latex, line width=0.4pt, red, bend right] (C) to (F);
\end{tikzpicture}
& &
\begin{tikzpicture} [scale=0.9]
\draw (0,0) node[circle,fill,inner sep=2pt] (A) {}
(0,-1.5) node[circle,fill,inner sep=2pt] (B) {}
(0,-3) node[circle,fill,inner sep=2pt] (C) {}
(1.5,-1.5) node[circle,fill,inner sep=2pt] (D) {}
(1.5,-3) node[circle,fill,inner sep=2pt] (E) {}
(3,-3) node[circle,fill,inner sep=2pt] (F) {}
(2.8,-0.5) node[circle] (G) {$(\gamma,\gamma')<0$}
(1.5,-3.5) node[] (Z) {};
\draw[shorten >=2pt, shorten <=2pt, , >=latex, line width=0.4pt, red] (D) -- node[right] {} (A);
\draw[shorten >=2pt, shorten <=2pt, , >=latex, line width=0.4pt, red] (F) -- node[right] {} (D);
\draw[shorten >=2pt, shorten <=2pt, , >=latex, line width=0.4pt, red] (E) -- node[right] {} (A);
\draw[shorten >=2pt, shorten <=2pt, , >=latex, line width=0.4pt, red] (F) -- node[right] {} (B);
\end{tikzpicture}
\end{align*}
where for any two basis elements $e_{\beta}$ and $e_{\gamma}$ connected by a red line, such that $\beta \leq \gamma$,  we have the relation
$$
e_{\beta} \wedge e_{\gamma} = -q e_{\gamma} \wedge e_{\beta}.
$$
When the pairing of roots is zero, we have the incomparable and comparable situations: 
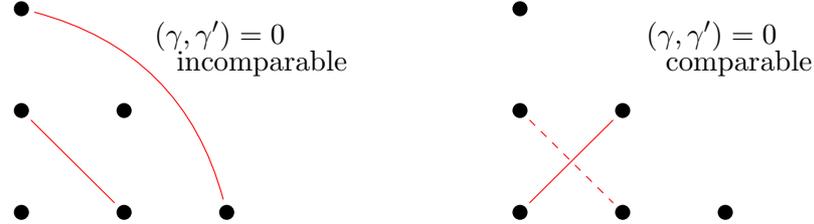
\begin{align*}
\begin{tikzpicture} [scale=0.9]
\draw (0,0) node[circle,fill,inner sep=2pt] (A) {}
(0,-1.5) node[circle,fill,inner sep=2pt] (B) {}
(0,-3) node[circle,fill,inner sep=2pt] (C) {}
(1.5,-1.5) node[circle,fill,inner sep=2pt] (D) {}
(1.5,-3) node[circle,fill,inner sep=2pt] (E) {}
(3,-3) node[circle,fill,inner sep=2pt] (F) {}
(2.9,-0.4) node[circle] (G) {$(\gamma,\gamma')=0$}
(3.3,-0.8) node[circle] (G) { ~~~incomparable};
\draw[-, shorten >=2pt, shorten <=2pt, , >=latex, line width=0.4pt, red] (B) -- node[right] {} (E);
\draw[-, shorten >=2pt, shorten <=2pt, , >=latex, line width=0.4pt, red, bend right] (F) to (A);
\end{tikzpicture}
& &
\begin{tikzpicture} [scale=0.9]
\draw (0,0) node[circle,fill,inner sep=2pt] (A) {}
(0,-1.5) node[circle,fill,inner sep=2pt] (B) {}
(0,-3) node[circle,fill,inner sep=2pt] (C) {}
(1.5,-1.5) node[circle,fill,inner sep=2pt] (D) {}
(1.5,-3) node[circle,fill,inner sep=2pt] (E) {}
(3,-3) node[circle,fill,inner sep=2pt] (F) {}
(2.8,-0.4) node[circle] (G) {$(\gamma,\gamma')=0$}
(3.2,-0.8) node[circle] (G) { comparable};
\draw[-, shorten >=2pt, shorten <=2pt, , line width=0.4pt, red] (C) -- node[right] {} (D);
\draw[-, dashed, shorten >=2pt, shorten <=2pt, , line width=0.4pt, red] (B) -- node[right] {} (E);
\end{tikzpicture}
\end{align*}
In the incomparable case all pairs commute up to scalar multiple, while for the comparable case we have the relation \eqref{eqn:sl4incomparable}, where the dotted line in the diagram joins the two factors of the extra \emph{quantum term} appearing on the right hand side of \eqref{eqn:sl4incomparable}.
\end{eg}

\begin{eg}
In this example we look at the use of Bergman's diamond lemma in Corollary \ref{cor:BergmannBasis} for the case of $U_q(\frak{sl}_4)$. As seen in the previous example, in this case we have a single relation that is not a commutation relation up to scalar multiple, namely the relation 
$$
e_{32} \wedge e_{41} = - e_{41} \wedge e_{32} - \nu e_{42} \otimes e_{31}.
$$
Thus the only overlap ambiguities that are not directly resolvable are the following
\begin{align*}
e_{21} \otimes e_{32} \otimes e_{41}, & & e_{31} \otimes e_{32} \otimes e_{41}, & & e_{32} \otimes e_{41} \otimes e_{41},\\
e_{32} \otimes e_{32} \otimes e_{41}, & & e_{32} \otimes e_{41} \otimes e_{42}, & & e_{32} \otimes e_{41} \otimes e_{43}.
\end{align*}
Considering the first identity, and performing the two possible reductions, we arrive at the two expressions
\begin{align*}
- q^{-1} e_{32} \otimes e_{21} \otimes e_{41},  & & - e_{21} \otimes e_{41} \otimes e_{32} - \nu e_{21} \otimes e_{42} \otimes e_{31}.
\end{align*}
Both these expressions can be further reduced to the single expression 
\begin{align*}
-e_{41} \otimes e_{32} \otimes e_{21} - \nu e_{42} \otimes e_{31} \otimes e_{21}
\end{align*}
meaning that the ambiguity is resolvable. Considering next the third element, and performing the two possible reductions, we arrive at $0$ and the expression
\begin{align*}
- e_{41} \otimes e_{32} \otimes e_{41} - \nu e_{42} \otimes e_{31} \otimes e_{41}.
\end{align*}
This expression can be further reduced to $0$, meaning that the ambiguity is resolvable. The other four ambiguities can be resolved analogously.
\end{eg}


\subsection{A Filtration for the Quantum Exterior Algebra} \label{subsection:filtration}

In this subsection, we introduce a filtration on the quantum exterior algebra and describe its associated graded algebra. Just as for the irreducible case treated in \cite{MTSUK}, the filtration will be used in the next subsection to show that our quantum exterior algebra is a Frobenius algebra. 

We begin by noting that our choice of basis for the cotangent space $\Lambda^{(0,1)}_q$ is indexed by the set $\Delta^+$ of positive roots of $\mathfrak{sl}_{n+1}$. Moreover, our choice of reduced decomposition of the longest element of the Weyl group gives $\Delta^+$ a total order (as explained in Appendix \ref{subsection:AnRootSystem}). Thus, following the approach of Appendix \ref{app:filtration}, we have an associated $\mathcal{Q}^+$-monoid grading on the tensor algebra of $\Lambda^{(0,1)}_q$. Finally, this induces a filtration on the quantum exterior algebra. 

\begin{lem} \label{lem:theassgradedlemma}
The associated graded algebra of the filtration is generated by the elements $e_{\gamma}$, for $\gamma \in \Delta^+$, subject to the relations  
\begin{align} \label{eqn:ass.grad.rels}
e_{\beta} \otimes e_{\beta}, & &  e_{\beta} \otimes e_{\gamma}  = - q^{(\beta,\gamma)} e_{\gamma} \otimes e_{\beta}, & & \textrm{ for } \beta < \gamma \in \Delta^+.
\end{align} 
\end{lem}
\begin{proof}
For $(\beta,\gamma) = 0$ and $\beta \prec \gamma$, we have the relation
$$
e_{\beta} \otimes e_{\gamma}  + e_{\gamma} \otimes e_{\beta} - \nu \, e_{\beta'} \otimes e_{\gamma'}.
$$
We see that the first two terms have the degree $\beta + \gamma$ and the third term has degree $\beta'+ \gamma'$. Since $\beta + \gamma < \beta'+ \gamma'$ in $\mathcal{Q}^+$, the leading term is 
$$
e_{\beta} \otimes e_{\gamma}  + e_{\gamma} \otimes  e_{\beta} = e_{\beta} \otimes  e_{\gamma}  + q^{(\beta,\gamma)}e_{\gamma} \otimes e_{\beta}.
$$
For two distinct roots $\beta,\gamma$ whose pairing is zero, or which are not comparable, we have the relation
$$
e_{\beta} \otimes e_{\gamma} + q^{(\beta,\gamma)} e_{\gamma} \otimes e_{\beta}.
$$
Since both summands have degree $\beta + \gamma$, the relation is equal to its leading term. Finally, it is clear that $e_{\beta} \otimes e_{\beta}$ is equal to its leading term. These relations imply that $\mathrm{gr}^{\mathscr{F}}$ has dimension at least $2^{|\Delta^+|}$, and since we know that $V^{(0,\bullet)}$ has dimension  $2^{|\Delta^+|}$, we can conclude that we have a full set of relations.
\end{proof}


\begin{eg}
For the special case of $U_q(\mathfrak{sl}_3)$, the order on the cotangent basis elements is given by 
\begin{align*}
e_{(2,1)} ~ < ~  e_{(3,1)} ~ < ~  e_{(3,2)}. 
\end{align*}
Thus for tensor products of order two, we see for example that
\begin{align*}
e_{(2,1)} \otimes e_{(3,1)} ~ < ~ e_{(3,2)} \otimes e_{(2,1)} ~ < ~ e_{(3,1)} \otimes e_{(3,2)}. 
\end{align*}
Moreover, since all the relations are clearly equal to their leading terms, we see that the associated graded algebra $\mathrm{gr}^\mathscr{F}$ is isomorphic to $\Lambda^{(0,\bullet)}_q$.
\end{eg}

\begin{eg}
Let us now look at the more involved case of $U_q(\mathfrak{sl}_4)$. An explicit presentation of the lexicographical order of the basis elements of the cotangent space is given by the diagram
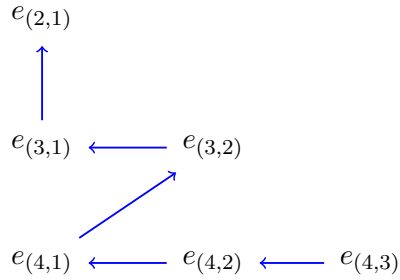
\begin{align*}
\begin{tikzpicture} [scale=0.9]
\draw (0,0.4) node (A) {$e_{(2,1)}$}
(0,-1.5) node (B) {$e_{(3,1)}$}
(0,-3.2) node (C) {$e_{(4,1)}$}
(2.5,-1.5) node (D) {$e_{(3,2)}$}
(2.5,-3.2) node (E) {$e_{(4,2)}$}
(4.8,-3.2) node (F) {$e_{(4,3)}$};
\draw[->, shorten >=2pt, shorten <=2pt, , line width=0.7pt, blue] (B) -- node[right] {} (A);
\draw[->, shorten >=2pt, shorten <=2pt, , line width=0.7pt, blue] (D) -- node[right] {} (B);
\draw[->, shorten >=2pt, shorten <=2pt, , line width=0.7pt, blue] (C) -- node[right] {} (D);
\draw[->, shorten >=2pt, shorten <=2pt, , line width=0.7pt, blue] (E) -- node[right] {} (C);
\draw[->, shorten >=2pt, shorten <=2pt, , line width=0.7pt, blue] (F) -- node[right] {} (E);
\end{tikzpicture}
\end{align*}
where the total order is  the transitive closure of the relations implied by the arrows of the diagram. Consider now the single commutation relation that is not a commutation up to scalar multiple
$$
e_{(3,2)} \otimes e_{(4,1)}  +  e_{(4,1)} \otimes e_{(3,2)} + \nu e_{(4,2)} \otimes e_{(3,1)}.
$$
Since $e_{(3,1)} < e_{(3,2)}$, we see that the leading term is given by 
\begin{align*}
e_{(3,2)} \otimes e_{(4,1)}  + e_{(4,1)} \otimes e_{(3,2)}
\end{align*}
showing that the associated graded algebra is a quantum affine algebra.
\end{eg}

\subsection{The Frobenius and Koszul Properties} \label{section:Frob}

In this subsection we establish some important additional  algebraic properties of the quantum exterior algebra. We use the associated graded algebra $\mathrm{gr}^{\mathscr{F}}$ to prove that the quantum exterior algebra is a Frobenius algebra, and then explicitly describe its Nakayama automorphism. Moreover,  we observe that the quantum exterior algebra is a Koszul algebra. 

A \emph{Frobenius algebra} is an associative algebra $A$ equipped with a non-degenerate bilinear  map
$
B: A \times A \to \mathbb{C}
$
satisfying 
\begin{align} \label{eqn:Frobenius}
B(ab,c) = B(a,bc), & & \textrm{ for all } a,b,c \in A.  
\end{align}

\begin{prop} 
For any linear isomorphism $\iota: \mathrm{gr}^{\mathscr{F}}_{2\rho} \to \mathbb{C}$, the bilinear map 
\begin{align*}
B: \mathrm{gr}^{\mathscr{F}} \otimes \mathrm{gr}^{\mathscr{F}}  \to \mathbb{C}, & & v \otimes w \mapsto \iota(v \wedge w)
\end{align*}
gives $\mathrm{gr}^{\mathscr{F}}$ the structure of a Frobenius algebra.
\end{prop}
\begin{proof}
It follows from the relation set given in \eqref{eqn:ass.grad.rels} that, for any $v \in \mathrm{gr}^{\mathscr{F}}$, there exists a $v'$ such that 
\begin{align*}
v \wedge v' = \lambda \bigwedge_{\gamma \in \Delta^+} e_{\gamma}, & & \textrm{ for some } \lambda \in \mathbb{C},
\end{align*}
where the basis elements are ordered according to our choice of convex ordering on $\Delta^+$. Thus $B$ is a non-degenerate bilinear form. Since \eqref{eqn:Frobenius} is clearly satisfied, we see that $B$ gives $\mathrm{gr}^{\mathscr{F}}$ the structure of a Frobenius algebra.
\end{proof}

The corollary below now follows from Bongale's theorem for filtered algebras, as presented in Appendix \ref{app:filtration}.

\begin{cor} \label{cor:Frobenius}
For any linear isomorphism $\iota: \Lambda^{(0,|\Delta^+|)}_q \to \mathbb{C}$, the bilinear map 
\begin{align*}
B: \Lambda^{(0,\bullet)}_q  \otimes  \Lambda^{(0,\bullet)}_q  \to \mathbb{C}, & & v \otimes w \mapsto \iota(v \wedge w)
\end{align*}
gives $\mathrm{gr}^{\mathscr{F}}$ the structure of a Frobenius algebra.
\end{cor}
\begin{proof}
From Bongale's theorem we know that a Frobenius structure is given by 
\begin{align*}
B: \Lambda^{(0,\bullet)}_q  \otimes  \Lambda^{(0,\bullet)}_q  \to \mathbb{C}, & & v \otimes w \mapsto \iota([v \wedge w]_{|\Delta^+|}),
\end{align*}
where $\iota: \mathrm{gr}^{\mathscr{F}}_{2\rho} \cong \mathbb{C}$ is some choice of linear isomorphism, and $\rho$ is the half-sum of positive roots of $\mathfrak{sl}_{n+1}$. The given $B$ now follows from the fact that  $\mathrm{gr}^{\mathscr{F}}_{2\rho}$ and $\Lambda^{(0,|\Delta^+|)}_q$ are obviously linearly isomorphic.
\end{proof}

For a general Frobenius algebra $A$,  there exists an algebra automorphism $\sigma:A \to A$ of $A$, uniquely defined by the identity $B(x,y) = B(y,\sigma(x))$, for all $x,y \in A$. We see that the bilinear form $B$ of a Frobenius algebra is symmetric if and only if $\sigma = \id$.  With a view to describing the Nakayama automorphism of our quantum exterior algebra, we first describe the Nakayama automorphism of $\mathrm{gr}^{\mathscr{F}}$. For the remainder of this subsection, we find it convenient to denote 
\begin{align} 
e_{\widehat{\gamma}} := \bigwedge_{~\gamma \, \neq \, \theta \, \in \Delta^+} e_{\theta} \in \Lambda^{(0,|\Delta^+|-1)}_q,
\end{align}
where the basis elements are ordered according to our choice of convex ordering on $\Delta^+$. By abuse of notation, we also use $e_{\widehat{\gamma}}$ to denote the corresponding product in  $\mathrm{gr}^{\mathscr{F}}$.

\begin{prop} 
The Nakayama automorphism of $\mathrm{gr}^{\mathscr{F}}$ is determined by 
\begin{align*}
\sigma(e_{\gamma}) = (-1)^{|\Delta^+|-1}e_{\gamma}, & &  \textrm{ for any } \gamma \in \Delta^+.
\end{align*}
\end{prop}
\begin{proof}
Note first that, for any $\gamma \in  \Delta^+$, the only basis element that pairs non-trivially with  $e_{\gamma}$ is $e_{\widehat{\gamma}}$. It follows from the relations given in Theorem \ref{thm:THERELATIONS} that 
\begin{align} \label{eqn:Nakayama.1}
B\!\left(e_{\gamma}, e_{\widehat{\gamma}}\right) = (-1)^{|\Delta^+|-1}q^{\chi_{\gamma} }B\left(e_{\widehat{\gamma}}, e_{\gamma}\right)\!,
\end{align}
where we have denoted 
$$
\chi_{\gamma} := -\sum_{\gamma > \theta \in \Delta^+} \, (\gamma,\theta) + \sum_{\gamma < \theta \in \Delta^+} \, (\gamma,\theta) .
$$
Now for an arbitrary positive root $\alpha_{ij} \in \Delta^+$, the positive roots that pair non-trivially with $\alpha_{ij}$ are precisely those of the form 
\begin{align*}
\alpha_{ik}, ~ \alpha_{fj}, ~ \alpha_{ai}, ~ \alpha_{jb}, & & \textrm{ for } (i,k), (f,j), (a,i), (j,b) \in \Delta^+.
\end{align*}
Let us now collect together those roots that are strictly less then $\alpha_{ij}$ with respect to the convex order $\leq$. We then pair them with $\alpha_{ij}$ and sum the resulting scalars to give
\begin{align*}
\sum_{k=i+1,}^{j-1} (\alpha_{ij}, \alpha_{ik}) + \sum_{f=1}^{i-1} (\alpha_{ij},\alpha_{fj}) + \sum_{a=1}^{i-1} (\alpha_{ij},\alpha_{ai}) =  j - i -1,
\end{align*} 
where we have observed that  the last two summands cancel. Following the same procedure for those roots that are strictly greater then $\alpha_{ij}$, we get
\begin{align*}
\sum_{k=j+1}^{n+1} (\alpha_{ij}, \alpha_{ik}) + \sum_{f=i+1}^{j-1} (\alpha_{ij},\alpha_{fj}) + \sum_{b=j+1}^{n+1} (\alpha_{ij},\alpha_{jb}) = j-i-1,
\end{align*} 
where we have observed that the first and the last summands cancel. Thus 
\begin{align*}
\sum_{\gamma < \theta \in \Delta^+} (\gamma,\beta) = \sum_{\gamma > \theta \in \Delta^+} (\gamma,\beta),
\end{align*}
meaning that the identity in \eqref{eqn:Nakayama.1} reduces to 
\begin{align} \label{eqn:Nakayama.2}
B\!\left(e_{\gamma}, e_{\widehat{\gamma}}\right) = (-1)^{|\Delta^+|-1}B\!\left(e_{\widehat{\gamma}}, e_{\gamma}\right)\!.
\end{align}
This identifies the Nakayama automorphism on the algebra generators $e_{\gamma}$, and hence on the whole algebra $\mathrm{gr}^{\mathscr{F}}$. 
\end{proof}

In the following corollary we use our description of the Nakayama automorphism of $\mathrm{gr}^{\mathscr{F}}$ to produce an analogous description of the Nakayama automorphism of $\Lambda^{(0,\bullet)}_q$. 

\begin{cor}  \label{cor:Nakayama}
The Nakayma automorphism of $\Lambda^{(0,\bullet)}_q$ is determined by 
\begin{align*}
\sigma(e_{\gamma}) = (-1)^{|\Delta^+|-1}e_{\gamma}, & & \textrm{ for any } \gamma \in \Delta^+.
\end{align*}
\end{cor}
\begin{proof}
For any $\gamma \neq \beta \in \Delta^+$, we see that the product $e_{\gamma} \wedge e_{\widehat{\beta}}$ has weight not equal to $2\rho$, and hence it is zero. For the product $e_{\gamma} \wedge e_{\widehat{\gamma}}$, we see that since the space of $(|\Delta^+|-1)$-forms is one-dimensional, there exists a scalar $c$ such that 
$
e_{\gamma} \wedge e_{\widehat{\gamma}} = c e_{\widehat{\gamma}} \wedge e_{\gamma}.
$
Moving to the associated graded algebra, we see that this implies the relation
$
e_{\gamma} \wedge e_{\widehat{\gamma}}  = c e_{\widehat{\gamma}} \wedge e_{\gamma},
$
where, by abuse of notation, we have not differentiated notationally between the product of $\Lambda^{(0,\bullet)}_q$ and the product of $\mathrm{gr}^{\mathscr{F}}$. Now we already know from the presentation of the Nakayama automorphism of $\mathrm{gr}^{\mathscr{F}}$ that 
$$
e_{\gamma} \wedge e_{\widehat{\gamma}} = (-1)^{|\Delta^+|-1} e_{\widehat{\gamma}} \wedge e_{\gamma}.
$$
Thus we can conclude that $c = (-1)^{|\Delta^+|-1}$.
\end{proof}

Thus we see that, just as in the classical case, the Frobenius structure of $\mathrm{gr}^{\mathscr{F}}$ and $\Lambda^{(0,\bullet)}_q$ is \emph{graded symmetric}, which is to say, for any two homogeneous elements $v$ and $w$ we have
$$
B(v,w) = (-1)^{|v||w|}B(w,v).
$$
Indeed, when $|\Delta^+|$ is odd, both Frobenius algebras are symmetric.

We finish by observing that  $\Lambda^{(0,\bullet)}_q$ is also a Koszul algebra. Recall that a Koszul algebra is a $\mathbb{Z}_{\geq 0}$-graded algebra admitting  a linear minimal graded free resolution. We refer the reader to the standard text \cite{Leonid.Quadratic} for more details on Koszul algebras.

\begin{prop}
The algebra $\Lambda^{(0,\bullet)}_q$ is a Koszul algebra.
\end{prop}
\begin{proof}
The algebra $\Lambda^{(0,\bullet)}_q$ is clearly a PBW-algebra in the sense of Priddy \cite[\textsection 4.1]{Leonid.Quadratic}. Thus it follows from  Priddy's theorem \cite[Theorem 3.1]{Leonid.Quadratic} that it is Koszul.
\end{proof}


\section{The Heckenberger--Kolb Dolbeault Complex for the Quantum Grassmannians}

In this section we examine the restriction of the dc $\Omega_q^{(0,\bullet)}(\mathrm{SU}_{n+1})$ to the quantum Grassmannians and show that we recover the anti-holomorphic Heckenberger--Kolb dc. The embedding of the higher Heckenberger--Kolb forms into the higher $\OO_q(\mathrm{SU}_{n+1})$-forms is a delicate issue, and necessitates the introduction of some general results  in \textsection \ref{section:generalHigherEmbeddings}.

\subsection{Some Preliminaries on Quantum Homogeneous Spaces}  \label{subsection:QHSpacesPrelims}

In this preliminary subsection, we recall the definition of a quantum homogeneous space, and the associated generalisations of the fundamental theorem of Hopf modules. 

Let $A$ be a Hopf algebra. We say that a left coideal subalgebra $B \sseq A$ is a \emph{quantum homogeneous $A$-space} if $A$ is faithfully flat as a right $B$-module and $B^+A = AB^+$. In \cite[Theorem 1]{Tak} it was shown that, for the Hopf algebra surjection $\pi_B:A \to A/B^+A$, with its associated right $\pi_B(A)$-coaction 
$
\Delta_{R,\pi_B} := (\id \otimes \pi_B) \circ \Delta: A \to A \otimes \pi_B(A),
$ 
the space of coinvariants 
$$
A^{\co(\pi_B(A))} := \Big\{b \in A \,|\, \Delta_{R,\pi_B}(b) = b \otimes 1 \Big\}
$$ 
is equal to $B$. We call $\pi_B(A)$ the \emph{isotropy Hopf algebra} of $B$.

For any quantum homogeneous space $B = A^{\co(\pi_B(A))}$, we define  $^A_B\mathrm{Mod}_B$ to be the category whose  objects are  left \mbox{$A$-comodules} \mbox{$\DEL_L:\mathcal{F} \to A \otimes \mathcal{F}$}, endowed with a $B$-bimodule structure such that   $\DEL_L(bfc) = \Delta(b)\DEL_L(f)\Delta(c)$,  for all  $f \in \mathcal{F}, \, b,c \in B$, and whose morphisms  are left $A$-comodule, $B$-bimodule, maps. We call this category the \emph{category of $(A,B)$-relative Hopf modules}.

Let ${}^{\pi_B \,}\!\mathrm{Mod}_B$ denote the category  whose objects are given by left $\pi_B(A)$-comodules $\Delta_L: V \to \pi_B(A) \otimes V$, endowed with a right $B$-module structure satisfying the identity $\Delta_L(vb) = v_{(-1)}\pi_B(b_{(1)}) \otimes v_{(0)}b_{(2)}$, for all $v \in V, \, b \in B$, and whose morphisms are left $\pi_B(A)$-comodule maps and right $B$-module maps.

Consider the functor $\Phi:{}^A_B\mathrm{Mod}_B \to {}^{\pi_B \,}\mathrm{Mod}_B$, given by $\Phi(\mathcal{F}) := \mathcal{F}/B^+\mathcal{F}$, where the left $\pi_B(A)$-comodule structure of $\Phi(\mathcal{\F})$ is given by 
$
\Delta_L[f] := \pi_B(f_{(-1)})\otimes [f_{(0)}],
$
with square brackets denoting the coset of an element in $\Phi(\mathcal{\F})$. In the other direction, we define a functor $\Psi: {}^{\pi_B \,}\mathrm{Mod}_B \to {}^A_B\mathrm{Mod}_B$ using the cotensor product $\square_{\pi_B(A)}$, which we find convenient to denote by $\square_{\pi_B}$. Setting $\Psi(V) := A \,\square_{\pi_B} V$, where the left $A$-comodule structure of $\Psi(V)$ is defined on the first tensor factor, the right $B$-module structure is the diagonal one, and if $\gamma$ is a morphism in ${}^{\pi_B\,}\mathrm{Mod}_B$, then $\Psi(\gamma) := \id \otimes \gamma$. 

An adjoint equivalence of categories between~${}^A_B\mathrm{Mod}_B$ and~${}^{\pi_B \,}\mathrm{Mod}_B$ is given by the functors $\Phi$ and $\Psi$, and the unit natural isomorphism
\begin{align*}
\unit: \F \to \Psi \circ \Phi(\F), & & f \mapsto f_{(-1)} \otimes [f_{(0)}].
\end{align*}
We call this \emph{Takeuchi's equivalence}. The \emph{dimension} $\mathrm{dim}(\F)$ of an object $\F \in {}^A_B\mathrm{Mod}_B$ is the vector space dimension of $\Phi(\F)$. 
As established in \cite[\textsection 1]{Tak}, for any $\mathcal{F} \in {}^A_B \mathrm{Mod}$, an isomorphism in the category ${}^A_A\mathrm{Mod}$ is given by 
\begin{align} \label{eqn:theALPHAmap}
\alpha_{\mathcal{F}}: A \otimes_B \mathcal{F} \to A \otimes \Phi(\mathcal{F}), & & a \otimes f \mapsto af_{(-1)} \otimes [f_{(0)}].
\end{align}

Consider next the full subcategory  $^A_B\textrm{Mod}_0$  of ${}^A_B\mathrm{Mod}_B$ consisting of those  objects $\F$ which satisfy $B^+\F = \F B^+$. The corresponding full subcategory ${}^{\pi_B\,}\mathrm{Mod}_0$ of ${}^{\pi_B\,}\mathrm{Mod}_B$ consists of those objects with the trivial right $B$-action, and is of course equivalent to the category of left $\pi_B(A)$-comodules. Both sub-categories come equipped with evident monoidal structures. Indeed,  Takeuchi's equivalence restricts to a monoidal equivalence between the subcategories, see \cite[\textsection 4]{MMF2} for further details.

\subsection{Some Remarks on Quantum Homogeneous Tangent Spaces} \label{subsection:remarksQHTS} 

In this subsection we recall the generalisation of the results of \textsection \ref{subsection:TangentSpaces} from Hopf algebras to quantum homogeneous spaces.

Let $A$ be a Hopf algebra, and $W \subseteq A^{\circ}$ a Hopf subalgebra of $A^{\circ}$, such that 
$$
B := \, {}^W\!A = \Big\{b \in A \,|\, b_{(1)} \langle w, b_{(2)}\rangle = \e(w)b, \textrm{ for all } w \in W \Big\}
$$
is a quantum homogeneous $A$-space, and denote by $B^{\circ}$ its dual coalgebra. A \emph{tangent space} for $B$ is a subspace $T \subseteq B^{\circ}$ such that $T \oplus \mathbb{C}1$ is a right coideal of $B^{\circ}$ and $\mathrm{ad}(W)T \subseteq T$. For any tangent space $T$, a right $B$-ideal of $B^{+} := B \cap \mathrm{ker}(\e)$ is given by 
\begin{align*}
I := \big\{ x \in B^+ \,|\, X(x) = 0, \textrm{ for all } X \in T \big\},
\end{align*}
meaning that the quotient $V^1(B) := B^+/I$ is naturally an object in the category ${}^{\pi_B}\mathrm{Mod}_B$. We call $V^1(B)$ the \emph{cotangent space} of $T$. Consider now the object 
\begin{align*}
\Omega^1(B) := A \square_{\pi_B} V^1(B).
\end{align*}
If $\{X_i\}_{i=1}^n$ is a basis for $T$, and $\{e_i\}_{i=1}^n$ is the dual basis of $V^1(B)$, then the map 
\begin{align*}
\exd: A \to \Omega^1(B), & & a \mapsto \sum_{i=1}^n (X_i \triangleright a) \otimes e_i
\end{align*}
is a derivation, and the pair $(\Omega^1(B),\exd)$ is a left $A$-covariant fodc over $B$. This gives a bijective correspondence between isomorphism classes of finite-dimensional tangent spaces and finitely-generated left $A$-covariant fodc \cite{HKTangent}.

Let $T_A$ be a finite-dimensional tangent space for the Hopf algebra $A$, with associated fodc $\Omega^1(A)$. The restriction of $\Omega^1(A)$ to $B$ is clearly a left $A$-covariant fodc, which we denote by $\Omega^1(B)$. Moreover, as shown in \cite[Corollary 9]{HKTangent}, the corresponding tangent space is given by $\mathrm{ad}(W)T_A|_B \subseteq B^{\circ}$. We note that the obvious map from $V^1(B)$, the cotangent space of $\Omega^1(B)$, to  $\Lambda^1(A)$, the cotangent space  of $\Omega^1(A)$, is injective if and only if $\mathrm{ad}(W)T_A|_B  = T_A|_B $. Finally, let $\{X_i\}_{i =1}^n$ be a basis for $T_A$, and denote by $\{e_i\}_{i=1}^n$ the corresponding dual basis of $\Lambda^1(A)$. Assume that a basis of the tangent space of the restricted fodc is given by the elements $\{X_i\}_{i \in J}$, for some subset $J \subseteq \{1, \dots, n\}$, where by abuse of notation, we do not distinguish notationally between an element $X_i \in T_A$ and its restriction to an element of $B^{\circ}$. It now follows that a basis for the image of $V^1(B)$ in $\Lambda^1(A)$ is given by $\{e_i\}_{i \in J}$.

\subsection{Lusztig Differential Calculi for the $A$-Series Quantum Flag Manifolds}

For $S$ a proper subset of the set of simple roots $\Pi$,  consider the Hopf subalgebra 
\begin{align*}
U_q(\mathfrak{l}_S) := \Big< K_i, E_j, F_j \,|\, i = 1, \ldots, n; \, j \in S \Big> \subseteq U_q(\mathfrak{sl}_{n+1}).
\end{align*} 
We denote $S^c := \Pi \backslash S$, and represent $S^c$ graphically by coloured nodes in the Dynkin diagram of $\mathfrak{sl}_{n+1}$. In our description of quantum tangent spaces below, we find it convenient to consider the sets of weights
\begin{align*}
\Delta_S^{+}:= (\mathbb{Z}_{\geq 0} \, S) \cap \Delta^{+}, & & \overline{\Delta^{+}_S} := \Delta^{+} \setminus \Delta_S^{+},
\end{align*}
where as usual $\Delta^+$ denotes the set of positive roots of $\mathfrak{sl}_{n+1}$.
Consider also the coideal subalgebra of $U_q(\mathfrak{l}_S)$-invariants
\begin{align*}
\O_q\big(\mathrm{SU}_{n+1}/L_S\big) := {}^{U_q(\mathfrak{l}_S)}\O_q(\mathrm{SU}_{n+1}),
\end{align*} 
with respect to the natural left $U_q(\mathfrak{sl}_{n+1})$-module structure on $\OO_q(\mathrm{SU}_{n+1})$. 
As is well-known, $\OO_q(\mathrm{SU}_{n+1}/L_S)$ is a quantum homogeneous space (see \cite[\textsection 5.4]{GAPP} for a detailed explanation of this fact in the notation of this paper). We call $\OO_q(\mathrm{SU}_{n+1}/L_S)$ the {\em quantum flag manifold} associated to $S$. 

When $S^c$ consists of a single simple root $\alpha_r$, we call the corresponding quantum flag manifold the \emph{quantum $r$-plane Grassmannian}, and denote it by $\O_q(\mathrm{Gr}_{n+1,r})$. 
It follows immediately from the definition of the algebra $\OO_q(\mathrm{SU}_{n+1}/L_S)$ that it contains as subalgebras those quantum Grassmannians $\OO_q(\mathrm{Gr}_{n+1,r})$ such that $\alpha_r \in S^c$. Moreover, it can be shown these subalgebras generate $\OO_q(\mathrm{SU}_{n+1}/L_S)$ as an algebra.

Consider now an arbitrary algebra $A$ endowed with a dc $\Omega^{\bullet}(A)$. For any subalgebra $B \subseteq A$, the sub-$B$-bimodule of $\Omega^{\bullet}(A)$ generated by elements of the form $\exd b$, for $b \in B$, forms a dc over $B$. We call this dc the \emph{restriction}  of $\Omega^{\bullet}(A)$ to $B$. Clearly, we have an analogous notion of restriction for fodc. 


\begin{defn}
For a choice of simple roots $S \subseteq \Pi$, we denote by 
$$
\Big(\Omega^{(0,\bullet)}_q(\mathrm{SU}_{n+1}/L_S), ~ \adel \Big)
$$ 
the restriction to  $\OO_q(\mathrm{SU}_{n+1}/L_S)$ of the dc $\Omega^{(0,\bullet)}_q(\mathrm{SU}_{n+1})$, and call it the \emph{Lusztig dc} of $\OO_q(\mathrm{SU}_{n+1}/L_S)$. Moreover, we denote by $T^{(0,1)}_S$, and $V^{(0,1)}_S$, the associated tangent space, and cotangent space, respectively. 
\end{defn}

In the following subsection we will investigate these dc for the special case of the quantum Grassmannians and show that they coincide with the celebrated Heckenberger--Kolb dc. Following this, we examine in greater detail the Lusztig dc of the full quantum manifold and prove a direct $q$-deformation of the classical Borel--Weil theorem. The remaining quantum flag manifolds will be treated in a subsequent work.

\subsection{The Heckenberger--Kolb FODC for the Quantum Grassmannians} \label{subsection:FOLusztigHK}

We say that a dc  over a quantum homogeneous $A$-space $B$ is \emph{irreducible} if it is simple as an $(A,B)$-relative Hopf module. Take any simple root $\alpha_r \in \Pi$, let $S^c = \{\alpha_r\}$, and denote the corresponding quantum Grassmannian by $\OO_q(\mathrm{Gr}_{n+1,r})$. In \cite{HK} Heckenberger and Kolb showed that  $\OO_q(\mathrm{Gr}_{n+1,r})$  admits precisely two non-isomorphic left $\OO_q(\mathrm{SU}_{n+1})$-covariant finite-dimensional irreducible fodc. We call them the \emph{holomorphic}, and \emph{anti-holomorphic}, \emph{Heckenberger--Kolb fodc}, and denote them respectively by 
\begin{align*}
\Omega^{(1,0)}_q(\mathrm{Gr}_{n+1,r}), & & \Omega^{(0,1)}_q(\mathrm{Gr}_{n+1,r}).
\end{align*}
Their respective tangent spaces are given by 
\begin{align*}
T_{\alpha_r^c}^{(1,0)} := \mathrm{ad}(U_q(\mathfrak{l}_S))F_r, & &   T_{\alpha_r^c}^{(0,1)}  := \mathrm{ad}(U_q(\mathfrak{l}_S))E_r.
\end{align*}
where to lighten notation, we have denoted $\alpha_r^c := \Pi \backslash \{\alpha_r\}$. Denoting
\begin{align*}
 \overline{\Delta^{+}_S} := \Big\{ \alpha_i + \cdots + \alpha_r + \cdots + \alpha_j \,|\, 1 \leq i \leq r \leq j \leq n  \Big\},
\end{align*}
it was shown in \cite[\textsection 3]{HKdR} that
\begin{align} \label{eqn:HKTangent}
T_{\alpha_r^c}^{(0,1)} = \mathrm{span}_{\mathbb{C}}\Big\{E_{\gamma} ~ | ~ \gamma \in  \overline{\Delta^{+}_S} \Big\} \subseteq \OO_q(\mathrm{Gr}_{n+1,r})^{\circ},
\end{align}
with an analogous identity for $T_{\alpha_r^c}^{(1,0)}$.  Moreover, as shown in \cite[\textsection 3.3]{HKdR}, the maximal prolongations  of the two Heckenberger--Kolb fodc have classical dimension, that is
\begin{align*}
\mathrm{dim}\!\left(\Omega^{(0,k)}_q(\mathrm{Gr}_{n+1,r})\right) = \mathrm{dim}\!\left(\Omega^{(k,0)}_q(\mathrm{Gr}_{n+1,r})\right) = \binom{(n+1-r)r}{k}
\end{align*}
for all $k = 0, \dots, (n+1-r)r$. We now observe that our fodc over $\OO_q(\mathrm{SU}_{n+1})$ restricts to the anti-holomorphic Heckenberger--Kolb dc.  

\begin{prop} \label{prop:HKFOEmbedding}
The Lusztig fodc of the quantum Grassmannian $\OO_q(\mathrm{Gr}_{n+1,r})$, which is to say, the restriction of $\Omega^{(0,1)}_q(\mathrm{SU}_{n+1)})$ to $\OO_q(\mathrm{Gr}_{n+1,r})$, is the anti-holomorphic Hecken-berger--Kolb fodc, for all $\alpha_r \in \Delta^+$. Moreover, the image of $V^{(0,1)}_{\alpha_r^c}$ in $\Lambda^{(0,1)}_q$ is spanned by the basis elements $e_{\gamma}$, for $\gamma \in \overline{\Delta^+_S}$.
\end{prop}
\begin{proof}
For any $\gamma \in \Delta^+_S$, the root vector $E_{\gamma}$ is a linear combination of elements of $U_q(\mathfrak{l}_S)$. Thus we have that $\langle E_{\gamma}, b \rangle = 0$, for any $b \in \OO_q(\mathrm{Gr}_{n+1,r})^+$, meaning that the image of $E_{\gamma}$ in $\OO_q(\mathrm{Gr}_{n+1,r})^{\circ}$ is zero. We now see that the restriction of $T^{(0,1)}$ (which we recall was defined in Corollary \ref{cor:tangentspace}) to the quantum Grassmannian, is spanned by the cosets, in $\OO_q(\mathrm{Gr}_{n+1,r})^{\circ}$, of the tangent vectors $E_{\gamma}$, for $\gamma \in \overline{\Delta^+_S}$.  From the remarks at the end of \textsection \ref{subsection:remarksQHTS}, we see that the tangent space of the restriction fodc is given by 
$$
\mathrm{ad}(U_q(\mathfrak{l}_S))\Big\{E_{\gamma} \,|\, \gamma \in \overline{\Delta^+_S}\Big\} = \mathrm{ad}(U_q(\mathfrak{l}_S))T^{(0,1)}_{\alpha_r^c} = T^{(0,1)}_{\alpha_r^c},
$$
where in the penultimate identity we have used the description of the Heckenberger--Kolb tangent space given in \eqref{eqn:HKTangent}, and in the last identity we have used the fact since $T^{(0,1)}_{\alpha_r^c}$ is a tangent space, it must be closed under the action of $U_q(\mathfrak{l}_S)$. Finally, it now follows from the discussion at the end of \textsection \ref{subsection:remarksQHTS}, that the  elements $e_{\gamma}$, for $\gamma \in \overline{\Delta^+_S}$, for a basis of $V^{(0,1)}_{\alpha_r^c}$ in $\Lambda^{(0,1)}_q$.  
\end{proof}

\begin{eg}
For the case of $A_3$, we have three simple roots $\{\alpha_1, \, \alpha_2, \, \alpha_3\}$ to choose from. For the choice $S = \{\alpha_2, \alpha_3\}$, we get the quantum projective $3$-space $\OO_q(\mathbb{C}\mathrm{P}^3)$, whose tangent space is given explicitly by
$$
T_{\alpha_1^c}^{(0,1)} = \Big\{E_3E_2E_1, \, E_2E_1, \, E_1\Big\} \subseteq \OO_q(\mathbb{C}\mathrm{P}^3)^{\circ},
$$
where we have used the fact that the tangent space is contained in the dual of $\OO_q(\mathbb{C}\mathrm{P}^3)$, to simplify the Lusztig root vectors. We can represent this pictorially as
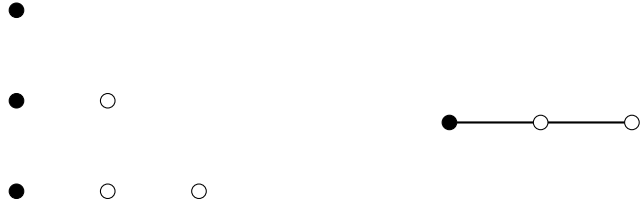
\begin{align*}
\begin{tikzpicture}[scale=1.2]
  \filldraw (0,0) circle (0.08) node[below=0.15cm] {};
  \draw (1,0) circle (0.08) node[below=0.15cm] {};
  \draw (2,0) circle (0.08) node[below=0.15cm] {};
  \filldraw (0,1) circle (0.08) node[left=0.15cm] {};
  \draw (1,1) circle (0.08) node[right=0.15cm] {};
  \filldraw (0,2) circle (0.08) node[above=0.15cm] {};
\end{tikzpicture}
& & 
\begin{tikzpicture}[scale=1.2]
  \draw[thick] (0.08,1) -- (0.92,1);
    \draw[thick] (1.08,1) -- (1.92,1);
  \filldraw (0,1) circle (0.08) node[below] {};
  \draw (1,1) circle (0.08) node[below] {};
  \draw (2,1) circle (0.08) node[below] {};
  \draw (0,0) node {};
\end{tikzpicture}
\end{align*}
where the first diagram has coloured nodes when the corresponding basis element of $\Lambda^{(0,1)}_q$ is contained in the cotangent space of $\OO_q(\mathbb{C}\mathrm{P}^3)$, and the  second diagram is the coloured Dynkin diagram corresponding to the choice of simple crossed root $S = \{\alpha_1\}$. The choice $S = \{\alpha_1, \alpha_2\}$ gives another copy of $\OO_q(\mathbb{C}\mathrm{P}^3)$, and it admits a similar description. Finally, the choice $S = \{\alpha_1, \alpha_3\}$ gives the first quantum Grassmannian which is not a quantum projective space, that is to say, $\OO_q(\mathrm{Gr}_{4,2})$. Its  tangent space is given explicitly by 
\begin{align*}
T_{\alpha_2^c}^{(0,1)} := \Bigg\{ \begin{array}{ll}
E_1E_2, & E_2, \\
E_3E_1E_2, & E_3E_2
\end{array}\Big\} \subseteq \OO_q(\mathrm{Gr}_{4,2})^{\circ}.
\end{align*}
We can represent this pictorially as
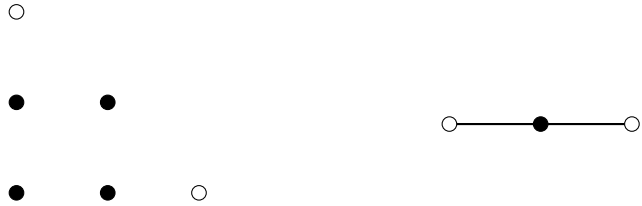
\begin{align*}
\begin{tikzpicture}[scale=1.2]
  \filldraw (0,0) circle (0.08) node[below=0.15cm] {};
  \filldraw (1,0) circle (0.08) node[below=0.15cm] {};
  \draw (2,0) circle (0.08) node[below=0.15cm] {};
  \filldraw (0,1) circle (0.08) node[left=0.15cm] {};
  \filldraw (1,1) circle (0.08) node[right=0.15cm] {};
  \draw (0,2) circle (0.08) node[above=0.15cm] {};
\end{tikzpicture}
& & 
\begin{tikzpicture}[scale=1.2]
  \draw[thick] (0.08,1) -- (1,1);
    \draw[thick] (1.08,1) -- (1.92,1);
  \draw (0,1) circle (0.08) node[below] {};
  \filldraw (1,1) circle (0.08) node[below] {};
  \draw (2,1) circle (0.08) node[below] {};
  \draw (0,0) node {};
\end{tikzpicture}
\end{align*}
where  again the first diagram has coloured nodes when the corresponding basis element is contained in the cotangent space of $\OO_q(\mathrm{Gr}_{4,2})$, and the second diagram is the crossed Dynkin diagram corresponding to the choice of simple crossed root.
\end{eg}

\begin{eg} \label{eg:counterFODC}
Let $\alpha_r \in \Pi$, and let $\OO_q(\mathrm{Gr}_{n+1,r})$ be the associated quantum Grassmannian. We note that the subspace $T:= \mathbb{C}E_r$ is a tangent space for $\mathcal{O}_q(\mathrm{SU}_{n+1})$. The restriction of the associated fodc $\Gamma$ to $\OO_q(\mathrm{Gr}_{n+1,r})$ has as its tangent space
$$
T^{(0,1)}_{\alpha_r^c} = \mathrm{ad}(U_q(\mathfrak{l}_S))E_r,
$$
which is to say, $\Gamma$ restricts to the Heckenberger--Kolb anti-holomorphic fodc. In particular, we see that for $n > 2$, 
$$
\mathrm{dim}(T^{(0,1)}_{\alpha_r^c} ) > \mathrm{dim}(T),
$$
meaning that we cannot have an embedding of cotangent spaces.
\end{eg}

\subsection{Some General Lemmas} \label{section:generalHigherEmbeddings}

In this subsection we make some general observations so as to help elucidate the delicate issue of embedding higher forms on a quantum homogeneous space $B \subseteq A$ into higher forms on the Hopf algebra $A$ itself. We begin with a simple technical observation that we find instructive to highlight as a lemma.

\begin{lem} \label{lem:CATMixUp}
Take an object $\mathcal{F} \in {}^A_B\mathrm{Mod}$, and an object $\mathcal{M} \in {}^A_A\mathrm{Mod}$, which we consider as an object in ${}^A_B\mathrm{Mod}$ with respect to the obvious forgetful functor from ${}^A_A\mathrm{Mod}$ to ${}^A_B\mathrm{Mod}$. For a morphism $\xi: \mathcal{F} \to \mathcal{M}$  in the category ${}^A_B\mathrm{Mod}$, which we also assume to be an algebra map, we consider the linear map 
\begin{align*}
\widehat{\xi}: \Phi(\mathcal{F}) \to F(\mathcal{M}), & & [f] \mapsto [\xi(f)].
\end{align*}
If $\widehat{\xi}$ is an embedding then $\xi$ is also an embedding.
\end{lem}
\begin{proof}
Since $\xi$ is by assumption a left $B$-module map we have that 
$$
\xi(B^+\mathcal{F}) \subseteq B^+\xi(\mathcal{F}) \subseteq A^+\mathrm{M}.  
$$
Thus we see that $\widehat{\xi}$ is well defined. Commutativity of the diagram of $B$-module maps
\begin{align*}
\xymatrix{ 
\mathcal{F}  \ar[rrrr]^{\unit_B} \ar_{\xi}[d] & & & & \ar[d]^{\id \otimes \widehat{\xi}}  A \square_{\pi_B} \Phi(\mathcal{F}) \\       
\mathcal{M}  \ar[rrrr]_{\unit_A ~~~ }  & & & & A \otimes F(\mathcal{M}),
}
\end{align*}
follows directly from the formulae for $\unit_A$ and $\unit_B$. Thus we see that if $\widehat{\xi}$ is injective, then  $\xi$ is also injective.
\end{proof}

We note that the converse to this result is not true. Explicitly, if $\xi$ is an embedding then it is not necessarily true that $\widehat{\xi}$ is an embedding. An example  is given by Example \ref{eg:counterFODC}.

\begin{lem} \label{lem:theMIXUPembedding}
Let $B \subseteq A$ be a quantum homogeneous space, $\Omega^{\bullet}(A)$ a dc over $A$, and $\Omega^{\bullet}(B)$ a dc over $B$. Let $p:\Omega^{\bullet}(B) \to \Omega^{\bullet}(A)$ be a morphism in ${}^A_B\mathrm{Mod}$ that is also an algebra map, and assume that $A p(\Omega^1(B))$ is a right $A$-submodule of $\Omega^{1}(A)$.  Then the image of the function 
\begin{align*}
\widehat{p}: \Phi(\Omega^{\bullet}(B)) \to F(\Omega^{\bullet}(A)),
\end{align*}
is a subalgebra of $ F(\Omega^{\bullet}(A))$. 
\end{lem}
\begin{proof}
We have assumed that $Ap(\Omega^1(B))$ is a right $A$-submodule of $\Omega^1(A)$. Let us also assume that  $Ap(\Omega^k(B))$ is a right $A$-submodule of $\Omega^1(A)$. Then 
\begin{align*}
Ap(\Omega^{k+1}(B))A = &  Ap(\Omega^{k}(B)) \wedge p(\Omega^1(B))A \\
 \subseteq  &  Ap(\Omega^{k}(B))A  \wedge p(\Omega^1(B)) \\
 \subseteq &  Ap(\Omega^{k}(B))  \wedge p(\Omega^1(B))\\
  = &  A p(\Omega^{k+1}(B)).
\end{align*}
Thus it follows from an inductive argument that $Ap(\Omega^{\bullet}(B))$ is a right $A$-submodule of the dc $\Omega^{\bullet}(A)$. 

Let us next note that, for any $\omega \in \Omega^{\bullet}(B)$, and any $a \in A$, it holds that 
$$
[ap(\omega)] = \e(a)[p(\omega)] \in \mathrm{im}(\widehat{p}). 
$$
Thus, we see that $F(Ap(\Omega^{\bullet}(B)))$ is equal to the image of $\widehat{p}$.

Since $\Omega^{\bullet}(A)$ is an algebra object in the the category ${}^A_A\mathrm{Mod}_A$, its image under $F$ is an algebra object in $\mathrm{Mod}_A$. Explicitly, the multiplication $m$ is given by 
\begin{align} \label{eqn:monoidalmult}
m\big([\omega] \otimes [\nu]\big) = [\omega S(\nu_{(-1)}) \wedge \nu_{(0)}],
\end{align}
where we have used the monoidal structure map given in \eqref{eqn:generalINVERSEMonMult}. We claim that the image of $\widehat{p}$ is a subalgebra of $F\big(\Omega^{\bullet}(A)\big)$. To show this consider the case where $\omega, \nu \in \Omega^{\bullet}(B)$. Since $Ap(\Omega^{\bullet}(B))$ is a right $A$-submodule of $\Omega^{\bullet}(A)$, we have that 
$$
p(\omega) S(\nu_{(-1)}) \wedge p(\nu_{(0)}) \in A\Omega^{\bullet}(B).
$$
Hence the corresponding coset in $F(\Omega^{\bullet}(A))$ is contained in the subspace $F(A\Omega^{\bullet}(B))$, which is to say, it is contained in the image of $\widehat{p}$. Thus we have a subalgebra as claimed.
\end{proof}

\subsection{Embedding the Higher-Order Forms}

As we saw in \textsection \ref{subsection:FOLusztigHK}, the restriction of our fodc over $\OO_q(\mathrm{SU}_{n+1})$ to the quntum Grassmannians coincides with the Heckenberger--Kolb fodc. We will now extend this result to higher forms. However, it is important to note that in general, an inclusion of fodc does not imply an inclusion of higher-order forms. This is illustrated in Example \ref{eg:counterMP} below. 

\begin{lem} \label{lem:RightAmodule}
The subspace $\OO_q(\mathrm{SU}_{n+1})\Omega^{(0,1)}_q(\mathrm{Gr}_{n+1,r})$ is a right $\OO_q(\mathrm{SU}_{n+1})$-submodule of $\Omega^{(0,1)}_q(\mathrm{SU}_{n+1})$.
\end{lem}
\begin{proof}
It follows from the description of the right $\OO_q(\mathrm{SU}_{n+1})$-module structure of $\Lambda^{(0,1)}_q$ given in \eqref{eqn:quantumaction}, and the description of the image of $V^{(0,1)}_{\alpha_r^c}$ in $\Lambda^{(0,1)}_q$ given in Proposition \ref{prop:HKFOEmbedding}, that $V^{(0,1)}_{\alpha_r^c}$ is a right $\OO_q(\mathrm{SU}_{n+1})$-submodule of  $\Lambda^{(0,1)}_q$. It follows from the discussions of \cite[\textsection 1]{Tak} (see also \cite[Lemma 2.5]{MMF2}) that
\begin{align*}
\OO_q(\mathrm{SU}_{n+1})\Omega^{(0,1)}_q(\mathrm{Gr}_{n+1,r}) \cong \OO_q(\mathrm{SU}_{n+1}) \otimes V^{(0,1)}_{\alpha_r^c}.
\end{align*}
Thus we see that we have a right sub-module.
\end{proof}

\begin{prop} \label{prop:HKEmbedding}
The unique dga map $p$ from $\Omega^{(0,\bullet)}_q(\mathrm{Gr}_{n+1,r})$ to $\Omega^{(0,\bullet)}_q(\mathrm{SU}_{n+1})$ induces an isomorphism between the Lusztig dc and the Heckenberger--Kolb dc of the quantum Grassmannians. 
\end{prop}
\begin{proof}
It follows from Lemma \ref{lem:theMIXUPembedding}, and Lemma \ref{lem:RightAmodule} above, that the image of $\widehat{p}$ is a subalgebra of $F\big(\Omega^{(0,\bullet)}_q(\mathrm{SU}_{n+1})\big)$. It was shown in Proposition \ref{prop:HKFOEmbedding} that, for all $\gamma \in \overline{\Delta^+_S}$, the basis element $e_{\gamma}$ is in the image of $\widehat{p}$. Hence the image of $\widehat{p}$ contains all basis elements of the form
\begin{align*}
\Big\{ e_{\gamma_1} \wedge \cdots \wedge e_{\gamma_k} \,|\, \gamma_1 < \cdots < \gamma_k \in \overline{\Delta^+_S} \Big\}. 
\end{align*}
This means that the image of $\widehat{p}$ has dimension at least as large as the quantum exterior algebra of the Heckenberger--Kolb dc, meaning that $\widehat{p}$ must be an embedding. Lemma \ref{lem:CATMixUp} now implies that the dga map $p$ from the Heckenberger--Kolb dc to the Lusztig dc is an embedding. Since the Heckenberger--Kolb dc is the maximal prolongation of its space of $1$-forms, $p$ must also be surjective, and so, it is an isomorphism.
\end{proof}

We finish with an example that demonstrates why one needs to be careful when considering maximal prolongations, even in the case of covariant fodc over quantum homogeneous spaces.

\begin{eg} \label{eg:counterMP}
For the special case of $\OO_q(\mathrm{SU}_2)$, we have a single non-empty subset $S$ of $\Pi$, namely $\Pi = \{\alpha_1\}$ itself.  The associated quantum flag manifold is the celebrated \emph{quantum projective line} $\OO_q(\mathbb{CP}^1)$, also known as the \emph{Podle\'s sphere} and denoted by $\OO_q(S^2)$. (In the literature $\OO_q(S^2)$ is often called the \emph{standard Podle\'s sphere} to  distinguish it from the other members of the one parameter family Podle\'s $q$-spheres, see \cite[\textsection 4.5]{KSLeabh} for details.) The Podle\'s sphere is the simplest example of a quantum Grassmannian, and its Heckenberger--Kolb dc is the \emph{Podle\'s dc}, originally discovered in \cite{PodlesCalc}, where it was classified with respect to  a natural set of requirements. Indeed, the Heckenberger--Kolb construction and  classification theorem can be considered as a vast generallisation of the work of Podle\'s.

Consider now the following subspace of $U_q(\mathfrak{sl}_2)$:
\begin{align*}
T := \mathrm{span}_{\mathbb{C}}\{FK,\, E\}.
\end{align*}
It follows directly from the coproduct formulae of $U_q(\mathfrak{sl}_2)$ that $T$ is a tangent space for $\OO_q(\mathrm{SU}_2)$, and we denote the associated fodc by $\Gamma^1$. Moreover, we see that $T$ is closed under the left action of $U_q(\mathfrak{h})$, meaning that the tangent space of the restriction of $\Gamma^1$ to the Podle\'s sphere $\OO_q(S^2)$ is spanned by the cosets of $E$ and $FK$ in $\OO_q(S^2)^{\circ}$. The associated fodc is of course the Podle\'s fodc. Let us denote by $\{f,e\}$ the dual basis of the Podle\'s cotangent space $V^1$. Then it holds that 
$$
\Omega^{2}_q(S^2) = \OO_q(S^2) f \wedge e,
$$
and that all other higher forms are zero, as explained for example in \cite[\textsection 1]{Maj}.

Let us now consider $\Gamma^\bullet$ the maximal prolongation of $\Gamma^1$. We note that the  ideal $I \subseteq \OO_q(\mathrm{SU}_2)^+$ corresponding to $T$ contains the elements $u^+_{11}$ and $u_{22}^+$. Identifying the cotangent space of $\Gamma^1$ with $V^1$ through the canonical embedding, we have the associated relations
\begin{align*}
\omega(u^+_{11}) = [u_{12}] \otimes [u_{21}] = qf \otimes e, & & \omega(u^+_{22}) = [u_{21}] \otimes [u_{12}] = qe  \otimes f.
\end{align*}
It follows that $\Gamma^2 = 0$ and hence that the image of $\Omega^{2}_q(S^2)$ in $\Gamma^\bullet$ is zero. 
\end{eg}


\section{The Lusztig DC for  the Full Quantum Flag Manifolds}

In this section we look at the Lusztig dc for the full quantum flag manifolds, and we see that it is a direct $q$-deformation of the anti-holomorphic subcomplex of the classical $A$-series full quantum flag manifold. We present the dc as a direct sum of line modules over $\OO_q(\mathrm{F}_{n+1})$. Notably, we show that the right $\OO_q(\mathrm{F}_{n+1})$-module structure is not given by the Hopf algebra counit, as is the case for the Heckenberger--Kolb calculi. Thus it lives outside the subcategory of those relative Hopf modules for which the simple monoidal version of Takeuchi's equivalence applies.

\subsection{Some Details on Full Quantum Flag Manifolds} \label{section:QFM.Intro}

We begin by recalling some well-known results about the $A$-series full quantum flag manifolds. This is mainly to fix notation, and  we refer the reader to \cite[\textsection 5]{GAPP} for a detailed exposition for this material in the notation of this paper.

The quantum flag manifold $\OO_q(\mathrm{F}_{n+1})$ associated to the choice of simple nodes $S = \varnothing$ is called the \emph{full quantum flag manifold} of $\OO_q(\mathrm{SU}_{n+1})$. In this case the isotropy Hopf algebra is isomorphic to $\OO(\mathbb{T}^{n})$, which is to say the group Hopf algebra of the weight lattice $\mathcal{P} \cong  \mathbb{Z}^{n}$. Thus the simple objects in ${}^{A}_{B}\mathrm{Mod}_0$ are labelled by  elements $\lambda$ of the weight lattice $\mathcal{P}$. As explained in \cite[Example 5.5]{GAPP}, this means that we have a $\mathcal{P}$-algebra grading
$$
\OO_q(\mathrm{SU}_{n+1}) \cong \bigoplus_{\lambda \in \mathcal{P}} \EE_{\lambda}.
$$
Moreover, the grading is \emph{strong}, which is to say, 
\begin{align*}
\EE_{\lambda} \otimes_{\OO_q(\mathrm{F}_{n+1})} \EE_{\mu} \cong \EE_{\lambda}\EE_{\mu} = \EE_{\lambda + \mu}, & & \textrm{ for all } \lambda, \, \mu \in \mathcal{P},
\end{align*}
where the first isomorphism is given by multiplication. Thus we see that $\EE_{\lambda}$ is an invertible bimodule with inverse $\EE^{-1}_{\lambda} = \EE_{-\lambda}$. Motivated by classical geometry, we follow the terminology of \cite{BeggsMajid:Leabh} and call an invertible bimodule in  ${}^{A}_{B}\mathrm{Mod}_0$  a \emph{relative line module}, or simple a \emph{line module}.

As shown in \cite[\textsection Theorem 4.1]{Stok}, a set of algebra generators for $\OO_q(\mathrm{SU}_{n+1})$ is given by 
\begin{align} \label{eqn:QPHSGens}
z^{\varpi_x}_{i} := c^{\varpi_x}_{f_i,v_{N_x}}, & & \overline{z}_i^{\,\varpi_x} := c^{-w_0(\varpi_x)}_{v_i,f_{N_x}} & & \textrm{ for } i = 1, \dots, N_x, \textrm{ and } x = 1, \dots, n,
\end{align}
where $N_x :=  \dim(V_{\varpi_x})$, and $\{v_i\}_{i=1}^{N_{x}}$, and $\{f_i\}_{i=1}^{N_{x}}$, are a choice of dual weight bases of $V_{\varpi_x}$, and $V_{-w_0(\varpi_x)}$ respectively, such that $v_{N_x}$ is a highest weight vector of $V_{\varpi_x}$.  
For any $\lambda = \sum_{i=1}^{n} a_i \varpi_i \in \mathcal{P}^+$, we also find it convenient to introduce the notation 
\begin{align*}
z^{\lambda} := (z^{\varpi_1}_{N_1})^{a_1} \cdots \, (z^{\varpi_{n}}_{N_n})^{a_{n}}.
\end{align*}
Finally, we observe that the $\mathcal{P}$-grading of $\OO_q(\mathrm{SU}_{n+1})$ is determined by $|z^{\varpi_x}_i| = \varpi_x$, and $|\overline{z}^{\,\varpi_x}_i| = -\varpi_x$, for all $i = 1, \dots, N_x$. Explicitly, for $\lambda \in \mathcal{P}^+$, the line module $\EE_{\lambda}$  is generated as a left $\OO_q(\mathrm{F}_{n+1})$-module by the subspace $z^{\lambda}U_q(\frak{sl}_{n+1})$.


\subsection{A Higher Order Extension of Majid's Framing Theorem}

As we saw from Example \ref{eg:counterFODC}, the contangent space of a covariant first-order differential calculus over a quantum homogeneous space $B \subseteq A$ does not necessarily embed into the cotangent space of the Hopf algebra $A$. Necessary and sufficient conditions for this to happen were discussed at the end of \textsection \ref{subsection:remarksQHTS}. It was proved by Majid in \cite[Theorem 2.1]{Maj} that right $\pi_B(A)$-covariance of the fodc on $\Omega^1(A)$  guarantees an embedding. While this requirement is not necessary, it is quite natural from a categorical point of view, and it is this version that we now extend to higher forms.

Consider the category ${}^A_A\mathrm{Mod}^{\pi_B}$ whose objects are left $A$-Hopf modules $\mathcal{M}$ endowed with a right \mbox{$\pi_B(A)$-comodule} structure $\Delta_R:\mathcal{M} \to \mathcal{M} \otimes \pi_B(A)$ giving $\mathcal{M}$ the structure of an \mbox{$A$-$\pi_B(A)$-bicomodule}, and for which $\Delta_R(am) = \Delta_{R,\pi_B}(a)\Delta_R(m)$, for all $a \in A, m \in \mathcal{M}$. 

As observed in \cite{GAPP}, we can extend the fundamental theorem of Hopf modules to an equivalence between ${}^A_A\mathrm{Mod}^{\pi_B}$ and $\mathrm{Mod}^{\pi_B}$, the category of right $\pi_B(A)$-comodules: For $\mathcal{M} \in {}^A_A\mathrm{Mod}^{\pi_B}$, we endow $F(\mathcal{M})$ with the right $\pi_B(A)$-coaction $[m] \mapsto [m_{(0)}] \otimes m_{(1)}$. In the other direction, for some $(V,\Delta_R) \in \mathrm{Mod}^{\pi_B}$, we endow $A \otimes V$ with the right $\pi_B(A)$-comodule tensor product $\Delta_{R,\pi_B} \otimes \Delta_R$.

We also have an equivalence between ${}^A_A\mathrm{Mod}^{\pi_B}$ and the category of relative Hopf modules ${}^A_B\mathrm{Mod}$. This was observed in \cite[Proposition 3.4]{GAPP} (in the more general setting of principal pairs of quantum homogeneous spaces). This equivalence builds on Takeuchi's equivalence for left $B$-modules \cite[Theorem 2]{Tak}, which was later generalised to Schneider's equivalence for principal comodule algebras \cite[ Theorem I]{Schneider90}. First we define a functor 
\begin{align*}
A \otimes_B - : {}^A_B\mathrm{Mod} \to {}^A_A\mathrm{Mod}^{\pi_B},  & &  \FF \mapsto A \otimes_B \FF,
\end{align*}
where the left $A$-module structure of $A \otimes_B \FF$ is given by the first tensor factor, the left $A$-comodule structure is the tensor product coaction $\Delta \otimes \Delta_L$, and the right $\pi_B(A)$-comodule structure is given by $\Delta_{R,\pi_B} \otimes \mathrm{id}$. The functor acts on morphisms as $\id  \otimes -$. In the other direction, we have the functor $\mathrm{co}(-)$ which maps $\mathcal{M}$ to its space of coinvariants 
$$
\mathcal{M}^{\co(\pi_B)} := \big\{ m \in \mathcal{M} \, | \, \Delta_R(m) = m \otimes 1 \big\}, 
$$
and restricts the domain of morphisms.

Finally, just as for any Hopf algebra, the antipode $S$ of $\pi_B(A)$ gives an equivalence 
\begin{align*}
\mathrm{S}: \mathrm{Mod}^{\pi_B} \to {}^{\pi_B}\mathrm{Mod}, & & \mathrm{S^{-1}}: {}^{\pi_B}\mathrm{Mod}\to \mathrm{Mod}^{\pi_B}. 
\end{align*}

Combing these equivalences with Takeuchi's relative Hopf module equivalence gives us the following diagram in the category of all categories:
\[
\xymatrix{
{}^A_A\mathrm{Mod}^{\pi_B}  \ar@<-0.5ex>[rrrr]_{F} \ar@<-0.5ex>[d]_{\mathrm{co}(-)} & & & & \ar@<-0.5ex>[llll]_{A \otimes -} \ar@<0.5ex>[d]^{\mathrm{S}^{-1}}  \mathrm{Mod}^{\pi_B} \\
{}^A_B\mathrm{Mod}  \ar@<-0.5ex>[u]_{A \otimes_B -} \ar@<-0.5ex>[rrrr]_{\Phi} & & &  &  {}^{\pi_B}\mathrm{Mod}  \ar@<-0.5ex>[llll]_{\Psi} \ar@<0.5ex>[u]^{\mathrm{S}}
}
\]
We will now look at some natural isomorphisms associated to this diagram. 

\begin{prop}
A natural isomorphism between the functor $A \otimes_B -$ and the functor $(A \otimes -) \circ \mathrm{S} \circ \Phi$ is given by the components  
\begin{align} \label{eqn:NatTrans1}
\alpha_{\F}: A \otimes_B \F \to A \otimes \mathrm{S}(\Phi(\F)), & & a \otimes f \mapsto af_{(-1)} \otimes [f_{(0)}],
\end{align}
where $\mathcal{F}$ is an object in ${}^A_B\mathrm{Mod}$.
\end{prop} 
\begin{proof}
Recall from \eqref{eqn:theALPHAmap} that \eqref{eqn:NatTrans1} gives an isomorphism in the category ${}^A_A\mathrm{Mod}$. We will now show that this is also a morphism of right $\pi_B(A)$-comodules, which is to say, we will show that the following diagram commutes:
\begin{align*}
\xymatrix{ 
A \otimes_B \mathcal{F}  \ar[rrrr]^{\Delta_{R,\pi_B} \otimes  \, \mathrm{id}} \ar[d]_{\alpha_{\FF}} & & & & \ar[d]^{\alpha_{\FF} \otimes \mathrm{id}}  A \otimes_B \mathcal{F} \otimes \pi_B(A) \\       
A \otimes \mathrm{S}(\Phi(\mathcal{F}))  \ar[rrrr]_{\Delta_{R,\pi_B} \otimes \Delta_{R}}  & & & & A \otimes \mathrm{S}(\Phi(\mathcal{F})) \otimes \pi_B(A)
}
\end{align*}
This follows from the calculation 
\begin{align*}
(\Delta_{R,\pi_B} \otimes \Delta_R)(\alpha_{\FF}(a \otimes f)) = & \, \big(\Delta_{R,\pi_B} \otimes  \, \Delta_R)(af_{(-1)} \otimes [f_{(0)}]\big) \\
= & \, a_{(1)}f_{(-3)} \otimes [f_{(0)}] \otimes \pi_B\big(a_{(2)} f_{(-2)} S(f_{(-1)})\big)\\
= & \, a_{(1)} f_{(-1)} \otimes [f_{(0)}] \otimes \pi_B(a_{(2)})\\
= & \, (\alpha_{\FF} \otimes \mathrm{id}) \circ (\Delta_{R,\pi_B} \otimes \mathrm{id})(a \otimes f).
\end{align*}
It is now readily confirmed that the family of isomorphisms  
$$
\Big\{ \alpha_{\FF} \,|\, \FF \in {}^A_B\mathrm{Mod}^{\pi_B}\Big\}
$$ 
give a natural isomorphism. 
\end{proof}

From this follows a number of other natural isomorphisms. We highlight the following which will be a key ingredient in the higher order extension of Majid's framing theorem below. We omit the proof, which is clear.

\begin{cor} \label{cor:thesecondNatIso}
A natural isomorphism between the two functors $F \circ (A \otimes_B - )$ and $\mathrm{S} \circ \Phi$ is given by the components 
\begin{align*}
\beta_{\F}: F(A \otimes_B \F) \to \mathrm{S}(\Phi(\F)), & & [a \otimes f] \mapsto \e(a)[f],
\end{align*}
where $\mathcal{F} \in {}^A_B\mathrm{Mod}$.
\end{cor}

In the following corollary, we now apply this natural isomorphism to the question of embedding higher differential forms.

\begin{cor} \label{cor:CovariantHiggherEmbedding}
Let $B \subseteq A$ be a quantum homogeneous space, $\Omega^{\bullet}(A)$ a dc in the category ${}^A_A\mathrm{Mod}^{\pi_B}$, and $\Omega^{\bullet}(B) \in {}^A_B\mathrm{Mod}$ the restriction to a dc on $B$. Then the map
\begin{align*}
i: \mathrm{S} \circ \Phi(\Omega^{\bullet}(B)) \to F(\Omega^{\bullet}(A)), & & [\omega] \mapsto [\omega]
\end{align*}
is a monomorphism in the category $\mathrm{Mod}^{\pi_B}$. Hence, we have that 
\begin{align*}
\Omega^{\bullet}(B) \cong A \, \square_{\pi_B} i(\mathrm{S} \circ \Phi(\Omega^{\bullet}(B))).
\end{align*}
\end{cor}
\begin{proof}
Note first that $\Omega^{\bullet}(B)$ is contained in the right coinvariant forms of $\Omega^{\bullet}(A)$, giving us  an embedding 
$$
\Omega^{\bullet}(B) \hookrightarrow (\Omega^{\bullet}(A))^{\co(\pi_B)}.
$$
Operating on this inclusion by the functor $A \otimes_B -$ gives us the inclusion
\begin{align} \label{eqn:firstinclusion}
A \otimes_B \Omega^{\bullet}(B) \hookrightarrow A \otimes_B (\Omega^{\bullet}(A))^{\co(\pi_B)}.
\end{align}
Considering next the inverse of the unit $\unit$ of the equivalence between ${}^A_A\mathrm{Mod}^{\pi_B}$ and ${}^A_B\mathrm{Mod}$, we see that
\begin{align*}
\unit^{-1}:A \otimes_B (\Omega^{\bullet}(A)^{\co(\pi_B)}) \cong \Omega^{\bullet}(A), & & a \otimes \omega \mapsto a\omega.
\end{align*}
Combining this with the inclusion in  \eqref{eqn:firstinclusion}, we see that
\begin{align*}
A \otimes_B \Omega^{\bullet}(B) \hookrightarrow \Omega^{\bullet}(A), & & a \otimes \omega \mapsto a\omega.
\end{align*} 
Operating on this by the functor $F$, we get the inclusion 
\begin{align*}
F(A \otimes_B \Omega^{\bullet}(B)) \hookrightarrow F(\Omega^{\bullet}(A)).
\end{align*}
Finally, combining this with the appropriate component of the natural isomorphism given in Corollary \ref{cor:thesecondNatIso}, we get the claimed inclusion 
\end{proof}


\subsection{The Lusztig Differential Calculus of $\OO_q(\mathrm{F}_{n+1})$}

In this subsection we look at the restriction of $\Omega^{(0,1)}_q(\mathrm{SU}_{n+1})$ to the $A$-series full quantum flag manifolds, observing that they have classical dimension. Moreover, we also describe the maximal prolongation of this fodc as a direct sum of relative line modules.

\begin{lem} \label{lem:higherfullflagforms}
The embedding 
$
V^{(0,\bullet)}_{\varnothing} \hookrightarrow \Lambda^{(0,\bullet)}_q
$
is an isomorphism of $\OO_q(\mathbb{T}^n)$-comodules. Hence the dc $\Omega^{(0,\bullet)}_q(\mathrm{F}_{n+1})$ has classical dimension and it holds that 
\begin{align} \label{eqn:higherfullflagforms}
\Omega^{(0,\bullet)}_q(\mathrm{F}_{n+1}) ~ = ~ \OO_q(\mathrm{SU}_{n+1}) \square_{\,\pi_B} \Lambda^{(0,\bullet)}_q ~ \cong ~ \Omega^{(0,\bullet)}_q(\mathrm{SU}_{n+1})^{\mathrm{co}(\OO(\mathbb{T}^{n})}.
\end{align}
Moreover, by transference of structure, this isomorphism endows $V^{(0,\bullet)}_{\varnothing}$ with a right $\OO_q(SU_{n+1})$-module structure.
\end{lem}
\begin{proof}
We recall from Corollary \ref{cor:FOBasisElements} that the fodc $\Omega^{(0,1)}_q(\mathrm{SU}_{n+1})$ is right $\OO_q(\mathbb{T}^n)$-covariant. Thus it follows from Corollary \ref{cor:CovariantHiggherEmbedding} that we have an embedding in the category $\mathrm{Mod}^{\pi_B}$
\begin{align} \label{eqn:fullembeddingofForms}
V^{(0,\bullet)}_{\varnothing} \hookrightarrow \Lambda^{(0,\bullet)}_q,
\end{align}
which is to say, an embedding of $\OO(\mathbb{T}^n)$-modules. It follows from the description of the Heckenberger--Kolb tangent space given in \textsection \ref{subsection:FOLusztigHK} that every basis element of $T^{(0,1)}$ pairs non-trivially with some element of some quantum Grassmannian. Now every quantum Grassmannian $\OO_q(\mathrm{Gr}_{n+1,r})$ is contained as a subalgebra of $\OO_q(\mathrm{F}_{n+1})$. Thus we see that the embedding of $V^{(0,1)}_{\varnothing}$ in $\Lambda^{(0,1)}_q$ is in fact an isomorphism, and hence that the dc $\Omega^{(0,\bullet)}_q(\mathrm{F}_{n+1})$ has classical dimension. Moreover, following the argument of Proposition \ref{prop:HKEmbedding}, we can show that the image of $V^{(0,\bullet)}_{\varnothing}$  is a subalgebra of $\Lambda^{(0,\bullet)}_q$. Thus we see that \eqref{eqn:fullembeddingofForms} is in fact an isomorphism. The identities in \eqref{eqn:higherfullflagforms} now follows directly, as does the induced $\OO_q(SU_{n+1})$-module structure on $V^{(0,\bullet)}_{\varnothing}$.
\end{proof}

For the case of full quantum flag manifolds, we recall that the isotropy Hopf algebra is $\mathcal{O}(\mathbb{T}^n)$. This means that any finite-dimensional comodule of the isotropy Hopf algebra is one-dimensional, and hence invertible. Equivalently, every relative Hopf module $\mathcal{F} \in {}^A_B\mathrm{Mod}_0$ decomposes into a direct sum of line modules. The following proposition describes the decompositions of our space of $k$-forms into line modules.

\begin{prop} \label{prop:fodcEE}
We have an isomorphism, in the category ${}^A_B\mathrm{Mod}$,
\begin{align} \label{eqn:formlineIso}
\Omega^{(0,k)}_q(\mathrm{F}_{n+1})  \cong \bigoplus_{\gamma_1  < \cdots < \gamma_k} \EE_{\gamma_1 + \cdots + \gamma_k}, 
\end{align}
where summation is over ordered $k$-tuples of positive roots of $\mathfrak{sl}_{n+1}$. 
\end{prop}
\begin{proof}
It follows from the fact that $\Lambda^{(0,\bullet)}_q$ is an $\OO(\mathbb{T}^n)$-comodule algebra, that every basis element is a weight vector. Taken together with Corollary \ref{cor:FOBasisElements}, we now see that 
\begin{align*}
|e_{\gamma_1} \wedge \cdots \wedge e_{\gamma_k}| = \gamma_1 + \cdots + \gamma_k.
\end{align*}
It now follows that 
\begin{align} 
\Omega^{(0,k)}_q(\mathrm{F}_{n+1})  \cong  & \, \bigoplus_{\gamma_1  < \cdots < \gamma_k} \OO_q(\mathrm{SU}_{n+1}) \square_{\pi_B} \mathbb{C} e_{\gamma_1} \wedge \cdots \wedge e_{\gamma_k} \\
 \cong & \bigoplus_{\gamma_1  < \cdots < \gamma_k} \EE_{\gamma_1 + \cdots + \gamma_k}, 
\end{align}
where summation is over ordered $k$-tuples of positive roots of $\mathfrak{sl}_{n+1}$.
\end{proof}

As we will see in the next subsection, the Lusztig dc does not live in subcategory ${}^A_B\mathrm{Mod}_0$. This means that  we do not have access to the convenient monoidal form of Takeuchi's equivalence, as treated in \cite{MMF2}. This means that in order to describe forms twisted by line modules, we need some new tools. 

\begin{lem}
Let $B \subseteq A$ be a quantum homogeneous space, and $\F,\mathcal{G} \in {}^A_B\mathrm{Mod}_B$. Then an isomorphism in the category ${}^{\pi_B}\mathrm{Mod}$ is given by 
\begin{align*}
\psi:\Phi(\F \otimes_B \mathcal{G}) \to \Phi(\F) \otimes_B \mathcal{G}, & & [f \otimes g] \mapsto [f] \otimes g.
\end{align*}
\end{lem}
\begin{proof}
For $f \in \F, \, g \in \mathcal{G}$, and $b \in B$, we note that
\begin{align*}
[fb] \otimes g = [f] \otimes bg,  & & [bf] \otimes g = \e(b)[f] \otimes g.
\end{align*}
From this we can conclude that $\psi$ is well-defined. An inverse is given by 
\begin{align*}
\psi^{-1}:\Phi(\F) \otimes_B \mathcal{G} \to \Phi(\F \otimes_B \mathcal{G}), & & [f] \otimes g \mapsto [f \otimes g],
\end{align*}
which is confirmed to be well-defined in an analogous manner.
\end{proof}

\begin{lem} \label{lem:rightAMonoidalTakeuchi}
Let $A,B, \F$, and $\mathcal{G}$ be as in the lemma above, and moreover, assume that $\Phi(\F)$ has a right $A$-action extending its right $B$-action. Then an isomorphism in the category ${}^{\pi_B}\mathrm{Mod}$ is given by 
\begin{align*}
\phi: \Phi(\F \otimes_B \mathcal{G}) \to \Phi(\F) \otimes \Phi(\mathcal{G}), & & [f \otimes g] \mapsto [f]g_{(-1)} \otimes [g_{(0)}].
\end{align*}
\end{lem}
\begin{proof}
Just as in the previous lemma, it is routine to check that $\phi$ is well defined as a morphism in the category ${}^{\pi_B}\mathrm{Mod}$. Consider next the map
\begin{align*}
\phi^{-1}: \Phi(\F) \otimes \Phi(\mathcal{G}) \to \Phi(\F) \otimes_B \mathcal{G}, & & [f] \otimes [g] \mapsto [f]S(g_{(-1)}) \otimes g_{(0)}.
\end{align*}
Using the previous lemma to identify $\Phi(\F \otimes_B \mathcal{G})$ and $\Phi(\F) \otimes_B \mathcal{G}$, we see that $\phi^{-1}$ does indeed define an inverse to $\phi$, and hence we have an isomorphism.
\end{proof}

Using the isomorphism given above, together with the right $\OO_q(\mathrm{SU}_{n+1})$-module structure of $V^{(0,1)}_{\varnothing}$ given in Lemma \ref{lem:higherfullflagforms}, we can now describe the space of forms twisted by a line module as a direct sum of line modules. 

\begin{prop} \label{prop:omegaEE}
For any $\lambda \in \mathcal{P}$, we have an isomorphism, in the category ${}^A_B\mathrm{Mod}$,
\begin{align} \label{eqn:twisted:formlineIso}
\Omega^{(0,k)}_q(\mathrm{F}_{n+1}) \otimes_{\OO_q(\mathrm{F}_{n+1})} \EE_{\lambda} \cong \bigoplus_{\gamma_1  < \cdots < \gamma_k} \EE_{\gamma_1 + \cdots + \gamma_k + \lambda}, 
\end{align}
where summation is over ordered $k$-tuples of positive roots of $\mathfrak{sl}_{n+1}$.  
\end{prop}
\begin{proof}
It follows from Lemma \ref{lem:rightAMonoidalTakeuchi} that we have an isomorphism 
\begin{align*}
\Phi\!\left(\Omega^{(0,1)}_q(\mathrm{F}_{n+1}) \otimes_{\OO_q(\mathrm{F}_{n+1})} \EE_{\lambda}\right) \to \Phi(\Omega^{(0,1)}_q(\mathrm{F}_{n+1})) \otimes \Phi(\EE_{\lambda}).
\end{align*}
The claimed isomorphism in \eqref{eqn:twisted:formlineIso} now follows from Takeuchi's equivalence, Proposition \ref{prop:fodcEE} and the additivity of weights with respect to tensor products.
\end{proof}

\subsection{The Full Lusztig DC and the Sub-Category ${}^A_B\mathrm{Mod}_0$}

It follows from \cite[Proposition 3.4]{HKdR} that all the Heckenberger--Kolb anti-holomorphic fodc are contained in the subcategory ${}^A_B\mathrm{Mod}_0$. Thus in particular, the Podle\'s fodc is contained in this subcategory. However, as the following proposition shows, the Podle\'s sphere is the only full quantum flag manifold for which this happens. 

\begin{prop}
For $n \geq 2$, the fodc $\Omega^{(0,1)}_q(\mathrm{F}_{n+1})$ is not an object in ${}^A_B\mathrm{Mod}_0$.
\end{prop}
\begin{proof}
Consider the element  
$$
b := u_{11}u_{32}u_{33}u_{44} \cdots u_{n+1,n+1} \in \OO_q(\mathrm{SU}_{n+1}).
$$
Its weight is given by the sum of the weights occurring in $V_{\varpi_n}$, which is to say, it is zero. Thus we see that  $b \in \OO_q(\mathrm{F}_{n+1})$.  

Acting on $e_{21} \in V^{(0,1)}_{\varnothing}$ by $b$, we get the non-zero element
\begin{align*}
e_{21}b = e_{21}u_{11}u_{32}u_{33} \cdots u_{n+1,n+1}  
= \, &  q^{-1} \nu e_{31},
\end{align*}
where the identity $e_{21}u_{32} = \nu e_{31}$ makes sense because $n \geq 2$ by assumption. Thus we see that $b$ does not act on $e_{21}$ by $\e(b) = 0$, meaning that $V^{(0,1)}_{\varnothing}$ is not contained in the subcategory ${}^{\pi_B}\mathrm{Mod}_0$. It now follows from the subcategory equivalence between ${}^{\pi_B}\mathrm{Mod}_0$ and ${}^A_B\mathrm{Mod}_0$ that the fodc is not contained in the subcategory ${}^A_B\mathrm{Mod}_0$, for $n \geq 2$.
\end{proof}

\subsection{Some Examples and Observations}

In this subsection we treat some special cases of the results of the previous section. Moreover, we make some observations about the fact that the Lusztig dc of $\OO_q(\mathrm{F}_{n+1})$ are not contained in the  category $\mathrm{Mod}_0$.

\begin{eg}
In addition to being the simplest example of a quantum Grassmannian, the Podle\'s sphere $\OO_q(S^2)$ is also the simplest full quantum-flag manifold. As shown in \cite{HK}, and independently in \cite{Maj}, we have an $\OO_q(S^2)$-bimodule decomposition 
$$
\Omega^1_q(S^2) \cong \Gamma^{(1,0)} \oplus \Gamma^{(0,1)},
$$
generalising the decomposition of the complexified $1$-forms of $S^2$ into holomorphic and anti-holomorphic forms. As a comparison of tangent spaces shows, $\Gamma^{(0,1)}$ coincides with  $\Omega^{(0,1)}_q(S^2)$.

Recall that the weight lattice of $\mathfrak{sl}_2$ is isomorphic to $\mathbb{Z}$, meaning that we can label the line modules as $\EE_{k}$, for $k \in \mathbb{Z}$. The root lattice corresponds to the sublattice $2\mathbb{Z}$, and  we have that $\alpha_1 = 2\varpi_1$. From this we see that 
$$
\Omega^{(0,1)}_q(S^2) \cong \EE_{2}
$$
This reproduces Majid's quantum frame bundle description of the anti-holomorphic one-forms of the Podle\'s calculus \cite[Theorem 3.1]{Maj}. 
\end{eg}

\begin{eg}
It follows from the description of the module structure of $\Lambda^{(0,1)}_q$ that the decomposition in Proposition \ref{prop:fodcEE} is not a decomposition as right $\OO_q(\mathrm{F}_{n+1})$-modules. However, the \emph{column} decomposition 
\begin{align*}
\Omega^{(0,1)}_q(\mathrm{F}_{n+1}) \cong \bigoplus_{i=1}^{n} M_i, & & \textrm{ where } M_i := \bigoplus_{j=i}^{n} \EE_{\alpha_i + \, \cdots \, + \alpha_{j}},
\end{align*} 
is a decomposition in the category ${}^A_B\mathrm{Mod}_B$. We note that the only irreducible summand is $\EE_n$. Finally, it is interesting to note that 
$$
\bigoplus_{i=1}^{n} \EE_{\alpha_i \, + \cdots \, + \alpha_{n}}
$$
is a sub-bimodule of the fodc, and that it is an object in the sub-category ${}^A_B\mathrm{Mod}_0$. 
\end{eg}

\begin{eg}
For $\OO_q(\mathrm{F}_3)$, the $1$-forms are given explicitly by
\begin{align*}
\Omega^{(0,1)}_q(\mathrm{F}_3) \cong  \EE_{\alpha_1} \oplus \EE_{\alpha_2} \oplus \EE_{\alpha_1 + \alpha_2} \cong M_1 \oplus M_2.
\end{align*} 
For the degree two forms we see that 
\begin{align*}
\Omega^{(0,2)}_q(\mathrm{F}_3) \cong  \EE_{\alpha_1 + \alpha_2} \oplus  \EE_{2\alpha_1 + \alpha_2} \oplus  \EE_{\alpha_1 + 2\alpha_2}. 
\end{align*}
Finally, the space of top forms is given by the single line module
$$
\Omega^{(0,3)}_q(\mathrm{F}_3) \cong \EE_{2(\alpha_1 + \alpha_2)} = \EE_{2\rho},
$$
where $\rho$ is the half-sum of positive roots of $\mathfrak{sl}_3$. It follows from the formula for the $\OO_q(\mathrm{F}_3)$-module structure of $V^{(0,1)}_{\varnothing}$ that the space of top forms is contained in the subcategory ${}^A_B\mathrm{Mod}_0$. We note that the space of top forms is non-trivial, which is to say we do not have isomorphism $\EE_0 \cong \Omega^{(0,3)}(\mathrm{F}_3)$, or in other words, the dc is not \emph{orientable}. In the classical case, a K\"ahler manifold whose space of top anti-holomorphic forms is a trivial line module is called a \emph{Calabi--Yau manifold}. While every flag manifold is a K\"ahler manifold, {\emph no} flag manifold is a Calabi--Yau manifold.
\end{eg}

\begin{eg}
For the general $\OO_q(\mathrm{F}_{n+1})$ case, the space of top anti-holomorphic forms is given by
$$
\Omega^{(0,|\Delta^+|)}_q(\mathrm{F}_{n+1}) \cong \EE_{2\rho} \in {}^A_B\mathrm{Mod}_0.
$$
Again we see that the dc is not orientable, generalising the fact that $\mathrm{F}_{n+1}$ is not a Calabi--Yau manifold.
\end{eg}

\section{A Borel--Weil Theorem for the Full Quantum Flag Manifolds}

In this section we construct  a direct quantum group generalisation of the classical Borel--Weil theorem for the $A$-series full flag manifold. This gives, for the first time, a noncommutative differential geometric realisation of all the type-$1$ finite-dimensional irreducible representations of $U_q(\mathfrak{sl}_{n+1})$. This approach can easily be adapted to treat the quantum Grassmannians $\OO_q(\mathrm{Gr}_{n+1,m})$, whereupon it reproduces the quantum Borel--Weil theorem given in \cite{BwGrassDM} and \cite{CDOBBW}. We note that the existence of a tractable fodc $\Omega^{(0,1)}_q(\mathrm{SU}_{n+1})$ on the Hopf algebra $\OO_q(\mathrm{SU}_{n+1})$ allows for a significant simplification of these earlier proofs.

\subsection{Connections and Quantum Principal Bundles}

For $\Omega^1$ a fodc over an algebra $B$, and $\mathcal{F}$ a left $B$-module, a \emph{left  connection} for $\F$ is a $\mathbb{C}$-linear map 
$
\nabla:\mathcal{F} \to \Omega^1 \otimes_B \F
$
satisfying 
\begin{align} 
\nabla(bf) = \exd b \otimes f + b \nabla f, & & \textrm{ for all } b \in B, f \in \F.
\end{align}
The difference of two  connections $\nabla - \nabla'$ is a left $B$-module map. Any connection can be extended to a map $\nabla: \Omega^\bullet \otimes_B \mathcal{F} \to   \Omega^\bullet \otimes_B \mathcal{F}$ by the formula 
\begin{align*}
\nabla(\omega \otimes f) =   \exd \omega \otimes f + (-1)^{|\omega|} \, \omega \wedge \nabla f,
\end{align*}
where $f \in \F$, and $\omega$ is a homogeneous element of $\Omega^{\bullet}$ of degree  $|\omega|$. 
The \emph{curvature} of an connection is the left $B$-module map $\nabla^2: \mathcal{F} \to \Omega^2 \otimes_B
\mathcal{F}$. An connection is said to be {\em flat} if $\nabla^2 = 0$. Since $\nabla^2(\omega \otimes f) =
\omega \wedge \nabla^2(f)$, an connection is flat if and only if  the pair $(\Omega^\bullet \otimes_B \F,
\nabla)$ is a cochain complex. 

A \emph{left bimodule connection} on $\F$ is a pair $(\nabla,\sigma)$ where $\nabla$ is a left connection and $\sigma: \F \otimes_B \Omega^1(B) \rightarrow \Omega^1(B) \otimes_B \F $ is a bimodule map satisfying
\begin{align}\label{eqn:BiC}
 \sigma (f\otimes db)= \nabla (fb)-\nabla(f)b.
\end{align}
Note that in the commutative case $\sigma$ 
reduces to the standard flip map.

We now briefly recall Brzezinski and Majid's theory of quantum principal bundles, a framework for constructing  connections. To keep the preliminaries to a minimum, we restrict to the quantum homogeneous space case. For more details, we direct the interested reader to \cite[\textsection 5]{BeggsMajid:Leabh} or to  \cite[\textsection 2]{CDOBBW}. A \emph{homogeneous quantum principal bundle} is a pair $(B,\Omega^1(A))$ consisting of a quantum homogeneous space $B \subseteq A$, and a left $A$-covariant, right $\pi_B(A)$-covariant fodc $\Omega^1(A)$ over $A$. A \emph{principal connection} for $(B,\Omega^1(A))$ is a left $A$-module, right $\pi_B(A)$-comodule, map 
$$
\Pi: \Omega^1(A) \to \Omega^1(A),
$$
satisfying $\Pi^2 = \Pi$, and $\mathrm{ker}(\Pi) = A\Omega^1(B)A$, where $\Omega^1(B)$ is the restriction of $\Omega^1(A)$ to a fodc on $B$. We are interested in principal connections because they allow us to construct connections for any relative Hopf module. Again, to keep preliminaries to a minimum, we will restrict our attention to the subcategory ${}^A_B\mathrm{Mod}_0$. Any simple  object $\mathcal{F}$ in this subcategory is isomorphic to a subobject of $A$ (see \cite[Appendix A.3]{GAPP} for details). Now if the principal connection is \emph{strong}, that is to say, if $(\id - \Pi) \circ \exd B \subseteq \Omega^1(B)A$, then an connection for $\F \subseteq A$ is given by
\begin{align*}
\nabla: \F \to \Omega^1(B) \otimes_B  \F, & & f \mapsto  j\!\left(\sum_i (\id - \Pi)(\exd f) \right)\!,
\end{align*}
where $j: \Omega^1(B)\F  \to \Omega^1(B) \otimes_B \F$ is the inverse to the obvious  multiplication map from $\Omega^1(B) \otimes_B \F$ to $\Omega^1(B)\F$. (Faithful flatness guarantees that the multiplication map  is invertible). We call $\nabla$ the connection \emph{associated} to $\Pi$. If $\Pi$ is additionally assumed to be a left $A$-comodule map, then every associated connection will be a left $A$-comodule map.

\subsection{Holomorphic Structures and Principal Connections}

Since the Lusztig fodc on the Hopf algebra $\OO_q(\mathrm{SU}_{n+1})$ is left $\OO(\mathbb{T}^n)$-covariant we have a homogeneous quantum principal bundle. We start by showing that the zero map is a strong principal connection for the bundle.

\begin{lem}
We have the equality
\begin{align} \label{eqn:StrongId}
\Omega^{(0,1)}_q(\mathrm{SU}_{n+1}) = \OO_q(\mathrm{SU}_{n+1}) \Omega^{(0,1)}_q(\mathrm{F}_{n+1})  = \, \Omega^{(0,1)}_q(\mathrm{F}_{n+1})\OO_q(\mathrm{SU}_{n+1}).
\end{align}
Hence the zero map on $\Omega^{(0,1)}_q(\mathrm{SU}_{n+1})$ is a principal connection, and it is the unique principal connection for the bundle. 
\end{lem}
\begin{proof}
It follows from \cite[\textsection 2.3]{Tak} (see also the discussion in \cite[\textsection 2.3]{MMF2}) that
\begin{align*}
\OO_q(\mathrm{SU}_{n+1})\Omega^{(0,1)}_q(\mathrm{F}_{n+1}) = &  \,  \OO_q(\mathrm{SU}_{n+1})\Big(\OO_q(\mathrm{SU}_{n+1}) \square_{\OO(\mathbb{T}^n)} V_{\varnothing}^{(0,1)}\Big) \\
                                                = & \, \OO_q(\mathrm{SU}_{n+1}) \otimes V_{\varnothing}^{(0,1)},
\end{align*}
which gives us the first identity in \eqref{eqn:movingatotheright}.

To prove the second identity, we follow the approach of \cite[Lemma 5.7]{CDOBBW}. Note first that the Lusztig fodc is generated as a right $\OO_q(\mathrm{F}_{n+1})$-module by those elements of the form $\adel b$, for $b \in \OO_q(\mathrm{F}_{n+1})$. Thus we just need to show that, for any $a \in \OO_q(\mathrm{SU}_{n+1})$, 
\begin{align} \label{eqn:movingatotheright}
a \adel b \in \Omega^{(0,1)}_q(\mathrm{F}_{n+1}) \OO_q(\mathrm{SU}_{n+1}).
\end{align}
Consider the set of algebra generators $z_{i}^{\varpi_x}$ of $\OO_q(\mathrm{SU}_{n+1})$ presented in \eqref{eqn:QPHSGens}. Since $z_{i}^{\varpi_x}$ is a highest weight vector with respect to the left action of $U_q(\mathfrak{sl}_{n+1})$ on $\OO_q(\mathrm{SU}_{n+1})$, it follows from \eqref{eqn:tangent.exd} that it is contained in the kernel  of $\adel$, the exterior derivative of our fodc on $\OO_q(\mathrm{SU}_{n+1})$. Thus, for any $b \in \OO_q(\mathrm{F}_{n+1})$, we see that  
\begin{align} \label{eqn:firststep}
z_{i}^{\varpi_x} \adel b = \adel(z_{i}^{\varpi_x} b) - \adel(z_{i}^{\varpi_x})b  = \adel(z_{i}^{\varpi_x} b).
\end{align}
Consider next the element
\begin{align*}
(S \otimes \id) \circ \Delta(z^{\varpi_x}_{N_x}) = \sum_{a=1}^{N_x} S(c^{\varpi_x}_{f_{N_x},v_a}) \otimes c^{\varpi_x}_{f_a,v_{N_x}} =  \sum_{a=1}^{N_x} S(c^{\varpi_x}_{f_{N_x},v_a}) \otimes z^{\varpi_x}_{a} \in \EE_{-\varpi_x} \otimes \EE_{\varpi_x},
\end{align*}
where the fact that $S(z^{\varpi_x}_{f_{N_x},v_a})$ is an element of $\EE_{-\varpi_x}$ follows from the calculation
\begin{align*}
K_i \triangleright S(c^{\varpi_x}_{f_{N_x},v_a}) = S(c^{\varpi_x}_{K_i f_{N_x},v_a}) = q^{-(\alpha_i,\varpi_{x})} S(c^{\varpi_x}_{f_{N_x},v_a}).
\end{align*}
Multiplying both sides of the identity gives us the new identity
\begin{align} \label{eqn:theDetID}
\e(z^{\varpi_x}_{N_x}) = 1 = \sum_{a=1}^{N_x} S(z^{\varpi_x}_{f_{N_x},v_a}) z^{\varpi_x}_{a}.
\end{align}
Inserting this into \eqref{eqn:firststep} gives us the expression 
\begin{align*}
z_{i}^{\varpi_x} \adel b =  \sum_{a=1}^{N_x}  \adel(z_{i}^{\varpi_x} bS(z^{\varpi_x}_{f_{N_x},v_a})z^{\varpi_x}_{a}) = \sum_{a=1}^{N_x}  \adel(z_{i}^{\varpi_x} b S(z^{\varpi_x}_{f_{N_x},v_a}))z^{\varpi_x}_{a}.
\end{align*}
From the fact that our $\mathcal{P}$-grading is an algebra grading, we can conclude that 
\begin{align*}
z_{i}^{\varpi_x} b S(z^{\varpi_x}_{f_{N_x},v_a}) \in \EE_{0} = \OO_q(\mathrm{F}_{n+1}).
\end{align*}
Thus we see that \eqref{eqn:movingatotheright} is satisfied for $a = z^{\varpi_x}_i$. Consider next the dual generator $\overline{z}^{\varpi_x}_{i}$, and note that 
\begin{align*}
\overline{z}^{\varpi_x}_{i} \adel b = \adel(\overline{z}^{\varpi_x}_{i}b) - \adel(\overline{z}^{\varpi_x}_{i}) b.
\end{align*}
An analogous application of the identity in \eqref{eqn:theDetID} shows us that this element is again contained in $\Omega^{(0,1)}_q(\mathrm{F}_{n+1})\OO_q(\mathrm{SU}_{n+1})$, which is to say, \eqref{eqn:movingatotheright} is again satisfied. Thus the second identity in \eqref{eqn:StrongId} is satisfied.  It is now clear that the zero map is a strong principal connection, and moreover, that the zero map is the unique principal connection for the bundle.
\end{proof}

The following corollary explicitly describes the connection associated to $\Pi$ for a relative Hopf submodule of $\OO_q(\mathrm{SU}_{n+1})$. It is a direct consequence of the various constructions involved, and so, we omit it. 

\begin{prop}
For any simple relative Hopf module $\F \in {}^{A}_{B} \mathrm{Mod}_0$, considered as a sub-object of $\OO_q(\mathrm{SU}_{n+1})$, the connection associated to $\Pi$ is given explicitly by
\begin{align*}
\adel_{\F}: \F \to \Omega^{(0,1)}_q(\mathrm{F}_{n+1}) \otimes_{\OO_q(\mathrm{F}_{n+1})} \F, & & f \mapsto j \! \left(\sum_{\gamma \, \in \,\Delta^+} (E_{\gamma} \triangleright f) \otimes e_{\gamma} \right)\!.
\end{align*}
\end{prop}

Since the zero map is obviously an $\OO_q(\mathrm{SU}_{n+1})$-module map, the following corollary now follows directly from \cite[Theorem 4.1]{Bimodule.AC.ROB}. 

\begin{cor}
For every relative Hopf module $\F \in {}^A_B \mathrm{Mod}_0$, the associated connection $\adel_{\F}: \F \to \Omega^{(0,1)}_q(\mathrm{F}_{n+1}) \otimes_{\OO_q(\mathrm{F}_{n+1})} \F$ is a bimodule connection.
\end{cor}

We will now follow the representation theoretic argument given in \cite[\textsection 3.2]{HVBQFM} and establish uniqueness and flatness for relative line module connections.

\begin{lem}
For each line module $\EE_{\lambda}$, the connection is the unique  left $\OO_q(\mathrm{SU}_{n+1})$-covariant connection $\adel_{\EE_{\gamma}}$. Moreover,  the connection is flat.
\end{lem}
\begin{proof}
We know from Proposition \ref{prop:omegaEE} that 
\begin{align*}
\Phi\!\left(\Omega^{(0,1)}_q(\mathrm{F}_{n+1}) \otimes_{\OO_q(\mathrm{F}_{n+1})} \EE_{\lambda}\right) \cong \bigoplus_{\gamma \in \Delta^+} \Phi(\EE_{\gamma + \lambda}).
\end{align*}
Thus we see that the only morphism 
\begin{align*}
\Phi(\EE_{\lambda}) \to \Phi\!\left(\Omega^{(0,1)}_q(\mathrm{F}_{n+1}) \otimes_{\OO_q(\mathrm{F}_{n+1})} \EE_{\lambda}\right)\!.
\end{align*}
is the zero morphism. Since the difference of two connections is a left $\OO_q(\mathrm{F}_{n+1})$-module map, it now follows from Takeuchi's equivalence that $\adel_{\EE_{\lambda}}$ is the unique left $\OO_q(\mathrm{SU}_{n+1})$-covariant connection on $\EE_{\lambda}$. (For more details see \cite[Proposition 3.4]{HVBQFM}.)

Proposition \ref{prop:omegaEE} also implies that the only morphism 
$$
\Phi(\EE_{\lambda}) \to \Phi\!\left(\Omega^{(0,2)} \otimes_{\OO_q(\mathrm{F}_{n+1})} \EE_{\lambda}\right)
$$
is the zero map. Thus, since curvature is a left $\OO_q(\mathrm{F}_{n+1})$-module map, we can similarly conclude that $\adel_{\EE_{\lambda}}$ is  flat.  (For more details see \cite[Proposition 3.3]{HVBQFM}.)
\end{proof}

This means that we have a complex
$$
\adel_{\EE_{\lambda}}: \Omega^{(0,\bullet)}_q(\mathrm{F}_{n+1}) \otimes_{\OO_q(\mathrm{F}_{n+1})} \mathcal{E}_{\lambda} \to   \Omega^{(0,\bullet)}_q(\mathrm{F}_{n+1}) \otimes_{\OO_q(\mathrm{F}_{n+1})} \mathcal{E}_{\lambda},
$$
which is a direct $q$-deformation of the Dolbeault anti-holomorphic complex of the line modules over the $A$-series full flag manifolds. We denote the cohomology groups of the complex by $H^{(0,\bullet)}_{\adel}(\EE_{\lambda})$.

\subsection{A Borel--Weil Theorem for $\OO_q(\mathrm{F}_{n+1})$}

In this subsection we establish the main result of this section, namely a noncommutative  Borel--Weil theorem for the full quantum flag manifold $\OO_q(\mathrm{F}_{n+1})$. This gives a noncommutative differential geometric realisation of all the irreducible type-1 finite-dimensional representations of $U_q(\mathfrak{sl}_{n+1})$.

\begin{thm} \label{thm:BW}
For the full quantum flag manifold $\OO_q(\mathrm{F}_{n+1})$, it holds that:
\begin{enumerate}
 \item $H^0_{\adel}(\mathcal{E}_{\lambda})$ is an irreducible $U_q(\mathfrak{sl}_{n+1})$-module of highest weight $-w_0(\lambda)$, for all \mbox{$\lambda \in \mathcal{P}^+$}, 
 \item $H^0_{\adel}(\mathcal{E}_{\lambda}) = 0$, for all $\lambda \in \mathcal{P} \, \backslash \, \mathcal{P}^+$.
\end{enumerate}
\end{thm}
\begin{proof}
It follows from the construction of $\adel$, as presented in \textsection \ref{subsection:TangentSpaces},  that for any $e \in \EE_{\lambda}$, we have $\overline{\partial}(e)=0$ if and only if 
\begin{align} \label{eqn:HWBW}
E_{\gamma} \triangleright e = 0, & & \textrm{ for all } \gamma \in \Delta^+,
\end{align}
which is to say, if and only if $e$ is a highest weight vector, of weight $\lambda$.

If $\lambda \in \mathcal{P}/\mathcal{P}^+$, then it is clear that no such element exists, meaning that  $H^0_{\adel}(\mathcal{E}_{\lambda}) = 0$. For $\lambda \in \mathcal{P}^+$, recalling the Peter--Weyl decomposition of $\OO_q(\mathrm{SU}_{n+1})$, the only highest weight elements of $\OO_q(\mathrm{SU}_{n+1})$ of weight $\lambda$ are of the form 
$$
v \otimes v_{\lambda} \in V_{\lambda}^{\vee} \otimes V_{\lambda} \cong C(V_{\lambda}),
$$
where $v_{\lambda}$ is a highest weight vector of $V_{\lambda}$. Thus we see that  $H^0_{\adel}(\mathcal{E}_{\lambda})$ is an irreducible $U_q(\mathfrak{sl}_{n+1})$-module of highest weight $-w_0(\lambda)$, for all $\lambda \in \mathcal{P}^+$.
\end{proof}

Restricting to the case of the trivial line module $\EE_0$, we get the following immediate corollary.

\begin{cor}
The Lusztig fodc $\Omega^{(0,1)}_q(\mathrm{F}_{n+1})$ is \emph{connected}, which is to say 
$$
\mathrm{ker}\Big(\adel: \OO_q(\mathrm{F}_{n+1}) \to \Omega^{(0,1)}_q(\mathrm{F}_{n+1}) \Big) = \mathbb{C}1.
$$
\end{cor}

We note that, by contrast, the fodc on $\OO_q(\mathrm{SU}_{n+1})$ is not connected. However, this can be rectified by the addition of the anti-holomorphic forms, corresponding to the negative root vectors. This will be addressed in future work.

\subsection{Quantum Homogeneous Coordinate Rings}

Classically, for each $\lambda \in \mathcal{P}^+$, the line bundle $\EE_{\lambda}$ is positive, or equivalently ample. We denote by $S_{\lambda}[\mathrm{F}_{n+1}]$ the homogeneous coordinate ring  of the associated complex projective Kodaira embedding of $\mathrm{F}_{n+1}$. This ring admits a direct generalisation to the quantum group setting: For $\lambda = \sum_{i=1}^{n} \lambda_i \varpi_i$, the \emph{quantum homogeneous $\lambda$-coordinate ring} is the subalgebra 
\begin{align*}
S_{q,\lambda}[\mathrm{F}_{n+1}] := \Big\langle z^{\lambda} U_q(\mathfrak{sl}_{n+1})  \Big \rangle \subseteq \OO_q(\mathrm{SU}_{n+1}).
\end{align*}
These algebras have appeared many times in the literature, see for example  \cite{Resh.Lak,Soilbelman.flag,TaftTowber}. Moreover they constitute an important class of examples for noncommutative projective algebraic geometry \cite{RigalZadun, Rosen.NCS}, and for quantum cluster algebras \cite{GrL2011, GrL2014}.

We will now use the quantum Borel--Weil theorem established above to give a novel noncommutative differential geometric realisation of the quantum homogeneous coordinate rings, generalising the classical holomorphic section presentation of the ring. The proof directly generalises the proof of the analogous statement for the irreducible quantum quantum flag manifolds given in \cite{BwGrassDM,CDOBBW}, and so, we omit it.

\begin{prop} \label{prop:HHCR}
For the full quantum flag manifold $\OO_q(\mathrm{F}_{n+1})$, and any $\lambda \in \mathcal{P}^+$, 
\begin{align} \label{eqn:HCR}
S_{q,\lambda}[\mathrm{F}_{n+1}] \cong \bigoplus_{k \in \mathbb{Z}_{\geq 0}} H^0_{\adel}(\mathcal{E}_{k\lambda}),
\end{align}
where each summand on the right-hand side is considered as a subspace of $\OO_q(\mathrm{SU}_{n+1})$.
\end{prop}


\section{Outside the Nice  Setting} \label{section:Nonnice}

The reduced decompositions $\mathbf{j}$ and $\mathbf{j}'$ (see \ref{appendix.opposite} for the definition of $\mathbf{j}'$) are distinguished, among the reduced decompositions of the longest element of the Weyl group $S_{n+1}$, as those that are \emph{nice} in the sense of Littlemann. The notion of a nice decomposition is defined for all Weyl groups, and is given in terms of crystal graphs for finite-dimensional $U_q(\mathfrak{g})$-modules. See \cite{LittleCrystal} for more details.

In this section we consider some low-dimensional examples of Lusztig's root vectors for non-nice decompositions, or more precisely, for their associated commutation classes. Explicitly, we examine when the span of the root vectors gives a quantum tangent space for the full quantum flag manifold. Since the case of $U_q(\mathfrak{sl}_2)$ is trivial, and for $U_q(\mathfrak{sl}_3)$ both decompositions are nice, we start with the case of $U_q(\mathfrak{sl}_4)$, treating its $6$ non-nice commutation classes in detail. We then treat the case of $U_q(\mathfrak{sl}_5)$, examining those non-nice commutation classes, from a total of $62$ such classes, that are one or two braid moves from our choice of nice class.

\subsection{The Opposite Involution, Reduced Decompositions, and DC}

In this subsection we look at the action of the opposite involution on the differential calculi associated to reduced decompositions of $w_0$ the longest element of the Weyl group. In particular, we will show that the dimension of the maximal prolongation of a fodc is invariant under the opposite involution, as defined in Appendix \ref{appendix.opposite}.

\begin{lem}
For any reduced decomposition $\mathbf{i}$ of $w_0$, the space of positive root vectors $\mathfrak{n}^+_{\mathbf{i}}$ is a left, or a right, tangent space if and only if $\mathfrak{n}^+_{\mathbf{i}'}$ is a left, respectively right, tangent space. 
\end{lem}
\begin{proof}
This follows directly from Proposition \ref{prop:wzerontonprime}, and the fact that $\overline{w_0}$ is a Hopf algebra automorphism of $\OO_q(\mathrm{SU}_{n+1})$.
\end{proof}

The following lemma now describes, for the tangent space case, the relationship between the corresponding ideals.

\begin{lem}
Assume that $\mathfrak{n}^+_{\mathbf{i}}$ is a left, or a right, tangent space, and let us denote by $I \subseteq \OO_q(\mathrm{SU}_{n+1})^+$ the corresponding ideal. Then $\mathfrak{n}^+_{\mathbf{i}'}$ is a left, respectively right, tangent space with corresponding ideal $\overline{w_0}^{\,*}(I)$.
\end{lem}
\begin{proof}
Let us assume that $y \in I$, which is to say, let us assume that 
\begin{align}
\big\langle X, y \big\rangle = 0,  & & \textrm{ ~~~~~~~~~~~~\, for all } X \in \mathfrak{n}^+_{\textbf{i}}.
\end{align}
Equivalently, $y \in I$ if and only if
\begin{align*}
\Big\langle \overline{w_0}(X), \overline{w_0}^*(y)\Big\rangle = 0,  & & \textrm{ for all }  X \in \mathfrak{n}^+_{\textbf{i}}.
\end{align*}
Thus, denoting by $I' \subseteq \OO_q(\mathrm{SU}_{n+1})^+$ the ideal corresponding to the tangent space
$
\mathfrak{n}^+_{\mathbf{i}'} = \overline{w_0}(\mathfrak{n}^+_{\mathbf{i}}),
$
we see that $\overline{w_0}^*(y) \in I'$, and so $\overline{w_0}^*(I) \subseteq I'$. In the other direction, we see that $z \in I'$ if and only if
\begin{align*}
0 = \Big\langle \overline{w_0}(X), z\Big\rangle = \Big\langle X, \overline{w_0}^*(z)\Big\rangle, & & \textrm{ for all } X \in \mathfrak{n}^+_{\mathbf{i}}.
\end{align*}
Thus, since this is equivalent to having $\overline{w_0}^*(z) \in I$, we see that we have established the opposite inclusion, and hence equality.
\end{proof}

\begin{lem}
Let $\mathbf{i}$ be a reduced decomposition of the longest element of the Weyl group such that $\mathfrak{n}^+_{\mathbf{i}}$ is a right or a left tangent space, and denote by $\Lambda_{q,\mathbf{i}}^{\bullet}$ the associated quantum exterior algebra. Then, denoting by $\Lambda_{q,\mathbf{i}'}^{\bullet}$ the quantum exterior algebra of the quantum tangent space $\mathfrak{n}^+_{\mathbf{i}'} = \overline{w_0}(\mathfrak{n}^+_{\mathbf{i}})$, an algebra isomorphism $\Lambda^{\bullet}_{q,\mathbf{i}} \to \Lambda^{\bullet}_{q,\mathbf{i}’}$ is determined by
\begin{align*}
[a_1] \wedge \cdots \wedge [a_k] ~ \mapsto ~ [\overline{w_0}^*(a_1)] \wedge \cdots \wedge [\overline{w_0}^*(a_k)], & & \textrm{ for } k \in \mathbb{Z}_{\geq 0}.
\end{align*}
\end{lem}
\begin{proof}
Note first that we have a linear isomorphism
\begin{align*}
\rho: \Lambda^1_{q,\mathbf{i}} \to \Lambda^1_{q,\mathbf{i}’}, & & [a] \mapsto [\overline{w_0}^*(a)].
\end{align*}
We can uniquely extend this to an algebra isomorphism 
$$
\rho: \mathcal{T}\big(\Lambda^1_{q,\mathbf{i}}\big) \to \mathcal{T}\big(\Lambda^1_{\mathbf{i}'}\big),
$$
where $\mathcal{T}(\Lambda^1_{q,\mathbf{i}})$, and $\mathcal{T}(\Lambda^1_{q,\mathbf{i}'})$, denote the tensor algebra of $\Lambda^1_{q,\mathbf{i}})$, and $\Lambda^1_{q,\mathbf{i}'})$, respectively. For any element $y \in I$, we now see that  
\begin{align*}
\rho\big([y_{(1)}] \otimes [y_{(2)}]\big) =  [\overline{w_0}^*(y_{(1)})] \otimes [\overline{w_0}^*(y_{(2)})] = [\overline{w_0}^*(y)_{(1)}] \otimes [\overline{w_0}^*(y)_{(2)}] \in (\overline{w_0}^*(I))^{(2)},
\end{align*}
where we have used the fact that $\overline{w_0}^*$ is a Hopf algebra map. An analogous calculation in the opposite direction then confirms that the inverse $\rho^{-1}$ maps $(\overline{w_0}^*(I))^{(2)}$ to $I^{(2)}$. Thus $\rho$ descends to an algebra isomorphism between $\Lambda^{\bullet}_{q,\mathbf{i}}$ and $\Lambda^{\bullet}_{q,\mathbf{i}’}$ as claimed.
\end{proof}

The following proposition now follows directly. In particular, it means that one need only check the classical dimension requirement for $\mathbf{i}$ or $\mathbf{i}'$.

\begin{prop}
A reduced decomposition $\mathbf{i}$ gives a tangent space whose associated quantum exterior algebra has classical dimension if and only if the same is true for $\mathbf{i}'$.
\end{prop}

\subsection{The Case of $U_q(\mathfrak{sl}_4)$} \label{subsection:sl4}

For the quantised enveloping algebra  $U_q(\mathfrak{sl}_4)$, the Weyl group is $S_4$, and its longest element has $8$ commutation classes of reduced decompositions. We present them in diagrammatic form, each node being a choice of representative from a commutation class and an edge between two nodes indicating that the two decompositions differ by a braid move:

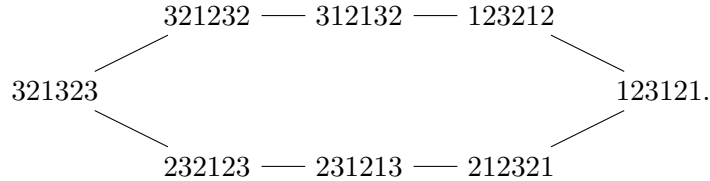
\begin{center}
\begin{tikzpicture}
\node  (C0) at (-2,1) {$321323$};
    \node (A1) at (0,2) {$321232$};
    \node (A2) at (2,2) {$312132$};
    \node (A3) at (4,2) {$123212$};
    \node (B1) at (0,0) {$232123$};
    \node (B2) at (2,0) {$231213$};
    \node (B3) at (4,0) {$212321$};
    \node (C1) at (6,1) {$123121$.};
    \draw (A1) -- (A2) -- (A3);
    \draw (B1) -- (B2) -- (B3);
    \draw (A3) -- (C1);
    \draw (B3) -- (C1);
    \draw (C0) -- (A1);
    \draw (C0) -- (B1);
\end{tikzpicture}
\end{center}

Clearly, operating on a  reduced decomposition by the opposite automorphism gives a well-defined map between commutation classes. We present this action in the following diagram, where a blue arrow between two decompositions denotes that the automorphism maps one decomposition to another.

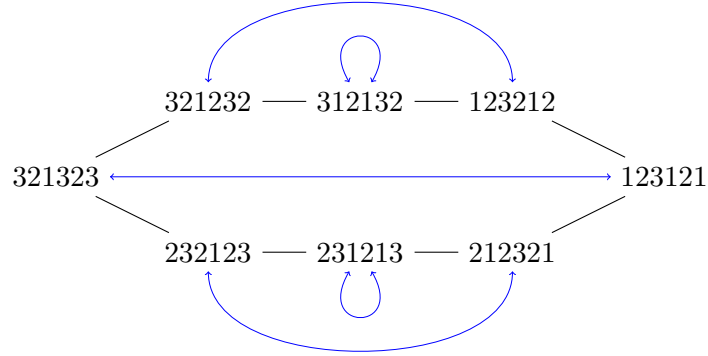
\begin{center}
\begin{tikzpicture}
\node (C0) at (-2,1) {$321323$};
\node (A1) at (0,2) {$321232$};
\node (A2) at (2,2) {$312132$};
\node (A3) at (4,2) {$123212$};
\node (B1) at (0,0) {$232123$};
\node (B2) at (2,0) {$231213$};
\node (B3) at (4,0) {$212321$};
\node (C1) at (6,1) {$123121$};
\draw (A1) -- (A2) -- (A3);
\draw (B1) -- (B2) -- (B3);
\draw (A3) -- (C1);
\draw (B3) -- (C1);
\draw (C0) -- (A1);
\draw (C0) -- (B1);
\draw [blue, line width=0.01pt, <->, out=120, in=60, looseness=8] (A2) to (A2);
\draw [blue, <->, line width=0.01pt, out=-120, in=-60, looseness=8] (B2) to (B2);
\draw [blue, <->, line width=0.01pt, out=90, in=90, looseness=0.9] (A1) to (A3);
\draw [blue, <->, line width=0.01pt, out=-90, in=-90, looseness=0.9] (B1) to (B3);
\draw [blue, <->, line width=0.01pt] (C0) -- (C1);
\end{tikzpicture}
\end{center}

We can calculate the associated spaces of root vectors using the formulae given in Appendix \ref{app:Lusztigroots}. The set of $6$ roots vectors associated to each decomposition always contains the simple root vectors $E_1, E_2$, and $E_3$. We present the other three non-simple root vectors in the table below, where for notational brevity we have denoted $[-,-]_+ := [-,-]_{q}$ and $[-,-]_- := [-,-]_{q^{-1}}$.

The formulae in Example \ref{eg:singlecomm} show us how the coproduct operates on the length two commutators.  A direct calculation shows that the coproduct of $[[E_2,E_3]_-,E_1]_-$ is equal to 
\begin{align*}
[[E_2,E_3]_-,E_1]_{-} \otimes K_1K_2K_3 + q^{-1}\nu E_1 \otimes [E_2,E_3]_-K_1 + q^{-1}\nu E_3 \otimes [E_2,E_1]_-\\
~~~~~~  + ~~ q^{-2}\nu^2 E_1E_3 \otimes E_2K_1  +1 \otimes [[E_2,E_3]_-,E_1]_-, ~~~~~~~~~~~~~~~~~~
\end{align*}
and moreover, that the coproduct of $[[E_1,E_2]_-,E_3]_+$ is equal to 
\begin{align*}
 [[E_1,E_2]_-,E_3]_{+}  \otimes K_1K_2K_3 + q^{-1}\nu [E_2,E_3]_+ \otimes E_1K_3 + \nu  [E_1,E_2]_- \otimes E_3 K_1K_2 \\
 ~~~~~~ + (q^{-1} - 1) \nu E_2 \otimes E_1E_3  + 1 \otimes [[E_1,E_2]_-,E_3]_+. ~~~~~~~~~~~~~~~~~~
\end{align*}

\bigskip

\begin{center}
\begin{tabular}{ |c|c|c| } 
 \hline
&   \\
 \textrm{Decomposition} & \textrm{Non-Simple Root Vectors}   \\
   &   \\
 \hline
   &   \\
$321323$ & ~~ $[E_3,E_2]_-, ~~ [E_2,E_1]_-, ~~ [[E_3,E_2]_-,E_1]_-$ \\ 
   &    \\
$321232$ & ~~ $[E_1,E_2]_-, ~~ [E_3,E_2]_-, ~~ [[E_3,E_2]_-,E_1]_-$  \\ 
   &    \\
 $232123$ & ~~ $[E_2,E_1]_-, ~~ [E_2,E_3]_-, ~~[[E_3,E_2]_-,E_1]_-$ \\ 
    &    \\
 $213213$  & ~~ $[E_2,E_1]_-, ~~ [E_2,E_3]_-, ~~ [[E_2,E_3]_-,E_1]_-$  \\ 
   &   \\
$132132$ & ~~ $[E_1,E_2]_-, ~~ [E_3,E_2]_-, ~~ [[E_1,E_2]_-,E_3]_+$  \\ 
    &    \\
 $123212$ & ~~ $[E_1,E_2]_-, ~~ [E_3,E_2]_-, ~~ [[E_1,[E_2,E_3]_-]_-$  \\ 
    &   \\
 $212321$ & ~~ $[E_2,E_1]_-,  ~~ [E_2,E_3]_-, ~~ [[E_1,E_2]_-,E_3]_-$  \\ 
 &   \\
 $123121$ & ~~ $[E_1,E_2]_-, ~~ [E_2,E_3]_-, ~~ [[E_1,E_2]_-,E_3]_-$  \\ 
  &    \\
 \hline
\end{tabular}
\end{center}

Together with the action of the opposite automorphism, this allows us to determine which spaces of positive root vectors give tangent spaces for the full quntum flag manifold $\OO_q(\mathrm{F}_4)$. In this higher rank situation, we find that $T \oplus \mathbb{C} 1$ is no longer always a subcoalgebra. For some cases it is a right coideal of $\OO_q(\mathrm{F}_4)^{\circ}$, and for some cases it is a left coideal.  This means that the corresponding fodc is either left, or respectively right, covariant. We present these results in the following diagram, building on the graphical presentation of the reduced decompositions given above:

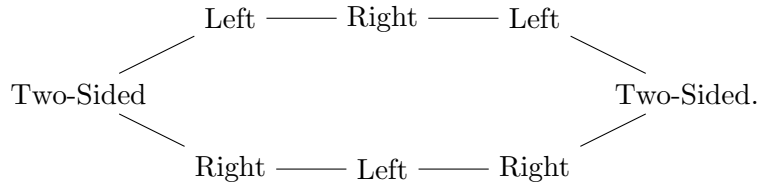
\begin{center}
\begin{tikzpicture}
\node  (C0) at (-2,1) {Two-Sided};
    \node (A1) at (0,2) {Left};
    \node (A2) at (2,2) {Right};
    \node (A3) at (4,2) {Left};
    \node (B1) at (0,0) {Right};
    \node (B2) at (2,0) {Left};
    \node (B3) at (4,0) {Right};
    \node (C1) at (6,1) {Two-Sided.};
    \draw (A1) -- (A2) -- (A3);
    \draw (B1) -- (B2) -- (B3);
    \draw (A3) -- (C1);
    \draw (B3) -- (C1);
    \draw (C0) -- (A1);
    \draw (C0) -- (B1);
\end{tikzpicture}
\end{center}

Finally, we address the question of the dimension of the maximal prolongations of the fodc associated to the different commutation classes of reduced decompositions. The following result states that those that have classical dimension are precisely those that give two-sided coideals. The proof is a direct calculation, analogous to the proof of Lemma \ref{lem:THERELATIONS}, and so, it is omitted.

\begin{prop} \label{prop:MPsl4}
For the Drinfeld--Jimbo quantised enveloping algebra $U_q(\mathfrak{sl}_4)$, the commutation classes of $\mathbf{j}$ and $\mathbf{j}'$ are the only commutation classes of reduced decompositions of $w_0$, the longest element of the Weyl group $S_{n+1}$, whose associated tangent spaces have dc with classical total degree.
\end{prop}

We can represent the above proposition diagrammatically as follows:
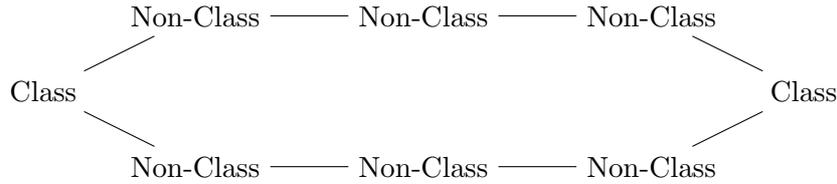
\begin{center}
\begin{tikzpicture}
\node  (C0) at (-3,1) {Class};
    \node (A1) at (-1,2) {Non-Class};
    \node (A2) at (2,2) {Non-Class};
    \node (A3) at (5,2) {Non-Class};
    \node (B1) at (-1,0) {Non-Class};
    \node (B2) at (2,0) {Non-Class};
    \node (B3) at (5,0) {Non-Class};
    \node (C1) at (7,1) {Class};
    \draw (A1) -- (A2) -- (A3);
    \draw (B1) -- (B2) -- (B3);
    \draw (A3) -- (C1);
    \draw (B3) -- (C1);
    \draw (C0) -- (A1);
    \draw (C0) -- (B1);
\end{tikzpicture}
\end{center}
where \emph{Class}, and \emph{Non-Class}, indicate whether or not the maximal prolongation of the fodc associated to the reduced decomposition has classical, or non-classical, dimension respectively. Thus we again see that the nice decompositions are confirmed as exceptional amongst the $8$ commutation classes.

\subsection{The Case of $U_q(\mathfrak{sl}_5)$}  \label{subsection:sl5}

In this subsection we treat the rank four case of $U_q(\mathfrak{sl}_5)$. For this example the number of reduced decompositions rises to $62$. A full table of the reduced decomposition classes can be found in \cite[Figure 3]{Ziegler.Bruhat}. Here we content ourselves with examining the sub-graph consisting of the reduced decomposition classes generated by the decomposition $4321432434$ and at most two braid moves:
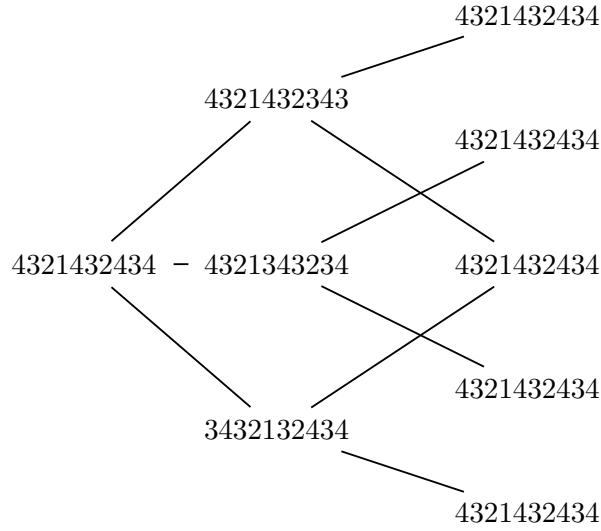
\begin{align*}
\begin{tikzpicture} [scale=1.1]
  \draw (0.7,0) node (A) {4321432434}
            (3,2) node(B) {4321432343}
            (3,0) node(C) {4321343234}
            (3,-2) node(D) {3432132434}
            (6,3) node(E) {4321432434}
            (6,1.5) node(F) {4321432434}
            (6,0) node(G) {4321432434}
            (6,-1.5) node(H) {4321432434}
            (6,-3) node(I) {4321432434}
            (8,0) node {};
  \draw[shorten >=2pt, shorten <=2pt, , >=latex, line width=0.7pt] (A) -- node[right] {} (B);
  \draw[ shorten >=2pt, shorten <=2pt, , >=latex, line width=0.7pt] (A) -- node[right] {} (C);
  \draw[ shorten >=2pt, shorten <=2pt, , >=latex, line width=0.7pt] (A) -- node[right] {} (D);
  \draw[ shorten >=2pt, shorten <=2pt, , >=latex, line width=0.7pt] (B) -- node[right] {} (E);
  \draw[ shorten >=2pt, shorten <=2pt, , >=latex, line width=0.7pt] (B) -- node[right] {} (G);
  \draw[ shorten >=2pt, shorten <=2pt, , >=latex, line width=0.7pt] (C) -- node[right] {} (F);
  \draw[ shorten >=2pt, shorten <=2pt, , >=latex, line width=0.7pt] (C) -- node[right] {} (H);
  \draw[ shorten >=2pt, shorten <=2pt, , >=latex, line width=0.7pt] (D) -- node[right] {} (G);
  \draw[ shorten >=2pt, shorten <=2pt, , >=latex, line width=0.7pt] (D) -- node[right] {} (I);
\end{tikzpicture}
\end{align*}
A direct calculation shows, that for certain decompositions, the associated positive root vectors form neither a left nor a right tangent space, a phenomenon that does not occur for lower ranks. We now present in diagrammatic form the tangent space characterisation of the commutation classes: 
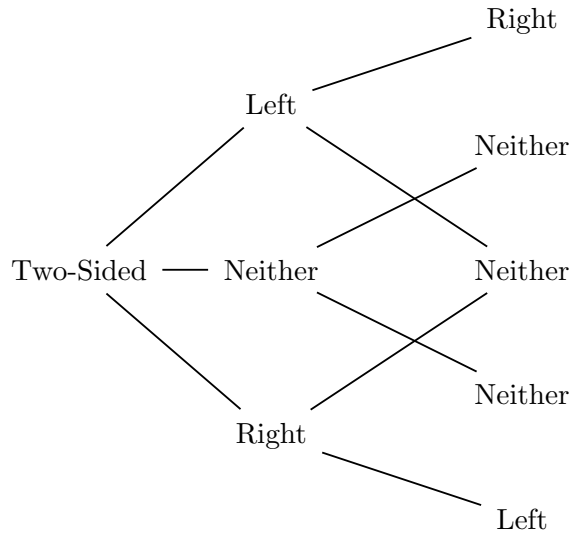
\begin{align*}
\begin{tikzpicture} [scale=1.1]
  \draw (0.7,0) node (A) {\textrm{Two-Sided}}
            (3,2) node(B) {\textrm{Left}}
            (3,0) node(C) {\textrm{Neither}}
            (3,-2) node(D) {\textrm{Right}}
            (6,3) node(E) {\textrm{Right}}
            (6,1.5) node(F) {\textrm{Neither}}
            (6,0) node(G) {\textrm{Neither}}
            (6,-1.5) node(H) {Neither}
            (6,-3) node(I) {Left}
            (8,0) node {};
  \draw[shorten >=2pt, shorten <=2pt, , >=latex, line width=0.7pt] (A) -- node[right] {} (B);
  \draw[ shorten >=2pt, shorten <=2pt, , >=latex, line width=0.7pt] (A) -- node[right] {} (C);
  \draw[ shorten >=2pt, shorten <=2pt, , >=latex, line width=0.7pt] (A) -- node[right] {} (D);
  \draw[ shorten >=2pt, shorten <=2pt, , >=latex, line width=0.7pt] (B) -- node[right] {} (E);
  \draw[ shorten >=2pt, shorten <=2pt, , >=latex, line width=0.7pt] (B) -- node[right] {} (G);
  \draw[ shorten >=2pt, shorten <=2pt, , >=latex, line width=0.7pt] (C) -- node[right] {} (F);
  \draw[ shorten >=2pt, shorten <=2pt, , >=latex, line width=0.7pt] (C) -- node[right] {} (H);
  \draw[ shorten >=2pt, shorten <=2pt, , >=latex, line width=0.7pt] (D) -- node[right] {} (G);
  \draw[ shorten >=2pt, shorten <=2pt, , >=latex, line width=0.7pt] (D) -- node[right] {} (I);
\end{tikzpicture}
\end{align*}
Just as for $U_q(\mathfrak{sl}_4)$, we get a dual picture by operating by the opposite automorphism. We will not treat the remaining $44$ cases.  Instead, in the next section we make a conjecture and pose some natural questions.

\subsection{Some Conjectures}

We finish with some conjectures, identifying interesting future directions of research. Motivated by the observations of \textsection \ref{subsection:sl4}  and \textsection \ref{subsection:sl5}, we conjecture that the Lusztig fodc, and the \emph{dual} fodc associated to the reduced decomposition $\mathbf{j}'$, are distinguished by having tangent spaces $T$ for which $T \oplus \mathbb{C}1$ is a coideal of $\OO_q(\mathrm{SU}_{n+1})^{\circ}$.

\begin{conj} \label{conj:coideal}
The two reduced decompositions $\mathbf{j}$ and $\mathbf{j}'$, of the longest element of the Weyl group of $\mathfrak{sl}_{n+1}$, are the only reduced decompositions for which the span of the positive root vectors is a right and left tangent space. 
\end{conj}

From the discussions of \textsection \ref{subsection:sl4}  and \textsection \ref{subsection:sl5}, we see that the attribute of giving a two-sided, left, or right tangent space, or not giving a tangent spaces at all, separates equivalence classes of reduced decompositions into four distinct non-empty families. It is natural to ask if these divisions admit a natural combinatorical description in one of the many equivalent formulations of reduced decomposition classes, such as rhombic tilings \cite{Elnitsky,RhomBottSam} or pseudoline arrangements \cite{Valtr}. This gives the following problem.

\begin{faidhb}
For the Drinfeld--Jimbo quantum group $U_q(\mathfrak{sl}_{n+1})$, determine those reduced decompositions of the longest elements of the Weyl group for which the associated space of positive root vectors  is a right or a left quantum tangent space.
\end{faidhb}

Motivated by the low-order maximal prolongation results presented so far in this subsection, and by Conjecture \ref{conj:coideal}, we make the following second conjecture.

\begin{conj}
The two reduced decompositions $\mathbf{j}$ and $\mathbf{j}'$, of the longest element of the Weyl group of $\mathfrak{sl}_{n+1}$, are the only reduced decompositions for which the associated space of positive root vectors  forms a tangent space whose associated fodc has a maximal prolongation of classical dimension. 
\end{conj}

Another natural question to ask is whether the approach of this paper can be extended to the $B, C$, and $D$-series, or even to the exceptionals. This gives the second problem of the paper.

\begin{faidhb}
Extend the approach of the present paper to the general Drinfeld--Jimbo setting $U_q(\mathfrak{g})$, for $\mathfrak{g}$ a complex semisimple Lie algebra.
\end{faidhb}

This problem can be expected to be quite challenging. Outside the $A$-series setting, there exist various formulae for the number of reduced decompositions of the longest element of the Weyl group. For example, there is Stanley's well-known work \cite{Stanley}. However, very little is even known about the number of commutation classes of reduced decompositions. 


\subsection{A General Class of Tangent Spaces for $U_q(\mathfrak{sl}_3)$}

In this final subsection, we will examine in more detail the case of $U_q(\mathfrak{sl}_3)$. Classically, $\mathfrak{sl}_3$ is of course the first $A$-series Lie algebra with a non-simple root vector. While there are no non-nice reduced decompositions for the rank $2$ case, it is still possible to consider other, more general deformations of the commutator bracket $[E_2,E_1]$. 

\begin{prop}
For $\theta \in \mathbb{C}$, the coproduct of $[E_2,E_1]_{\theta} = E_2E_1 - \theta E_1E_2$ is given by 
$$
[E_2,E_1]_{\theta} \otimes K_1K_2 + (q^{-1} - \theta) E_2 \otimes E_1K_2  +  (1-q^{-1}\theta) E_1 \otimes E_2K_1 + 1 \otimes [E_2,E_1]_{\theta}.
$$
Thus, for all $\theta \in \mathbb{C}$, the subspace
$$
T_{\theta} := \mathrm{span}_{\mathrm{C}}\Big\{E_1, \, E_2, \, [E_2,E_1]_{\theta} \Big \} \subseteq \OO_q(\mathrm{SU}_3)^{\circ},
$$
is a quantum tangent space for $\OO_q(\mathrm{F}_3)$. 
\end{prop}

We denote the maximal prolongation of the corresponding left $\OO_q(\mathrm{SU}_3)$-covariant fodc by $\Omega^{(0,\bullet)}_{q,\theta}$. Just as for the   Lusztig tangent spaces, we can now read off the module structure of the associated tangent space $\Lambda^1$.

\begin{cor}
All non-zero actions of non-diagonal generators $u_{lk}$ on the basis elements $e_{ji}$ are given by 
\begin{align*}
e_{21}\,u_{32} = (q - \theta)e_{31} , & & e_{32}\, u_{21} = (q^{-1} - \theta)e_{31},
\end{align*}
with the diagonal actions the same as for the Lusztig fodc.
\end{cor}

Just as in Example \ref{eg:SU3.module} we can present this in diagrammatic form 
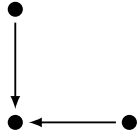
\begin{center}
\begin{tikzpicture}
  \draw (0,0) node[circle,fill,inner sep=2pt] (A) {}
        (0,-1.5) node[circle,fill,inner sep=2pt] (B) {}
        (1.5,-1.5) node[circle,fill,inner sep=2pt] (C) {};
  \draw[->, shorten >=2pt, shorten <=2pt, , >=latex, line width=0.7pt] (A) -- node[right] {} (B);
    \draw[->, shorten >=2pt, shorten <=2pt, , >=latex, line width=0.7pt] (C) -- node[right] {} (B);
\end{tikzpicture}
\end{center}
where the basis elements are arranged  in lower triangular form and we have drawn an arrow from one basis element $e$ to another $e'$ if there exists a generator $u_{ij}$ such that $eu_{ij}$ is a scalar multiple of $e'$.

\begin{prop} \label{prop:yamane}
For any $\theta \neq q^{\pm 1}$, it holds that 
$
e_{21} \wedge e_{32}  = - \theta \, e_{32} \wedge e_{21}
$
with all other products of generators equal to zero. In other words,  
$$
\Omega^{(0,2)}_{q,\theta} = \OO_q(\mathrm{SU}_3) \otimes e_{21} \wedge e_{32} = \OO_q(\mathrm{SU}_3) \otimes e_{32} \wedge e_{21} 
$$
and all higher forms are trivial. In particular, for $\theta \neq q^{\pm 1}$, the dc has non-classical dimension.
\end{prop}
\begin{proof}
A generating set for the ideal $I \subseteq \O_q(\mathrm{SU}_3)^+$ of the fodc is given by  the union of $G_1$ and $G_2$ (as defined in Proposition \ref{prop:Gens}) and the sets
\begin{align*}
G_{3} ~~:= & \, \{u_{21}u_{21}, u_{31}u_{31}, u_{32}u_{32}, u_{21}u_{31}, u_{32}u_{31}\}, \\
~~ G_{3,\theta}  := & \, \{u_{21}u_{32} - (q-\theta)u_{31}, \, u_{32}u_{21} - (q^{-1} - \theta)u_{31}\}.
\end{align*}
Just as for the Lusztig calculi, the generators contained  in $G_1$ and $G_2$ give the zero relation. From the elements of $G_{3}$ we get the relations
\begin{align}\label{eqn:Yamanerels1}
  e_{21} \wedge e_{21} = e_{32} \wedge e_{32} =  e_{31} \wedge e_{31} = 0,  & &  e_{21} \otimes e_{31} = - (q+q^{-1}-\theta) e_{31} \otimes e_{21},
\end{align}
while from the elements of $G_{3,\theta}$ we get the relations
\begin{align} \label{eqn:Yamanerels2}
e_{31} \wedge e_{32} = - (q + q^{-1} - \theta) e_{32} \wedge e_{31},  & & e_{21} \wedge e_{32}  = - \theta e_{32} \wedge e_{21}.
\end{align}
Consider now the right action of the element $u_{32}$ on $e_{21} \otimes e_{21}$:
\begin{align*}
(e_{21} \otimes e_{21})u_{32} = & \, (q-\theta)(q e_{31} \otimes e_{21} + e_{21} \otimes e_{31}),
\end{align*}
meaning that we get the relation
$
e_{21} \wedge e_{31} = -q e_{31} \wedge e_{21}.
$
Combining this with the second relation in  in \eqref{eqn:Yamanerels1}, we get that 
$$
e_{21} \wedge e_{31} = e_{31} \wedge e_{21} = 0.
$$
Similarly, acting on $e_{32} \otimes e_{32}$ by $u_{21}$ and combining with the  first relation in \eqref{eqn:Yamanerels2} allows us to conclude that $e_{31} \wedge e_{32}$ and $e_{32} \wedge e_{31}$ are also zero.  The claimed description of the degree two and higher forms now follows immediately.
\end{proof}

\begin{remark}
It is important to note that, in the proof of Proposition \ref{prop:yamane}, the set of relations produced from our choice of generating set is not closed under the right action of $\OO_q(\mathrm{SU}_3)$. Hence, they do not give a complete set of relations. We contrast this with the proof of Lemma \ref{lem:THERELATIONS}, where the image of the chosen generating set under $\omega$ is a right $\OO_q(\mathrm{SU}_{n+1})$-module of $\Lambda^{(0,1)}_q$.
\end{remark}

\appendix

\section{The $A$-Series root vectors} \label{app:Lusztigroots} 

In this appendix we give necessary  preliminaries concerning the Drinfeld--Jimbo quantum group $U_q(\mathfrak{sl}_{n+1})$, Lusztig's braid group action, and the root vectors associated to a reduced decomposition of the longest element of the $A$-series Weyl group. We also explicitly derive, for sake of completeness, the root vectors associated to the reduced decomposition  discussed in \textsection 3.1. 

\subsection{The $A_{n}$ Root System} \label{subsection:AnRootSystem}

In this subsection, so as to set notation, we recall some basic definitions and results about the $A_n$-root system associated to the special linear Lie algebra $\mathfrak{sl}_{n+1}$. Let $\{\e_i\}_{i=1}^{n+1}$ be the standard basis of $\mathbb{R}^{n+1}$,  and endow it with its canonical Euclidean structure. The root system $A_{n}$ is the pair $(V,\Delta)$, where $V$ is the subspace of $\mathbb{R}^{n+1}$ given by 
$$
V := \bigg \{ (x_1, \dots, x_{n+1}) \in \mathbb{R}^{n+1} \,|\, \sum_{i=1}^{n+1} x_i = 0 \bigg\},
$$
endowed with the restriction of the standard Euclidean inner product $(-,-)$ of $\mathbb{R}^{n+1}$, and $\Delta$, the set of roots, is the finite set 
$$
\Delta := \Big\{\alpha_{ij} := \e_i - \e_j \,|\, 1 \leq i \neq j \leq n+1 \Big\}.
$$
We adopt the usual notation 
\begin{align*}
\langle \beta,\gamma\rangle := \frac{(\beta,\gamma)}{(\beta,\beta)}, & & \textrm{ for } \beta, \gamma \in \Delta^+.
\end{align*}
We choose the standard  subset of positive roots
$$
\Delta^+ := \Big\{\alpha_{ij} := \e_i - \e_j \,|\, 1 \leq i < j \leq n+1 \Big\},
$$
and we denote by $\mathcal{Q}^+$ the $\mathbb{Z}_{\geq 0}$-span of $\Delta^+$. The associated set of simple roots is then given by
$$
\Pi := \Big\{\alpha_i := \alpha_{i,i+1} = \e_i - \e_{i+1} \,|\, i = 1, \dots, n\Big\},
$$
and the associated \emph{Cartan matrix} is given by $(a_{ij}) := (\langle \alpha_i,\alpha_j\rangle)$.

The \emph{standard  partial order} $\prec$ on $\Delta$ is defined by $\beta \prec \gamma$ if $\gamma - \beta \in \mathcal{Q}^+$. We denote the associated set of fundamental weights by $\{\varpi_1, \dots, \varpi_{n}\}$,  and write ${\mathcal P}^+$ for the $\mathbb{Z}_{\geq 0}$-span of the fundamental weights. 

We recall that the Weyl group of the root system $A_n$ is the symmetric group $S_{n+1}$. Denote by  $\{s_i \,|\, 1 \leq i \leq n\}$ the generating set of simple transpositions. These generators satisfy the \emph{quadratic relations} $s^2_i = 1$, the \emph{commutation relations} $s_is_j = s_js_i$, for \mbox{$|i - j| > 1$}, and the \emph{braid relations} $s_is_{i+1}s_i = s_{i+1}s_is_{i+1}$, for $1 \leq i \leq n - 1$. The \emph{length} $\ell(w)$ of an element $w \in W$ is the smallest non-negative integer $l$ for which there exists an expression $w = s_{i_1}s_{i_2} \cdots s_{i_l}$. Moreover, any such expression is called a \emph{reduced decomposition}. We define an equivalence relation on the set of reduced decompositions by setting $\mathbf{i} \sim \mathbf{i}'$ if $\mathbf{i}'$ can be obtained from $\mathbf{i}$ through a series of commutation relations. We call the associated equivalence classes \emph{commutation classes}.

In addition to the partial order $\succ$, we will also be interested in certain total orders known as convex orders \cite{papi}: A \emph{convex order} for $\Delta^+$ is a total order such that for any $\beta,\gamma \in \Delta^+$ satisfying $\beta + \gamma \in \Delta^+$ and $\beta < \gamma$, we necessarily have that $\beta < \beta + \gamma < \gamma$. As shown in \cite{papi}, convex orderings correspond to reduced decompositions of the longest element $w_0$ of the Weyl group. For a reduced decomposition $w_0 = s_{i_1} \cdots s_{i_d}$, with $d := |\Delta^+|$, we write
\begin{align*}
\beta_1 := \alpha_{i_1}, & &  \beta_k := s_{i_1} \cdots s_{i_{k-1}}(\alpha_{i_k}), & & \textrm{ for } k = 2, \dots , d.
\end{align*}
We have that  $\Delta^+ = \{\beta_1, \dots, \beta_d\}$ and a convex order $<$ on $\Delta^+$ is given by $\beta_1 < \cdots  < \beta_d$. Throughout the paper we deal with $\mathbf{j}$  the reduced decomposition
\begin{align*}
w_0 = (s_ns_{n-1} \cdots s_1)(s_ns_{n-1} \cdots s_{2}) \cdots (s_ns_{n-1})s_n.
\end{align*}
The associated convex order is given by 
\begin{align*}
\alpha_{ij} < \alpha_{kl}, & & \textrm{ whenever } j > l, \textrm{ or whenever } j=l \textrm{ and } i > k.
\end{align*} 
This order will be used in \textsection \ref{subsection:filtration}.

\begin{eg}
For the root system $A_3$, the convex order associated to the decomposition 
$$
w_0 = (s_3 s_2 s_1) (s_3 s_2) s_3
$$
can be represented graphically as follows: Identify the positive roots with the strictly upper triangular root vectors in $\mathfrak{sl}_{4}$. Then we have the diagram
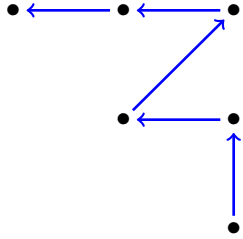
\begin{align*}
\begin{tikzpicture}[scale=1.45]
\draw (0,2) node {$\bullet$};
\draw (1,2) node {$\bullet$};
\draw (2,2) node {$\bullet$};
\draw (1,1) node {$\bullet$};
\draw (2,1) node {$\bullet$};
\draw (2,0) node {$\bullet$};
\draw [->, shorten >=5pt, shorten <=5pt,line width=1pt, blue] (1,2) -- (0,2);
\draw [->, shorten >=5pt, shorten <=5pt,line width=1pt, blue] (1,1) -- (1,2);
\draw [->, shorten >=5pt, shorten <=5pt,line width=1pt, blue] (2,2) -- (1,1);
\draw [ ->, shorten <= 5pt, shorten >= 5pt,line width=1pt, blue] (2,1) -- (2,2);
\draw [->, shorten >=5pt, shorten <=5pt,line width=1pt, blue] (2,0) -- (2,1);
\end{tikzpicture}
\end{align*}
where an arrow pointing from one node to another denotes that the domain is greater than the range. The total order is then given by the transitive closure of these inequalities.
\end{eg}

\subsection{The Drinfeld--Jimbo Quantised Enveloping Algebra $U_q(\mathfrak{sl}_{n+1})$}

Here we recall the definition of the Hopf algebra $U_q(\mathfrak{sl}_{n+1})$, following the conventions of \cite[\textsection 7]{KSLeabh}.  For any $q \in \bR$, such that $q \neq -1,0,1$,  the \emph{Drinfeld--Jimbo quantised enveloping} $U_q(\mathfrak{sl}_{n+1})$ is the noncommutative associative algebra generated by the elements $E_i, F_i, K_i$, and $K^{-1}_i$, for $ i=1, \ldots, n$, subject to the relations 
\begin{align*}
 K_iE_j = q^{a_{ij}} E_j K_i, ~~~~ K_iF_j= q^{-a_{ij}} F_j K_i, ~~~~ K_i K_j = K_j K_i, ~~~~ K_iK_i^{-1} = K_i^{-1}K_i = 1,\\
 E_iF_j - F_jE_i = \d_{ij}\frac{K_i - K\inv_{i}}{q-q^{-1}}, ~~~~~~~~~~~~~~~~~~~~~~~~~~~~~~~~~~~~~~~~~
\end{align*}
where $\nu := q-q^{-1}$, and the \emph{quantum Serre relations} 
\begin{align*}
E_{i \pm 1}^2E_i - [2]_q E_{i \pm 1}E_{i}E_{i \pm 1} + E_i E^2_{i \pm 1} = 0,  ~~~~~~~ F_{i \pm 1}^2F_i - [2]_q F_{i \pm 1}F_{i}F_{i \pm 1} + F_i F^2_{i \pm 1} = 0, \\
 E_{i}E_i  = E_jE_i, ~~~~~~~ F_iF_j = F_jF_i, ~~~~~~~ \textrm{ for } |i - j| > 1.~~~~~~~~~~~~~~~~~~~~~~~ 
\end{align*}
The formulae 
\begin{align*}
\DEL(K_i) = K_i \oby K_i, & & \DEL(E_i) = E_i \oby K_i + 1 \oby E_i, & & \DEL(F_i) = F_i \oby 1 + K_i\inv \oby F_i, 
\end{align*}
and $\e(E_i) = \e(F_i) = 0$, $\e(K_i) = 1$, define a Hopf algebra structure on $U_q(\mathfrak{sl}_{n+1})$ satisfying 
\begin{align*}
S(E_i) = - E_iK_i\inv, ~~ S(F_i) = -K_iF_i, ~~~~ S(K_i) = K_i\inv.
\end{align*}
We denote by $U_q(\mathfrak{h})$ the subalgebra of $U_q(\mathfrak{sl}_{n+1})$ generated by the elements $K_i,K^{-1}_i$, for $i=1, \dots, n$. For any $\lambda \in \mathcal{P}^+$, we denote by $V_{\lambda}$ the associated finite-dimensional irreducible type-$1$ representation of $U_q(\mathfrak{sl}_{n+1})$. For any $U_q(\mathfrak{sl}_{n+1})$-module $V$, an element $v \in V$ is said to be a \emph{weight vector} if, for some $\mathrm{wt}(v) \in \mathcal{P}^+$, we have $K_i v = q^{(\alpha_i,\mathrm{wt}(v))} v$, for all $i=1,\ldots, n$. We call $\mathrm{wt}(v)$ the \emph{weight} of $v$.

\subsection{The Quantum Coordinate Algebra $\OO_q(\mathrm{SU}_{n+1})$}

Let $V$ be a finite-dimensional \linebreak $U_q(\mathfrak{sl}_{n+1})$-module, $v \in V$, and $f \in V^*$, the linear dual of $V$. Consider the linear functional $c^{\textrm{\tiny $V$}}_{f,v}:U_q(\mathfrak{sl}_{n+1}) \to \bC$ defined by $c^{\textrm{\tiny $V$}}_{f,v}(X) := f\big(X(v)\big)$. 
%
The {\em coordinate ring} of $V$ is the subspace
\begin{align*}
C(V) := \text{span}_{\mathbb{C}}\!\big\{ c^{\textrm{\tiny $V$}}_{f,v} \,| \, v \in V, \, f \in V^*\big\} \sseq U_q(\mathfrak{sl}_{n+1})^*.
\end{align*}
It is easily checked that $C(V)$ is contained in $U_q(\mathfrak{sl}_{n+1})^\circ$, the Hopf dual of $U_q(\mathfrak{sl}_{n+1})$, and moreover that a Hopf subalgebra of $U_q(\mathfrak{sl}_{n+1})^\circ$ is given by 
\begin{align} \label{eqn:PeterWeyl}
\O_q(\mathrm{SU}_{n+1}) := \bigoplus_{\lambda \in \mathcal{P}^+} C(V_{\lambda}).
\end{align}
We call $\O_q(\mathrm{SU}_{n+1})$ the {\em quantum coordinate algebra of $\mathrm{SU}_{n+1}$}. Moreover, we call the decomposition of $\OO_q(\mathrm{SU}_{n+1})$ given in \eqref{eqn:PeterWeyl} the \emph{Peter--Weyl decomposition} of $\OO_q(\mathrm{SU}_{n+1})$. We denote by $\langle -,-\rangle: U_q(\frak{sl}_{n+1}) \times \OO_q(\mathrm{SU}_{n+1}) \to \mathbb{C}$ the pairing given by evaluation.

When $V = V_{\lambda}$, for $\lambda \in \mathcal{P}^+$, and $f \in V_{\lambda}^*$, $v \in V_{\lambda}$, we find it convenient to $c^{\lambda}_{f,v}$ for the associated coordinate functional.  For any choice of weight basis $\{e_i\}_{i=1}^{n}$ of $V_{\varpi_n}$, the associated \emph{ matrix coefficients } are 
\begin{align} \label{eqn:matrixgens}
u_{ij} := c^{\varpi_n}_{e^i,e_j},  & & \textrm{ where } \{e^i\}_{i=1}^{n} \textrm{ is the dual basis of } V^*_{\varpi_n}. 
\end{align}
Note that $\OO_q(\mathrm{SU}_{n+1})$ is generated as an algebra by the elements $u_{ij}$. We fix, for once and for all, a weight basis satisfying
\begin{align} \label{eqn:dualpairing}
\langle E_i, u_{i+1,i} \rangle = 1, & & \langle F_i, u_{i,i+1} \rangle = 1, & & \langle K^{\pm}_j, u_{ii} \rangle = q^{\pm(\delta_{j+1,i} -\delta_{ij})},
\end{align}
and all other pairings of generators are zero. It can be shown that the algebra $\OO_q(\mathrm{SU}_{n+1})$ is generated by the elements $u_{ij}$ subject to the relations
\begin{flalign*}
  u_{ij}u_{i'j}  = q u_{i'j}u_{ij}, & ~~~~~~   u_{ij}u_{ij'} = qu_{ij'}u_{ij},   \\                 
   u_{ij'}u_{i'j} = u_{i'j}u_{ij'}, & ~~~~~~   u_{ij}u_{i'j'} = u_{i'j'}u_{ij} + (q-q^{-1})u_{ij'}u_{i'j},  
\end{flalign*}
for $i,i',j,j' = 1, \dots, n+1$, such that $i < i', j<j'$, along with the quantum determinant relation which we omit. The Hopf algebra structure on $\OO_q(\mathrm{SU}_{n+1})$ is determined by $\DEL(u_{ij}) :=  \sum_{a=1}^{n+1} u_{ia} \oby u_{aj}$ and $\e(u_{ij}) := \d_{ij}$.

\subsection{Lusztig's Root Vectors}

In this subsection we recall Lusztig's braid group action on $U_q(\mathfrak{sl}_{n+1})$, and the space of positive root vectors  associated to a choice of a reduced decomposition of the longest element of the Weyl group. See  \cite{LusztigLeabh}, or \cite[\textsection 6.2]{KSLeabh}, for a more detailed presentation. Throughout, we will use twisted commutators, that is to say, for any algebra $A$ and some constant $c \in \mathbb{C}$, we denote
$
[a,b]_{c} := ab - c ba,
$
for $a,b \in A$.

To every $i=1, \ldots, n$, there corresponds an algebra automorphism $T_i$ of $U_q(\mathfrak{sl}_{n+1})$ which acts on the generators, for $j = 1, \dots, n$, as
\begin{flalign*}
&  T_i(K_j) = K_j K_i^{-a_{ij}}, ~~~~~~~~~~~~~ T_i(E_i) = -F_iK_i, ~~~~~~~~~~~~~  T_i(F_i)= -K_i^{-1}E_i,  \\
&  T_{k\pm 1}(E_k) = - [E_{k\pm1},E_k]_{q^{-1}}, ~~~~~~~~~  T_{i}(E_k) = E_{k}, ~~  \textrm{ if } |i - k| > 1,\\
&   T_{k\pm 1}(F_k) = - [F_k, F_{k\pm1}]_{q}, ~~~~~~~~\,~~~~  T_{i}(F_k) = F_{k}, ~\,\,~ \textrm{ if } |i - k| > 1. 
\end{flalign*}
The mapping $s_i \rightarrow T_i$, for $i = 1, \dots, n$, determines a homomorphism of the braid group $B_{n+1}$ into the group of algebra automorphisms of $U_q(\mathfrak{sl}_{n+1})$. We associate, to any choice of reduced decomposition $w_0 = s_{i_1}\cdots s_{i_{d}}$ of the longest element of the Weyl group, the following  elements
\begin{align*}
E_{\beta_k} : = T_{i_1}T_{i_2} \cdots T_{i_{k-1}}(E_{i_k}), & & \textrm{ for } k = 1, \dots, d.
\end{align*}
We call any such element a \emph{root vector} of $U_q(\mathfrak{sl}_{n+1})$. The set of all root vectors is linearly independent and we have an associated quantum generalisation of the classical Poincar\'{e}--Birkoff--Witt theorem \cite[\textsection 6.2.3]{KSLeabh}. For a reduced decomposition $\mathbf{i}$, we denote the vector space spanned by the associated space of positive root vectors by $\mathfrak{n}^+_{\mathbf{i}}$.

\subsection{The Explicit Calculation}

In this subsection we give an explicit description of the root vectors associated to the longest word decomposition given in \textsection 3.1. We note first that it follows from the discussions in \cite[1.8(d)]{LusztigJAMS} that 
\begin{align*} \label{eqn:BraidActionFormulae}
T_{k,k\pm 1}(E_{k}) = - E_{k \pm 1}.
\end{align*}

We find it convenient to introduce some notation. First we write
\begin{align*}
E_{ji} := [\cdots [E_{j-1},E_{j-2}]_{q^{-1}}, E_{j-3}]_{q^{-1}}, \cdots, E_i]_{q^{-1}}.
\end{align*}
Then, for any $i \leq j$, we denote
\begin{align*}
T_{ji} := T_j \circ T_{j-1} \circ \cdots \circ T_i.
\end{align*}
The following more general formulae can now be derived from this lemma. 

\begin{lem} \label{lem:formulabis}
For any $(i,j) \in \Delta^+$, it holds that
\begin{enumerate}
\item $T_{n,i+1}(E_i) = (-1)^{n-i} E_{n+1,i}$,
\item $T_{n,i-1}(E_{ji}) = (-1)^{j-i+1} E_{j-1,i-1}$.
\end{enumerate}
\end{lem}
\begin{proof}
1. ~  Assume that, for some $n > b \geq i+1$, we have
\begin{align*}
T_{b,i+1}(E_i) = (-1)^{b-i} E_{b+1,i}.
\end{align*}
Then it follows that 
\begin{align*}
T_{b+1} \circ T_{b,i+1}(E_{i}) = & \, (-1)^{b-i}T_{b+1}(E_{b+1,i})\\
= & \, (-1)^{b-i}T_{b+1}\!\left([\cdots [E_b,E_{b-1}]_{q^{-1}},E_{b-2}]_{q^{-1}}, \cdots , E_{i}]_{q^{-1}}\right)\\
= & \,(-1)^{b-i} [\cdots [T_{b+1}(E_b),E_{b-1}]_{q^{-1}},E_{b-2}]_{q^{-1}}, \cdots , E_{i}]_{q^{-1}}\\
= & \, (-1)^{b-i+1}[\cdots [[E_{b+1},E_b]_{q^{-1}},E_{b-1}]_{q^{-1}},E_{b-2}]_{q^{-1}}, \cdots , E_{i}]_{q^{-1}}\\
= & \, (-1)^{(b+1)-i}E_{b+2,i}.
\end{align*}
Since the identity clearly holds for $T_{i+1}(E_{i}) = - E_{i+2,i}$, the general case now follows by an inductive argument.

2. ~ It follows from \eqref{eqn:BraidActionFormulae} that, for any $k = i, \dots, j-1$, we have 
\begin{align*}
T_{n,i-1}(E_{k}) = & \, T_{n,k-1} \circ T_{k-2} \cdots T_{i-1}(E_{ji})\\
= & \, T_{n,k+2} \circ T_{k} \circ T_{k-1}(E_{k}) \\
= & \, - T_{n,k+2}(E_{k-1})\\
= & \, -E_{k-1}.
\end{align*}
It now follows that
\begin{align*}
T_{n,i-1}(E_{ji}) = & \, T_{n,i-1}([\cdots [E_{j},E_{j-1}]_{q^{-1}},E_{j-2}]_{q^{-1}}, \cdots , E_{i}]_{q^{-1}})\\
                            = & \, ([\cdots [T_{n,i-1}(E_{j}),T_{n,i-1}(E_{j-1})]_{q^{-1}},T_{n,i-1}(E_{j-2})]_{q^{-1}}, \cdots , T_{n,i-1}(E_{i})]_{q^{-1}})\\
                             = & \,(-1)^{j-i+1} [\cdots [E_{j-1},E_{j-2}]_{q^{-1}},E_{j-3}]_{q^{-1}}, \cdots , E_{i-1}]_{q^{-1}})\\
                             = & \,(-1)^{j-i+1} E_{j-1,i-1},
\end{align*}
giving the claimed identity.
\end{proof}

\begin{prop}
For the reduced decomposition $\mathbf{j}$, the associated root vectors are given, up to sign, by the elements
$$
\left \{E_{ji} \,|\, \textrm{ for } \alpha_{ij} \in \Delta^+ \right \}\!.
$$
\end{prop}
\begin{proof}
Any root vector will, by definition, be of the form
\begin{align*}
T_{n1} \circ T_{n2} \circ \cdots \circ T_{n,k+1}(E_{k}), & & \textrm{ for some }  1  \leq k \leq  n.
\end{align*}
From the first identity in Lemma \ref{lem:formulabis}, we know that 
$$
T_{n,k+1}(E_{k}) = (-1)^{n-k} E_{n+1,k}.
$$
Moreover, the second identity in Lemma \ref{lem:formulabis} tells us that when we operate on $E_{ji}$ by the operator $T_{n,i-1}$, we get the element $E_{j-1,i-1}$ up to sign. Moreover, we see that each $E_{ji}$ can be realised in this way, giving the claimed presentation of the root vectors.
\end{proof}

Finally, we note that each root vector is a weight vector with respect to the adjoint action of $U_q(\mathfrak{sl}_{n+1})$ on itself. Explicitly, it holds that  
\begin{align*}
|E_{ji}| = \alpha_{j-1} + \alpha_{j-2} + \cdots + \alpha_i, & & \textrm{ for } \alpha_{ij} \in \Delta^+.
\end{align*}

\subsection{The Opposite Involution of $U_q(\mathfrak{sl}_{n+1})$} \label{appendix.opposite}

For the $A$-series Lie algebras, the element $-w_0$ permutes the simple roots of $\mathfrak{sl}_{n+1}$, and this permutation corresponds to the non-trivial symmetry of the associated Dynkin diagram. 
This in turn gives the Hopf algebra involution 
$$
\overline{w_0}: U_q(\mathfrak{sl}_{n+1}) \to U_q(\mathfrak{sl}_{n+1})
$$
defined on generators as
\begin{align*}
\overline{w_0}(E_i) := E_{i'}, & & \overline{w_0}(K_i) = K_{i'}, & & \overline{w_0}(F_i) = F_{i'}, 
\end{align*}
where of course $i' := n+1-i$. We call $\overline{w_0}$ the \emph{opposite involution} of $U_q(\mathfrak{sl}_{n+1})$. For more details on automorphisms of Drinfeld--Jimbo quantum groups, see \cite[\textsection 6.1.6]{KSLeabh}.

We also have a natural associated inner automorphism of the Weyl group
\begin{align*}
S_{n+1} \to S_{n+1}, & & s_i \mapsto s_{i'},
\end{align*}
and we call it the \emph{opposite involution} of $W$. This in turn gives a permutation of the set of commutation classes of reduced decompositions, which we denote by $\mathbf{i} \mapsto \mathbf{i}'$, for any commutation class $\mathbf{i}$. We note that different choices of reduced decomposition of the longest element of the Weyl group can give different sets of root vectors. However, two decompositions in the same commutation class will give the same set of root vectors.

\begin{prop} \label{prop:wzerontonprime}
Let $\mathbf{i}$ be a reduced decomposition of $w_0$. Then it holds that 
\begin{align} 
\overline{w_0}(\mathfrak{n}^+_{\mathbf{i}}) = \mathfrak{n}^+_{\mathbf{i}'}, 
\end{align}
for all reduced decompositions of the longest element of the Weyl group $S_{n+1}$. 
\end{prop} 
\begin{proof}
It follows from the  \eqref{eqn:BraidActionFormulae} that 
\begin{align*}
\overline{w_0} \circ T_i(E_k) = T_{i'} \circ \overline{w_0}(E_k), & & \textrm{ for all } i,k = 1, \dots, n.
\end{align*}
It now follows that 
\begin{align*}
\overline{w_0} \circ T_i(E_{k_1}E_{k_2} \cdots E_{k_l}) = & \,  \overline{w_0}(T_i(E_{k_1}))\overline{w_0}(T_i(E_{k_2})) \cdots \overline{w_0}(T_i(E_{k_l})) \\
= & \,  T_{i'}(\overline{w_0}(E_{k_1})) T_{i'}(\overline{w_0}(E_{k_2}) \cdots T_{i'}(\overline{w_0}(E_{k_l}) \\
= & \, T_{i'} \circ \overline{w_0}(E_{k_1}E_{k_2} \cdots E_{k_l}).
\end{align*}
Thus we see that, for any $X \in U_q(\mathfrak{n}_{\mathbf{i}}^+)$, we have that 
\[
\overline{w_0} \circ T_i(X) = T_{i'} \circ \overline{w_0}(X).
\]
From this we now see that 
\begin{align*}
\overline{w_0}(T_{i_1} \circ \cdots T_{{i_k}}(E_j)) =  & T_{i'_1} \circ \cdots \circ T_{i'_k} \circ \overline{w_0}(E_j)
=   T_{i'_1} \circ \cdots \circ T_{i'_k}(E_{j'}).
\end{align*}
Thus we see that a root vector corresponding to the $w_0$ substring $s_{i_1}s_{i_2} \cdots s_{i_k}$ is mapped by $\overline{w_0}$ to the root vector corresponding to the substring $s_{i'_1}s_{i'_2} \cdots s_{i'_k}$. This means that $\overline{w_0}$ maps the positive quantum root space $\mathfrak{n}^+_{\mathbf{i}}$ to the positive quantum root space $\mathfrak{n}^+_{\mathbf{i}'}$ as claimed.
\end{proof}

In other words, the above proposition says that the following diagram  commutes:
\begin{align*}
\xymatrix{ 
\textbf{i}  \ar@{|->}[rrrr]^{'} \ar@{|->}[d] & & & & \ar@{|->}[d]  \textbf{i}' \\       
\mathfrak{n}^+_{\mathbf{i}}  \ar@{|->}[rrrr]_{\overline{w_0}}  & & & & \mathfrak{n}^+_{\mathbf{i}'},
}
\end{align*}
where the vertical arrows are given by the function that associates to a reduced decomposition the space of the corresponding root vectors.

We finish with an explicit calculation of the automorphism of $\OO_q(\mathrm{SU}_{n+1})$ dual to $\overline{w_0}$. This is undoubtedly known to the experts, but in the absence of a suitable reference, we include a proof. 

\begin{prop}
The algebra map $\overline{w_0}^*:\OO_q(\mathrm{SU}_{n+1}) \to \OO_q(\mathrm{SU}_{n+1})$, determined by 
\begin{align*}
\overline{w_0}^*(u_{ij}) \mapsto -q S(u_{n+2-j, n+2-i}), & & \textrm{ for } i,j = 1, \dots, n + 1,
\end{align*}
satisfies the identity 
\begin{align} \label{eqn:dualopposite}
\langle \overline{w_0}(X),a \rangle = \langle X,\overline{w_0}^*(a)\rangle, & & \textrm{ for all } X \in U_q(\mathfrak{sl}_{n+1}), \, a \in \OO_q(\mathrm{SU}_{n+1}),
\end{align}
which is to say, it is \emph{dual} to $\overline{w_0}$ with respect to the bilinear pairing. 
\end{prop}
\begin{proof}
A direct check of the defining relations of $\OO_q(\mathrm{SU}_{n+1})$ confirms that $\overline{w_0}^*$ gives a well-defined Hopf algebra algebra automorphism of $\OO_q(\mathrm{SU}_{n+1})$. Now since we have a dual pairing of Hopf algebras, it is sufficient to check \eqref{eqn:dualopposite} for the generators of $U_q(\mathfrak{sl}_{n+1})$ and $\OO_q(\mathrm{SU}_{n+1})$. 
For the generator $E_i$, we see that  
\begin{align*}
\big \langle \overline{w_0}(E_i), u_{n-i+2,n-i+1}) \big\rangle =  \big \langle E_{i'}, u_{n-i+2,n-i+1}) \big \rangle = 1,
\end{align*}
while on the other hand
\begin{align*}
\big\langle E_i, \overline{w_0}^*(u_{n-i+2,n-i+1}) \big\rangle =  -q \big \langle E_i, S(u_{i+1,i}) \big \rangle 
=  -q \big \langle S(E_i), u_{i+1,i} \big \rangle 
= q \big \langle E_iK_i, u_{i+1,i} \big \rangle
=  1.
\end{align*}
For all other matrix generators, both sides of \eqref{eqn:dualopposite} clearly give zero, which is to say, the identity is trivially satisfied. The case of the generators $F_i$ and $K_i$ is similarly confirmed. Thus we see that \eqref{eqn:dualopposite} holds.
\end{proof}


\section{Some Remarks on Filtered Algebras} \label{app:filtration}

We first present a simple system for constructing a filtration  from a total order on a set of algebra generators. Our motivation is to provide a convenient framework in which to state the proof of Lemma \ref{lem:theassgradedlemma}, and we make no claims of originality.

Consider an algebra $A$, together with a set of algebra generators $G$. This gives a surjective algebra map $\proj_A: \mathcal{T}(V) \twoheadrightarrow A$, where $V$ is the linear span of the elements of $G$, and $\mathcal{T}(V)$ is its tensor algebra. A set of \emph{relations} for $A$, with respect to $G$, is a choice of subset $R \subseteq  \mathrm{ker}(\proj_A)$ that generates $\mathrm{ker}(\proj_A)$ as a two-sided ideal in $\mathcal{T}(V)$.

Let $X$ be a non-empty totally-ordered finite set with $n$ elements. Moreover, let $V$ be an $n$-dimensional vector space with a basis $\{e_x\}_{x \in X}$ indexed by the elements of $X$. We identify $\mathbb{Z}_{\geq 0}^n$ with the free abelian monoid generated by $X$, and we endow it with the degree lexicographical order $<$ induced by the total order of $X$. A $\mathbb{Z}^n_{\geq 0}$-algebra grading on $\mathcal{T}(V)$ is uniquely  determined by $|e_{x}| = x$. Consider a general element $f \in \mathcal{T}(V)$, written as a sum $f = \sum_{d=1}^m f_d$,  where each summand is homogeneous with respect to the $\mathbb{Z}^n_{\geq 0}$-grading, and $|f_i| < |f_m|$, for all $i = 1, \dots, m-1$. We call $f_m$ the \emph{leading term} of $f$ and denote
$
\mathrm{LT}(f) := f_m.
$

An algebra filtration of the tensor algebra 
$$
\mathscr{F}^{\mathcal{T}} = \bigcup_{\lambda \in\mathbb{Z}^n_{\geq 0}} \mathscr{F}_{\lambda}^{\mathcal{T}} \subseteq \mathcal{T}(V)
$$
can then be defined by 
\begin{align*}
\mathscr{F}_{\lambda}^{\mathcal{T}} := \mathrm{span}_{\mathbb{C}} \left \{e_{x_1} \otimes \cdots \otimes e_{x_k} | \sum_{i=1}^k x_i \leq \lambda \right\}.
\end{align*}
For any two-sided ideal $J \subseteq \mathcal{T}(V)$,  we have the induced filtration $\mathscr{F}$ on the quotient $A := \mathcal{T}(V)/J$. Explicitly
\begin{align*}
A =  \bigcup_{\lambda \in\mathbb{Z}^n_{\geq 0}} \proj_A(\mathscr{F}_{\lambda}^{\mathcal{T}})  =:  \bigcup_{\lambda \in\mathbb{Z}^n_{\geq 0}} \mathscr{F}_{\lambda}.
\end{align*}
Consider the associated graded algebra
\begin{align*}
\mathrm{gr}^{\mathscr{F}} = \bigoplus_{\lambda \in\mathbb{Z}^n_{\geq 0}} \mathrm{gr}^{\mathscr{F}}_{\lambda} := \bigoplus_{\lambda \in\mathbb{Z}^n_{\geq 0}} \mathscr{F}_{\lambda}/\mathscr{F}_{ \lambda - 1}.
\end{align*}
For any $a \in \mathscr{F}_{\lambda}$, we find it convenient to denote the coset of $a$ in $\mathrm{gr}^{\mathscr{F}}_{\lambda}$ by $[a]_{\lambda}$. By construction of the filtration, the elements $[e_x]_x$ generate $\mathrm{gr}^{\mathscr{F}}$. We denote by
$$
\proj_{\mathscr{F}}: \mathcal{T}(V) \to \mathrm{gr}^{\mathscr{F}}
$$
the associated  surjective algebra map. Moreover, for $j \in J$, we see that 
\begin{align*}
\mathrm{LT}(j) \in \mathrm{ker}(\proj_{\mathscr{F}}).
\end{align*}

We finish with some remarks on Bongale's theorem \cite[Theorem 2]{Bongo} for filtered Frobenius algebras (see \textsection \ref{section:Frob} for the definition of Frobenius algebras). This result was originally stated for $\mathbb{Z}$-filtrations, but as observed in \cite[Lemma 4.10]{MTSUK}, it holds more generally. We state it for the free abelian monoid $\mathbb{Z}^n_{\geq 0}$ endowed with some total  order $\geq$. Consider  a finite-dimensional  algebra $A$, endowed with a filtration 
$$
A = \bigcup_{\mu \in \mathbb{Z}^n_{\geq 0}} \mathscr{F}_{\mu}.
$$
Let us assume that the associated graded algebra $\mathrm{gr}^{\mathscr{F}}$ satisfies
\begin{align*}
\mathrm{dim}\Big(\mathrm{gr}^{\mathscr{F}}_0\Big) = \mathrm{dim}\Big(\mathrm{gr}^{\mathscr{F}}_{\lambda}\Big) = 1, & & \mathrm{gr}^{\mathscr{F}}_{\mu} = 0, \textrm{ ~~~ for all } \mu > \lambda,
\end{align*}
for some $\lambda \in \mathbb{Z}^n_{\geq 0}$. \emph{Bongale's theorem} states that if $\mathrm{gr}^{\mathscr{F}}$ is a Frobenius algebra, then a Frobenius structure for $A$ is given by
\begin{align*}
B: A \otimes A \to \mathbb{C},  &  & (a,b) \mapsto  \iota\big([ab]_{\lambda}\big),
\end{align*}
where $\iota: \mathrm{gr}^{\mathscr{F}}_{\lambda} \to \mathbb{C}$ is any choice of linear isomorphism. Moreover, all Frobenius structures are of this form.

\end{document}